\newif\ifarxiv\arxivtrue
\newif\ifec\ecfalse
\newif\ifanonymous\anonymousfalse
\newif\ifdraft\draftfalse

\documentclass[a4paper,11pt]{scrartcl}

\ifarxiv
	\usepackage{authblk}
\fi

\usepackage[british]{babel}

\usepackage{tikz}
\usetikzlibrary{calc,patterns}

\usepackage{pgfplots}
\pgfplotsset{compat=1.18}
\usepgfplotslibrary{fillbetween}

\tikzstyle{vertex} = [shape=circle,draw=black]
\tikzstyle{namedVertex} = [shape=circle,draw=black]
\tikzstyle{edge} = [draw,->,thick]
\tikzstyle{labeledNodeS}=[circle, color=black!75!white, draw, inner sep = 0.1em, minimum size = 1.5em, scale=1.25]
\tikzstyle{normalEdge}=[very thick, >=stealth]

\ifarxiv
	\usepackage[utf8]{inputenc}
	\usepackage[T1]{fontenc}
\fi
\usepackage{framed}

\usepackage{adjustbox}				

\usepackage{mathtools}

\ifarxiv
	\usepackage{amsthm}		
\fi
\usepackage{amssymb}				
\usepackage{mathrsfs,stmaryrd}

\usepackage{aligned-overset} 	

\ifarxiv
	\usepackage[bookmarks=false,colorlinks=true, linkcolor=blue, urlcolor=blue, citecolor=blue, breaklinks=true]{hyperref}
\else
	\usepackage[colorlinks=false, breaklinks=true]{hyperref}
\fi
\usepackage{cleveref}			

\crefname{cons}{constraint}{constraints}
\Crefname{cons}{Constraint}{Constraints}
\creflabelformat{cons}{(#2#1#3)}
\crefname{claim}{claim}{claims}
\Crefname{claim}{Claim}{Claims}
\crefname{assumption}{assumption}{assumptions}
\Crefname{assumption}{Assumption}{Assumptions}

\ifarxiv
\usepackage[style=alphabetic,maxbibnames=99,maxalphanames=4,maxcitenames=99]{biblatex}
\newcommand{\citefull}[2][]{\citeauthor{#2}~\cite[#1]{#2}}
\newcommand{\citefulls}[2][]{\citeauthor*{#2}~\cite[#1]{#2}}
\fi

\ifec
	\newenvironment{proofClaim}[1][]{\ifthenelse{\equal{#1}{}}{\begin{proof}}{\begin{proof}[#1]}}{\end{proof}}
	\AtEndPreamble{%
		\theoremstyle{acmdefinition}
		\newtheorem{remark}[theorem]{Remark}
		\newtheorem{obs}[theorem]{Observation}
		\newtheorem{claim}[theorem]{Claim}
	}
\else	
	\ifarxiv
		\theoremstyle{definition}
		\newtheorem{definition}{Definition}[section]
		\newtheorem{example}[definition]{Example}
		
		\theoremstyle{plain}
		
		\newtheorem{proposition}[definition]{Proposition}
		
		\newtheorem{corollary}[definition]{Corollary}
		\newtheorem{lemma}[definition]{Lemma}
		\newtheorem{theorem}[definition]{Theorem}

		\newtheorem{claim}{Claim}
		\newenvironment{proofClaim}[1][]{\ifthenelse{\equal{#1}{}}{\begin{proof}}{\begin{proof}[#1]}}{\end{proof}}
		
		\theoremstyle{remark}
		\newtheorem{remark}[definition]{Remark}
		\newtheorem{obs}[definition]{Observation}

		\newenvironment{proofNormal}[1][Proof]{\begin{proof}[#1]}{\end{proof}}
		
		\usepackage{enumitem}
		\usepackage{bm} 	
		\newcommand{\caseitem}[1]{\def\casedescr{#1}%
			\item}
		\newlist{proofbycases}{enumerate}{1}
		\setlist[proofbycases]{
			leftmargin=3em,
			labelwidth=1.5em,
			label=\boldmath\bfseries\sffamily\arabic*. Case: \protect\casedescr:,
			ref=\arabic*,
			align=left
		}
	\else
		\newtheorem{obs}{Observation}
		\newcommand{\qedhere}{\relax}
		\newenvironment{proofClaim}[1][]{\ifthenelse{\equal{#1}{}}{\proof{Proof.}}{\proof{#1.}}}{\Halmos\endproof}
		\newenvironment{proofNormal}[1][]{\ifthenelse{\equal{#1}{}}{\proof{Proof.}}{\proof{#1.}}}{\Halmos\endproof}
	\fi
\fi

\usepackage{array}					

\usepackage{braket}					
\ifec\else 
	\ifarxiv\else

	\fi
	\usepackage{bm}						
\fi
\usepackage{nicefrac} 				

\allowdisplaybreaks					

\newcommand{\IN}{\mathbb{N}}	
\newcommand{\INs}{\IN^\ast}		
\newcommand{\IR}{\mathbb{R}}
\newcommand{\R}{\mathbb{R}}	

\newcommand{\Pc}{\mathcal{P}}

\newcommand{\abs}[1]{\left|#1\right|}

\newcommand{\noqed}{\let\qed\relax}

\newcommand{\cf}{cf.\ }
\newcommand{\fa}{f.a.\ }
\newcommand{\eg}{e.g.\ }

\newcommand{\ie}{i.e.\ }
\newcommand{\Ie}{I.e.\ }
\newcommand{\wrt}{w.r.t.\ }
\newcommand{\wlofg}{wlog.\ }

\usepackage{rotating}

\usepackage{longtable}

\ifarxiv
	\usepackage[font={small,it}]{caption} 
\fi

\usepackage{transparent}
\usepackage{color}
\usepackage{graphicx}

\usepackage[all]{xy}
\ifarxiv
	\usepackage{lmodern}
	\usepackage[babel]{csquotes}
\fi

\usepackage{framed,enumitem}

\newcommand{\termTime}[1][]{%
	\ifthenelse{\equal{#1}{}}
	{\Theta}
	{\Theta_{\mathrm{#1}}}
}

\usepackage{dsfont} 
\NewDocumentCommand{\CharF}{O{}}{%
	\ifthenelse{\equal{#1}{}}%
	{\mathds{1}}%
	{\mathds{1}_{#1}}%
}

\usepackage{siunitx} 
\newcommand{\norm}[1]{\left\Vert#1\right\Vert}

\usepackage{array}
\newcolumntype{L}[1]{>{\raggedright\arraybackslash}m{#1}}
\newcolumntype{C}[1]{>{\centering\arraybackslash}m{#1}}
\newcolumntype{R}[1]{>{\raggedleft\arraybackslash}m{#1}}

\newcommand*\diff{\mathop{}\!\mathrm{d}} 
\DeclarePairedDelimiterX{\scalar}[2]{\langle}{\rangle}{#1, #2} 
\newcommand{\VI}{\textrm{VI}}
\newcommand{\QVI}{\textrm{QVI}} 
\newcommand{\A}{\mathcal{A}} 
\usepackage{dsfont} 

\newcommand{\effWalkDelay}{effective walk delay}
\newcommand{\addmEpsDev}{admissible \epsDev} 
\newcommand{\epsDev}{$\gamma$-deviation}
\newcommand{\addmDev}{admissible deviation}
\newcommand{\setS}{constraint set}

\newcommand{\SCDE}[1][]{\ifthenelse{\equal{#1}{}}{SCDE}{\ifthenelse{\equal{#1}{fulls}}{side-constrained dynamic equilibria}{side-constrained dynamic equilibrium}}}
\newcommand{\globalSCDE}[1][]{\ifthenelse{\equal{#1}{}}{strict SCDE}{\ifthenelse{\equal{#1}{fulls}}{strict side-constrained dynamic equilibria}{strict side-constrained dynamic equilibrium}}}
\newcommand{\globalEL}[1][]{\ifthenelse{\equal{#1}{}}{strict CDE}{\ifthenelse{\equal{#1}{fulls}}{strict capacitated dynamic equilibria}{strict capacitated dynamic equilibrium}}}
\newcommand{\sCDEdf}[1][]{\ifthenelse{\equal{#1}{}}{strong BSDE}{strong dynamic Bernstein-Smith equilibrium}}
\newcommand{\wCDEdf}[1][]{\ifthenelse{\equal{#1}{}}{weak BSDE}{weak dynamic Bernstein-Smith equilibrium}}
\newcommand{\sCDEu}[1][]{\ifthenelse{\equal{#1}{}}{strong LPDE}{strong dynamic Larsson-Patriksson equilibrium}}
\newcommand{\wCDEu}[1][]{\ifthenelse{\equal{#1}{}}{weak LPDE}{weak dynamic Larsson-Patriksson equilibrium}}
\newcommand{\sCDEuP}[1][]{\ifthenelse{\equal{#1}{}}{strong MNSDE}{strong dynamic Marcotte-Nguyen-Schoeb equilibrium}}
\newcommand{\wCDEuP}[1][]{\ifthenelse{\equal{#1}{}}{weak MNSDE}{weak dynamic Marcotte-Nguyen-Schoeb equilibrium}}
\newcommand{\sCDEuE}[1][]{\ifthenelse{\equal{#1}{}}{strong ???}{strong ???}}

\newcommand{\HLP}[1][]{\ifthenelse{\equal{#1}{}}{HLP}{Hearn-Larsson-Patriksson equilibrium}}

\newcommand{\Sm}[1][]{\ifthenelse{\equal{#1}{}}{S}{Smith equilibrium}}
\newcommand{\Hey}[1][]{\ifthenelse{\equal{#1}{}}{BS}{Heydecker equilibrium}}
\newcommand{\BS}[1][]{\ifthenelse{\equal{#1}{}}{BS}{Bernstein-Smith equilibrium}}

\DeclareMathOperator{\supp}{supp}

\DeclareMathOperator{\esssup}{ess\,sup}
\DeclareMathOperator{\essinf}{ess\,inf}

\newcommand{\tStart}{t_0}
\newcommand{\tEnd}{t_f}
\newcommand{\planningInterval}{[\tStart,\tEnd]}

\newcommand{\truncated}[2]{#1\big\vert^{#2}}

\newcommand{\shiftN}{\gamma}

\newcommand{\emptyarg}{\,\cdot\,}

\newcommand{\flowVolume}[1][]{%
	\ifthenelse{\equal{#1}{}}
	{x}
	{x_{#1}}
}

\crefname{asmpt}{assumption}{assumptions}
\Crefname{asmpt}{Assumption}{Assumptions}

\ifdraft
	\usepackage[left=1cm, right=5cm, top=2cm, bottom=4cm]{geometry}
	\usepackage[ngerman, textsize=small, textwidth=4.5cm, color=lightgray]{todonotes}
	\definecolor{darkgreen}{rgb}{0.2,0.8,0.55}
	\definecolor{aliceblue}{rgb}{0.8, 0.9, 1.0}
	\newcommand{\tobias}[1]{\todo[color=red,size=small]{Tobias: #1}}
	\newcommand{\lukas}[1]{\todo[color=aliceblue,size=small]{Lukas: #1}}
	\newcommand{\lukasI}[1]{\todo[color=aliceblue,inline]{Lukas: #1}}
\else 
	\usepackage[left=3cm, right=3cm, top=2cm, bottom=4cm]{geometry}
	\newcommand{\tobias}[1]{\relax}
	\newcommand{\lukas}[1]{\relax}
	\newcommand{\lukasI}[1]{\relax}
\fi

\title{Side-Constrained Dynamic Traffic Equilibria}
\author{Lukas Graf and Tobias Harks}

\date{}

\newcommand{\abstracttext}{%
	We study  dynamic traffic assignment with side-constraints.
	We first give a counterexample to a key result from the literature
	regarding the existence of dynamic equilibria
	for volume-constrained traffic models in the classical linear edge-delay model. Our counterexample shows
	that the feasible flow space need not be convex and it further reveals
	that classical infinite dimensional variational inequalities are not suited for the definition of general side-constrained dynamic equilibria. We propose a new framework for side-constrained dynamic  equilibria
	based on the concept of \addmEpsDev s of flow particles in space and time. We then show under which assumptions the resulting equilibria can still be characterized by means of quasi-variational and variational inequalities, respectively.
	Finally, we establish first existence results for side-constrained dynamic equilibria for the non-convex setting of volume-constraints.%
}
\newcommand{\acknowledgetext}{%
	We thank Patrice Marcotte and Renxin Zhong as well as the anonymous reviewers for their helpful comments on an earlier version of this paper. We also thank Julian Schwarz, Michael Markl, David Watling and Michael Smith for many interesting discussions on the topics of this paper.%
}

\begin{document}
	
\thispagestyle{empty}
\maketitle


\begin{abstract}
	\abstracttext
\end{abstract}

\newpage
\tableofcontents
\newpage


\section{Introduction}
Traffic assignment problems have been
successfully applied in the past decades in order
to model, predict and optimize traffic distributions.
While in most models, the  network infrastructure
is equipped with capacities, it is usually assumed that
the excess of capacity is possible and leads to congestion, that is,
increased travel times, e.g., by increased waiting times in queues. 
In several realistic scenarios, however, 
there are also hard capacity constraints that must not be violated
by any feasible traffic flow. For instance, hard traffic volume restrictions
are imposed by local authorities in order to keep the exhaust gas emissions within urban residential areas below certain threshold values, see \citefulls{Grote2016}.
Following \citefulls{zhong11}, another example includes
tunnels, in which the number of vehicles inside the tunnel is limited to maintain sufficient reserve capacity/space for handling any possible incident (e.g., car accidents or disruptions due to disaster).

From a theoretical and computational perspective, the traffic assignment problem
 with hard side-constraints has been studied extensively for \emph{static flows} using methods
from convex optimization, see~\cite{Hearn98solving, Larsson99,Larsson95,Marcotte04}
and references therein. These works mostly considered edge-capacity constraints and
studied the optimization problem minimizing the Beckmann-McGuire-Winsten potential
subject to these constraints.
 The dual variables associated with the capacity constraints
are used as additional prices or  queueing delays and the optimal solutions
are interpreted as unconstrained Wardrop equilibria \wrt \emph{generalized travel costs} consisting of  the actual
delay plus the dual prices along a path.
This way, using the convexity of both the feasible space and the objective function, such special capacitated equilibria can be completely characterized as solutions
to associated variational inequalities. 
Most works cited above, however,  used the solutions of the convex optimization
problem or, equivalently, solutions to the associated variational inequalities \emph{as the definition} of a capacitated user (or Wardrop) equilibrium. Only a few works noted a conceptual gap between 
introducing a \emph{behavioural equilibrium concept} in the sense of defining an associated noncooperative game versus using solutions to a variational inequality formulation
as its definition -- see the discussions in \citeauthor*{CorreaCapEqInStaticFlows}~\cite[968]{CorreaCapEqInStaticFlows}, Bernstein and Smith~\cite{BernsteinS94} and \citefulls{Marcotte04}.

In this article, we revisit the \emph{dynamic traffic assignment problem with side-constraints}.
Based on the  path-delay operator model in the $L^2$-function space introduced in seminal works by \citeauthor*{FrieszLTW89}~\cite{FrieszLTW89,Han2013} (see also Friesz and Han~\cite{Friesz19}
for a recent overview),
we consider  general side-constraints
ranging from edge-volume constraints and path inflow constraints to abstract constraint sets.
While the dynamic traffic assignment problem with side-constraints is far less explored compared to the
static variant, there are a few works studying fundamental questions related
to the existence, structure and computability of constrained dynamic equilibria.
One of the central works in this area is the paper by Zhong, Sumalee, Friesz and Lam~\cite{zhong11},
who were the first to consider side-constraints within the general  path-delay operator model.
 They assumed a fixed flow volume
and flexible departure time choice and instantiated the network loading
using (linear) volume-delay functions. The side-constraints were given by arc-volume constraints. They defined a side-constrained dynamic user equilibrium
via solutions of an associated infinite-dimensional variational inequality (VI) which needs to be solved over
the space of capacity-feasible dynamic flows. As one of their main results (Proposition~3.1, p.~1040), they claimed the
existence of side-constrained dynamic user equilibria arguing that the respective  VI
always admits a solution. The proof of this claim uses that the capacity-feasible dynamic flow space is bounded, closed and convex. Only these properties would allow the invocation
of general existence results for VI's in appropriate function spaces by~Browder~\cite{BROWDER1968}.
 
 \subsection{Our Contribution}
 We study  dynamic traffic assignment problems using the general path-delay-operator
 form as proposed by \citefull{FrieszLTW89} and augment this model
 with side-constraints. In fact, we will consider a more general general \emph{walk}-delay-operator model, because some interesting side-constraints such as energy-constraints for electric vehicles (see \citefull{GHP22}) require cyclic routes. Our contribution consists of four types of results.

 \begin{enumerate}
\item We first show that the claim of existence of side-constrained dynamic user equilibria 
 defined as solutions to a VI (\citefull[Proposition~3.1, p.~1040]{zhong11}) is wrong -- we give a nontrivial counterexample to this claim.
 The underlying reason lies in the fact
 that the side-constrained dynamic flow space need not be convex in general.  
 The consequences of the counterexample are somewhat severe since not only does the
 assumed existence result break down but, perhaps more seriously, the counterexample
 reveals that the proposed VI is in fact not a suitable definition of a side-constrained dynamic equilibrium.
\item We introduce -- in line with prior works for the static flow model (e.g.,  \citefull[968]{CorreaCapEqInStaticFlows}, \citefull{BernsteinS94} and \citefull{Marcotte04}) -- a \emph{behavioral equilibrium concept} via formally introducing a noncooperative game modeling the
 space of feasible deviations of users given a dynamic flow. Roughly speaking, a dynamic flow is an equilibrium, if there is no arbitrarily small bundle of users that can switch their strategy in space and time and strictly reduce their 
 travel cost.
 The precise way a feasible deviation is defined leads to a whole set of equilibrium concepts and we propose several of them including dynamic extensions of  Larsson-Patriksson (LP), Bernstein-Smith (BS)
 and Marcotte-Nguyen-Schoeb (MNS) equilibria, respectively,  which were originally proposed for static equilibrium flows. 
 \item Given an equilibrium concept for side-constrained dynamic equilibrium flows,
 obvious questions related to their characterization, existence and computability arise. 
 Under natural assumptions on the structure of the side-constraints and the set of feasible deviations, we give necessary and sufficient conditions under which an equilibrium can be described as a solution to an associated quasi-variational inequality. Additionally, we give more restrictive assumptions under which they can even be characterized by a variational inequality.
 We further show that equilibrium solutions exist, 
 if the set of side-constrained dynamic flows is convex and the walk-delay operator 
 is sequentially weak-strong continuous. 
 \item As the counterexample to \citefulls{zhong11} suggests, the model with edge-volume constraints
 can lead to non-convex flow spaces which means that standard existence tools from
the infinite dimensional VI theory cannot be used.
 For modelling volume constraints, we first describe a network loading model and then introduce 
 abstract \emph{edge-load functions} which include flow volumes as a special case.
 We show existence of dynamic LP
 and MNS equilibria
 under mild continuity assumptions on the edge-load functions. Our existence proof is
 in some sense constructive as it uses an augmented Lagrangian function approach
 (see~\cite{Larsson95} for such an approach for static flows) for violated edge-load constraints
 and invokes, in a black-box fashion, solutions to the relaxed equilibrium problem.
 We further show, however, that this  augmented Lagrangian approach fails for other, stricter equilibrium concepts such as the dynamic BS equilibrium: we give 
 an example in which the flows for the unconstrained  model with penalties do not converge to a (strong) dynamic BS equilibrium.
 \end{enumerate}
   Finally, it is worth mentioning that our model and the subsequent characterization and existence results
require only mild continuity properties of the walk-delay operator and the edge load functions, respectively,  and
thus  apply for
several  realistic and well-studied network loading models including the Vickrey queueing model with point queues~\cite{CominettiCL15,CominettiCO17,Koch11,OlverSK21,Vickrey94}, with spillback~\cite{Sering2019} with departure-time choice~\cite{FrascariaO20,Han2013}, the  Lighthill-Whitham-Richards (LWR) model~\cite{HanPi16} and the classical link-delay model
 of \citefulls{Friesz93}.

 \subsection{Related Work}

Two of the earliest papers in the field of dynamic traffic assignment 
are papers by \citeauthor*{Friesz93}~\cite{Friesz93,FrieszLTW89} who introduced the formalism
of a path-delay operator and investigated variational inequality and optimal control formulations under
specific network loading models, see also Boyce, Ran and LeBlanc~\cite{BoyceRL95,RanBL93}.
For an overview of further relevant works, we refer to the survey article by Friesz and Han~\cite{Friesz19}. A key development in this field are the identification of certain continuity
conditions of the path-delay operator in order to establish equilibrium existence.
This has been successfully shown for various network loading models ranging
from the link-delay model~\cite{ZhuM00} and the Vickrey model with point queues~\cite{CominettiCL15,GrafHS20,Han2013a,Koch11,MeunierW10} to the  Lighthill-Whitham-Richards (LWR) model~\cite{FrieszHanPedro13,HanPi16}.
 
It is worth noting that dynamic traffic models with spillback (cf.~\cite{HanPi16,Sering2019}), can be interpreted as an alternative way of handling hard capacities limiting the flow volume on particular edges. However, there is an important conceptual difference to side-constrained dynamic traffic assignment models: 
Spillback models allow the injection of flow into arbitrary paths and whenever the capacity restriction on an edge is reached, additional flow is prohibited from entering this edge and has to wait on the previous edge instead (causing additional congestion there). Thus, the spillback effect influences the agents behaviour only indirectly via the increased path-delays. 
In a model with hard side-constraints, on the other hand, the volume constraints directly influence the agents' behaviour by restricting their strategy space: \Ie if entering a certain path at a specific time would lead to a violation of any edge capacity on that path, such a flow is considered to be infeasible.
Our model can be seen as a strict generalization of both ideas as we do allow spillback models for the network loading
but in addition we can model hard side-constraints.

For dynamic traffic assignment with hard side-constraints not much is known.
\citefulls{zhong11} considered the path-operator model with a linear volume-delay formulation
for the network loading. They defined side-constrained dynamic equilibria as solutions to an associated
infinite dimensional variational inequality and claimed existence of such equilibria.
\citefulls{Hoang19} transferred the static BMW equilibrium concept to
a dynamic model by discretizing time and then considering a time-expanded network.
In a similar way \citefulls{Hamdouch2004} extended the static equilibrium concept from~\cite{Marcotte04} to a dynamic setting.

 \subsection{Paper Organization}
 We start the paper by recapping in \Cref{sec:static} the theory of side-constrained static equilibrium flows. Already for static models, the issue about properly defining side-constrained equilibria arises and we try to sketch the historic development of the key concepts in the field.
 In \Cref{sec:dynamic}, we introduce the basic dynamic traffic assignment model.
 In \Cref{sec:counter}, we give a counterexample to \citefull[Proposition~3.1, p.~1040]{zhong11} which illustrates the need of rethinking an appropriate
 solution concept for traffic assignment models with side-constraints.
 
 In \Cref{sec:framework}, we introduce our abstract framework of side-constrained
 dynamic traffic equilibria. In \Cref{sec:characterization}, we turn to characterization
 results of  such equilibria  in terms of variational or quasi-variational inequalities.
 Finally, in \Cref{sec:SCviaNL}, we derive  two equilibrium existence results for a class of 
 non-convex volume-constrained traffic models using an augmented Lagrangian penalty function approach.


\section{The Static Model}\label{sec:static}
We are given a directed graph $G=(V,E)$
and a set of populations or commodities $I:= \{1, \dots,
n\}$, where each commodity $i \in I$ has a demand $d_i> 0$ that has
to be routed from a source $s_i \in V$ to a destination $t_i \in V$.
The demand interval $[0,d_i]$ represents a continuum of infinitesimally small agents each acting independently
choosing a cost minimal $s_i$,$t_i$-path. 
There are continuous and nondecreasing cost  functions $\ell_{e}: \R^E \rightarrow \R_{\geq 0}, e\in E$.
A \emph{path flow} for commodity~$i\in I$ is a nonnegative vector
$x_i \in \R^{|\Pc_i|}_{\geq 0}$ that  lives in the path flow polytope: 
\begin{align*}
X_i=\left\{x_i\in \R_{\geq 0}^{|\Pc_i|}\;\middle\vert \; \sum_{p\in\Pc_i} x_{i,p}=d_i \right\},
\end{align*}
where $\Pc_i$ denotes the set of simple $s_i,t_i$-paths in $G$ and
$\Pc \coloneqq \bigcup_{i \in I}\Pc_i$.
We assume that every $t_i$ is reachable in $G$ from $s_i$
for all $i\in I$, thus, $X_i\neq \emptyset$ for all $i\in I$.
Given a  path flow vector $x\in X:=\times_{i\in I}X_i$, the cost of a path $p\in\Pc_i$,
 is defined as
$\ell_{p}(x):=\sum_{e\in p}\ell_{e}(x),$
where $ x_e:=\sum_{i\in I}\sum_{p\in\Pc_i : e\in p} x_{i,p}$ is the aggregated load of edge $e\in E$.
\begin{definition}
A path flow $ x^*\in X$ is a \emph{Wardrop equilibrium} if for all $i\in I$:
\[ \ell_{p}( x^*)\leq \ell_{q}( x^*) \text{ for all }p,q\in \Pc_i \text{ with }x^*_{i,p}>0.\]
\end{definition}
The interpretation here is that all agents are travelling along cost minimal paths
given the overall load vector $(x^*_e)_{e\in E}$.
One can characterize Wardrop equilibria by means of variational inequalities
as follows (cf.~Patriksson~\cite[Sec. 3.2.1]{Patriksson1994tap}):
\begin{lemma}\label{lem:vi-static}
The following statements are equivalent:
\begin{enumerate}
\item $x^*\in X$ 
is a Wardrop equilibrium.
\item $\scalar{\ell( x^*)}{ x^*- y} \leq 0  \text{ for all }y\in X, \text{ with } \ell( x^*) \coloneqq (\ell_{p}( x^*))_{p\in\Pc}$ and $\scalar{\emptyarg}{\emptyarg}$ denoting the scalar product on $\IR^{\abs{\Pc}}$.
\end{enumerate}
\end{lemma}

For the case of  separable latency functions $\ell$,  Dafermos and Sparrow~\cite{DafS69}  related Wardrop equilibria to Nash equilibria of an associated noncooperative game.
 Formally, for $\varepsilon>0$ and $p,q\in \Pc_i$ with $x_p>0$, let
\[ x_w(\varepsilon,p,q):=\begin{cases}x_w-\varepsilon, & \text{ if }w=p\\
x_w+\varepsilon, & \text{ if }w=q\\
x_w, & \text{ else. }\end{cases}\]

\begin{definition}\label{def:dafermos-sparrow}
A  path flow $x^*\in Z$ is a \emph{Nash equilibrium}, if for all $ p,q\in \Pc_i, x^*_{i,p}>0,\varepsilon\in (0,x^*_{i,p}], i\in I$, we have
\[ \ell_{p}( x^*)\leq \ell_{q}(x^*(\varepsilon,p,q)). \]
\end{definition}
Dafermos and Sparrow~\cite{DafS69} showed that for continuous and separable latency functions, Nash equilibria and
Wardrop equilibria coincide.

If $\ell$ is  separable and non-decreasing,  it is the gradient of the Beckmann-McGuire-Winsten potential function  and one can further characterize Wardrop equilibria
as optimal solutions to a convex optimization problem:
\begin{lemma}[cf.~Dafermos~\cite{dafermos1980traffic}]\label{lem:dafermos}
The following statements are equivalent:
\begin{enumerate}
\item $x^*\in X$ 
is a Wardrop equilibrium.
\item $x^*\in\arg\min_{x\in X} \Big\{  \sum_{e\in E}\int_{0}^{ x_e} \ell_{e}(z) \diff z  \Big\} $.
\item $\scalar{\ell( x^*)}{ x^*- y} \leq 0  \text{ for all }y\in X$.
\end{enumerate}
\end{lemma}

\subsection{Side-Constrained Traffic Equilibria}
Suppose we have hard edge capacities $c_e\geq 0, e\in E$
which need to be satisfied for a flow $x\in X$ to be \emph{capacity-feasible}, that is,
$ x_e\leq c_e, e\in E$. Let $Z:=\{x \in X \vert x_e \leq c_e, e\in E\}$ denote the
set of capacity-feasible path flows.
Following Patriksson~\cite[73]{Patriksson1994tap} and Larsson and Patriksson~\cite{Larsson95} (which we abbreviate henceforth with LP),
a side-constrained equilibrium can be defined
via the notion of \emph{saturated} and \emph{unsaturated} paths.
Given a path flow $x\in Z$, a path $p$ is saturated, if in contains an edge $e\in p$ with $ x_e= c_e$ and, conversely, a path $p$ is unsaturated,
if $ x_e<c_e$ for all $e\in p$.

\begin{definition}\label{def:weakWE}
A path flow $x^*\in Z$  is a side-constrained \emph{LP-equilibrium}, if
\[ \ell_{p}( x^*)\leq \ell_{q}( x^*) \text{ for all }p,q\in \Pc_i \text{ with } x^*_{i,p}>0 \text{ and } q \text{ unsaturated}.\] 
\end{definition}

For our subsequent discussion it is worth reformulating
the definition of an LP-equi\-li\-bri\-um in terms of feasible \emph{additive} $\varepsilon$-deviations
in the spirit of Dafermos and Sparrow~\cite{DafS69}.
For $\varepsilon>0$ and $q\in \Pc_i$, let
\[ x_w(\varepsilon,q):=\begin{cases}x_w+\varepsilon, & \text{ if }w=q\\
x_w, & \text{ else. }\end{cases}\]
Define  \[ \tilde{\ell}_e(x):=\begin{cases} \ell_e(x), & \text{ if }x_e\leq c_e\\
 +\infty, &\text{else.}\end{cases}\]
We obtain the following equivalent definition:
\begin{lemma} \label{lem:lp}
A  path flow $x^*\in Z$ is a side-constrained LP-equilibrium iff for all $p,q\in \Pc_i, x^*_{i,p}>0, i\in I$, we have
\[ \tilde\ell_{p}( x^*)\leq \tilde\ell_{q}(x^{*}(\varepsilon,q)) \text{ for all }\varepsilon\in (0,x^*_{i,p}]. \]
\end{lemma}

As observed by \citefull{Marcotte04}, this definition has the drawback of admitting
 rather artificial equilibrium flows. Consider for instance
a graph with one edge followed by two edges in parallel.
If the capacity of the first edge is saturated, then \emph{any}
feasible flow is an LP equilibrium no matter how the flow
is distributed on the two subsequent edges. This leads to unrealistic
equilibria in case one of the two edges is more expensive but carries flow
and the other (cheaper) one has free capacity.
An alternative equilibrium concept proposed by Smith~\cite{Smith84} 
avoids this problem by allowing for path changes of $\varepsilon>0$
units of flow provided that \emph{after the change}, the resulting flow is still
feasible (this corresponds to considering deviations of the form $x^*(\varepsilon,p,q)$
in Lemma~\ref{lem:lp}).
This concept has the drawback that it allows for coordinated
deviations of bundles of users, which is unrealistic and, as shown by Smith,
leads to non-existence of equilibria for monotonic, continuous and non-separable latency functions.

In response to Smith's equilibrium concept, Bernstein and Smith (BS)~\cite{BernsteinS94}  proposed an alternative equilibrium concept addressing the issue of possible coordinated
deviations of bundles of users. They added the condition that only deviations need to be considered that involve ``small enough''
bundles of users.\footnote{Heydecker~\cite{Heydecker86} introduced yet another equilibrium definition
in response to Smith's definition, where the path costs of $p$ and $q$ are compared after the route switch of $\varepsilon$ units of flow.}
\begin{definition} (Bernstein and Smith~\cite[Definition~2]{BernsteinS94})\label{def:BS}
A  path flow $x^*\in Z$ is a side-constrained \emph{BS-equilibrium}, if for all $p,q\in \Pc_i, x^*_{i,p}>0, i\in I$, we have
\[ \tilde\ell_{p}( x^*)\leq \lim\inf_{\varepsilon\downarrow 0} \tilde\ell_{q}(x^*(\varepsilon,p,q)). \]
\end{definition}
Note that the original definition of Bernstein and Smith~\cite[Definition~2]{BernsteinS94} does not involve capacities but by using latency functions $\tilde\ell$ that jump to $+\infty$ as soon as arc capacities are exceeded
an equilibrium will be capacity-feasible.
Following \citefull{CorreaCapEqInStaticFlows}, the definition of BS-equilibrium can be rephrased as ``no arbitrarily small bundle of drivers on a common path can strictly decrease its cost by switching to another path''. 

\subsection{Beckmann-McGuire-Winsten Equilibria and Variational Inequalities}
For the case of separable latency functions, a subset of BS-equilibria can be characterized as solutions
to an associated convex optimization problem, where the Beckmann-McGuire-Winsten potential function over the
convex space of  capacity-feasible flows is minimized:

\begin{align}
\label{beckmann-opt}\tag{BMW}
\min &  \sum_{e\in E}\int_{0}^{ x_e} \ell_{e}(z) \diff z  \\
\text{s.t.: }& x\in Z.\notag
\end{align}
We obtain the following well-known characterization (see, \eg Patriksson~\cite{Patriksson1994tap}).
\begin{lemma}
\begin{enumerate}
\item $x^*\in Z$ is optimal for~\eqref{beckmann-opt} iff $x^*$ solves the following variational inequality
\begin{equation}\label{var-static}
\scalar{\ell( x^*)}{ x^*- y} \leq 0 \text{ for all $y\in Z$}.
\end{equation}
\item Every optimal $x^*\in Z$  to~\eqref{beckmann-opt}
is a  BS-equilibrium but not vice versa.
\end{enumerate}
\end{lemma}
The  second statement  implies the existence of BS-equilibria:
The space $Z$ is non-empty and compact and by the continuity of the objective in~\eqref{beckmann-opt}, the theorem of
Weierstra\ss\ implies that~\eqref{beckmann-opt} admits an optimal solution (which then is also
a BS-equilibrium).
The first and second statement together show that there are BS-equilibria
which need not solve the variational inequality stated in~\eqref{var-static}.
\citefulls{CorreaCapEqInStaticFlows} termed equilibria coming from optimal solutions
to~\eqref{beckmann-opt} as \emph{Beckmann-McGuire-Winsten} (BMW) equilibria while
\citefulls{Marcotte04} termed them  \emph{Hearn-Larsson-Patriksson} equilibria.
A further useful interpretation of solutions to~\eqref{beckmann-opt} is the use of the dual variables
associated with the capacity constraints $x\leq c$. It was shown in several works (\cf Hearn~\cite{Hearn80},~Daganzo~\cite{Daganzo1977}, Larsson and Patriksson~\cite{Larsson99})
that a  BMW-equilibrium can be interpreted as an unconstrained equilibrium, if the dual variables
are added as additional penalty terms to the user's cost function. In addition to this natural interpretation and possible implementation of BMW-equilibria via prices, the efficiency properties of BMW-equilibria in terms of induced total travel times are particularly appealing compared
to other possible side-constrained equilibria, see \citefulls{CorreaCapEqInStaticFlows}.
\subsection{Discussion}
It is very instructive to restate a remark made by \citefulls{Marcotte04}:
``Defining equilibrium meaningfully in side-constrained transportation networks represents a nontrivial task.''
Indeed, the above presentation already shows some subtle issues arising when defining
a sound notion of side-constrained traffic equilibria. While only a few works
formally introduced a behavioral equilibrium concept involving the notion of feasible $\varepsilon$-deviations (as in Bernstein and Smith~\cite{BernsteinS94}, Dafermos and Sparrow~\cite{DafS69}, Smith~\cite{Smith84}, Heydecker~\cite{Heydecker86}), 
most of the works in the transportation science literature used directly the
optimization formulations of the type~\eqref{beckmann-opt} or the variational inequality formulations as the definition of a side-constrained user equilibrium  (cf.~Daganzo~\cite{Daganzo1977}, Hearn and Ramana~\cite{Hearn98solving} or Larsson and Patriksson~\cite[Remark 11]{Larsson99}).\footnote{We use the term ``behavioral equilibrium concept'' 
in accordance with the concept of a Nash equilibrium for a noncooperative game
involving  payoff functions (or preference relations) and strategy spaces which are implicitly defined
via feasible deviations as in Dafermos and Sparrow~\cite{DafS69}. The term ``behavioral'' as used in Larsson and Patriksson~\cite{Larsson1998b} has the following different meaning:
``As such, these models are behavioural, in the sense that the effects of the side constraints are assumed to be immediately transferable to the perception of travel costs among the trip-makers, for example as queueing delays; their solutions are also characterized and interpreted as flows satisfying the Wardrop equilibrium conditions in terms of generalized travel costs that include link queueing delays.''}
This observation was also made in \citefulls{CorreaCapEqInStaticFlows}:
``It is interesting to note that the model [\eqref{beckmann-opt}] has been used before without the
formal introduction of the concept of a capacitated user equilibrium.'' 

 As we will see later in Section~\ref{sec:counter}, within the realm of \emph{ dynamic traffic assignments},
defining side-constrained dynamic equilibria via variational inequalities is not only imprecise (as it excludes other ``equilibria'' being not of this type)  but also 
leads to flawed statements about equilibrium existence and their characterizations.
We show that the natural infinite dimensional variational inequality formulation 
for a class of volume-constrained dynamic traffic assignments need per se
not be related to an equilibrium solution. 
Instead our goal in this work is to transfer behavioral equilibrium concepts in the spirit of Dafermos and Sparrow~\cite{DafS69},
Bernstein and Smith~\cite{BernsteinS94}, Smith~\cite{Smith84} and Heydecker~\cite{Heydecker86} to the domain of \emph{dynamic} side-constrained traffic assignments.


\section{Unconstrained Dynamic Flows}\label{sec:dynamic}

We consider the following model based on the walk-delay operator model of \citefull{FrieszLTW89}: We are given a finite directed graph $G=(V,E)$ and some fixed planning horizon $\planningInterval \subseteq \IR_{\geq 0}$. Additionally, we have a finite set of commodities $I$ and for every commodity $i \in I$ a source node $s_i \in V$, a sink node $t_i \in V$ and either a fixed network inflow volume $Q_i \geq 0$ (for the model with departure time choice) or a fixed bounded network inflow rate $r_i\in L^2_+(\planningInterval)$ (for the model without departure time choice) where $L^2_+(\planningInterval)$ denotes the set of non-negative $L^2$-integrable functions on $\planningInterval$. 

We denote by $\Pc_i$ a fixed set of $s_i$,$t_i$-walks and assume -- \wlofg -- that these sets are disjoint for different commodities. We then denote by $\Pc \coloneqq \bigcup_{i \in I}\Pc_i$ the set of all relevant walks. Note that we allow general walks instead of just simple paths as travelling along cycles is necessary in certain applications like electric vehicles (\cf \cite{GHP22}) and can also sometimes be advantageous in networks with hard edge-capacities (\eg \Cref{ex:ImprovementByCycles}). 
A flow in this network is given by a vector $h \in L^2_+(\planningInterval)^\Pc$ of $L^2$-integrable functions $h_{p}: \planningInterval \to \IR_{\geq 0}$ denoting the walk inflow rates for all walks of all commodities. We denote by
	\[\Lambda(Q) \coloneqq \Set{h \in L^2_+(\planningInterval)^\Pc | \sum_{p \in \Pc_i}\int_{t_0}^{t_f}h_{p}(t)\diff t = Q_i \text{ for all } i \in I, h_{p} \leq B_{p} \text{ for all } p \in \Pc},\]
and
	\[\Lambda(r) \coloneqq \Set{h \in L^2_+(\planningInterval)^\Pc | \sum_{p \in \Pc_i}h_p(t) = r_i(t) \text{ for almost all } t \in \planningInterval \text{ and all } i \in I}\]
the sets of all feasible walk inflows for the model with and without departure time choice, respectively. In the definition of $\Lambda(Q)$ the values $B_p \in \IR_{\geq 0}$ are some fixed walk-specific bounds on the walk inflow rates. In general, these are needed to ensure existence of equilibria in this model (see \Cref{rem:JustificationPathInflowBounds}). Note, however, that in some models such bounds can be introduced without loss of generality since choosing $B_p$ large enough only excludes flows which cannot be equilibria anyway (\cf \cite[Proposition 5.9]{Han2013}). We also observe that we can always view $\Lambda(r)$ as a subset of some $\Lambda(Q)$ by defining $Q_i \coloneqq \int_{\tStart}^{\tEnd}r_i(t)\diff t$ and choosing $B_p \geq \sup_{t \in \planningInterval}r_i(t)$ for every $p \in \Pc_i$. 
For the case of elastic demands we are given a non-increasing inverse demand function $\Theta_i: [0,Q_i] \to \IR$ such that for any possible demand $Q \leq Q_i$ the value $\Theta_i(Q)$ is the cost threshold at which a volume of $Q$ of all particles of commodity $i$ is still willing to travel while the rest stays at home.

Furthermore, we are given a function
	\[\Psi:  L^2_+(\planningInterval)^\Pc \to  \hat{M}(\planningInterval)^\Pc, h \mapsto \left(\Psi_p(h,\cdot):\planningInterval \to \IR_{\geq 0}\right)_{p \in \Pc}\]
mapping walk inflows to \effWalkDelay, \ie for any walk inflow $h$, commodity $i$, walk $p \in \Pc_i$ and time $t$ the value $\Psi_p(h,t)$ is to be understood as the total travel cost (\eg some weighted sum of travel time, penalty for late arrival and cost of energy consumption on the chosen route) of a particle of commodity $i$ starting at time $t$ to travel along walk $p$ under the traffic state induced by the walk inflow $h$. Here, we denote by $\hat{M}(\planningInterval)$ the set of measurable functions from $\planningInterval$ to $\IR \cup \set{\infty}$.

We can now define three standard types of dynamic equilibria (\cf \eg \cite{ZhuM00,Friesz93,HanFSH15}):

\begin{definition}\label{def:DE}
	\begin{itemize}
		\item 	$h^*\in \Lambda(r)$ is a \emph{dynamic equilibrium with fixed inflow rates}, if for all $i\in I$, the following condition holds:
		\begin{align}\label{eq:de-rate}
			h^*_p(t)&>0 \Rightarrow \Psi_p(h^*,t)\leq \Psi_q(h^*,t)\text{ for almost all }t\in \planningInterval, p, q\in \Pc_i.	
		\end{align}
		\item 	$h^* \in \Lambda(Q)$ is a \emph{dynamic equilibrium with fixed flow volumes and departure time choice}, if for all $i\in I$, there exists a $\nu_i \in \IR$ such that the following conditions hold:
		\begin{align}\label{eq:de-volume}
			\begin{aligned}
				h^*_p(t)>0 &\Rightarrow \Psi_p(h^*,t)\leq\nu_i \text{ for almost all }t\in \planningInterval, p\in \Pc_i\\
				h^*_p(t)<B_p &\Rightarrow \Psi_p(h^*,t)\geq \nu_i \text{ for almost all }t\in \planningInterval, p\in \Pc_i.				
			\end{aligned}
		\end{align}
		\item $(h^*,Q^*)$ with $Q^* \in \IR^I_{\geq 0}$ and $h^* \in \Lambda(Q^*)$ is a \emph{dynamic equilibrium with elastic demands and departure time choice}, if for all $i \in I$, there exists a $\nu_i \in \IR$ such that the following conditions hold:
		\begin{align}\label{eq:de-elastic}
			\begin{aligned}
				h^*_p(t)>0 &\Rightarrow \Psi_p(h^*,t)\leq \nu_i \text{ for almost all }t\in \planningInterval, p\in \Pc_i\\
				h^*_p(t)<B_p &\Rightarrow \Psi_p(h^*,t)\geq \nu_i \text{ for almost all }t\in \planningInterval, p\in \Pc_i \\
				Q^*_i < Q_i &\Rightarrow \nu_i = \Theta_i(Q^*_i) \\
				Q^*_i = Q_i &\Rightarrow \nu_i \leq \Theta_i(Q_i).
			\end{aligned}
		\end{align}
	\end{itemize}
\end{definition}

We observe in the following \namecref{lemma:ElasticCaseReduction} that the model with elastic demands can be seen as a special case of the model with fixed inflow volume. Thus, we will only consider the first two models for the rest of this paper.
Note, that this trick is certainly not new and has been applied for static models before, see~Patriksson~\cite[Section~2.2.4]{Patriksson1994tap}.

\begin{lemma}\label{lemma:ElasticCaseReduction}
	Consider a network $\mathcal{N}$ with elastic demand and departure time choice. We construct from this a new network $\mathcal{N}'$ with fixed flow volumes and departure time choice by adding for every commodity $i \in I$ an additional new edge $\tilde{e}_i$ connecting the commodity's source and sink node and defining $\Pc_i' \coloneqq \Pc_i \cup \Set{\tilde{p}_i}$ where $\tilde{p}_i \coloneqq (\tilde{e}_i)$ is the path consisting of only this new edge. Moreover, for each of these new paths we define the effective walk delay operator by $\Psi_{\tilde{p}_i}(h,t) \coloneqq \Theta_i(\sum_{p \in \Pc_i}\int_{t_0}^{t_f}h_p(t')\diff t')$ and choose $B_{\tilde{p}_i}$ such that $B_{\tilde{p}_i}\cdot(\tEnd-\tStart) > Q_i$. 
	
	Then, every dynamic equilibrium in $\mathcal{N}'$ corresponds (in the natural way) to a dynamic equilibrium in $\mathcal{N}$ and vice versa.
\end{lemma}

\begin{proof}
	First, assume that $h'$ is a dynamic equilibrium with fixed flow volume and departure time choice in $\mathcal{N}'$. Then, we define $Q^*_i \coloneqq \sum_{p \in \Pc_i}\int_{\tStart}^{\tEnd}h'_p(t)\diff t$ for every commodity~$i \in I$ and $h^*$ as the restriction of $h'$ to the subnetwork $\mathcal{N}$. We clearly have $h^* \in \Lambda(Q^*)$ and the first two conditions of~\eqref{eq:de-elastic} are satisfied because \eqref{eq:de-volume} holds for $h'$ (with the same $\nu_i$). Moreover, due to our choice of $B_{\tilde{p}_i}$ there must be a set of positive measure of times $t$ with $h'_{\tilde{p}_i}(t) < B_{\tilde{p}_i}$ and, therefore, \eqref{eq:de-volume} implies $\nu_i \leq \Psi_{\tilde{p}_i}(h',t) = \Theta_i(Q^*_i)$. Finally, if we have $Q^*_i < Q_i$, then there must be a set of positive measures of times $t$ with $h'_{\tilde{p}_i}(t) > 0$ and, thus, \eqref{eq:de-volume} also guarantees $\nu_i \geq \Psi_{\tilde{p}_i}(h',t) = \Theta_i(Q^*_i)$
	
	Now, assume that $(h^*,Q^*)$ is a dynamic equilibrium with elastic demand and departure time choice. Then, we can extend $h^*$ to a flow $h' \in \Lambda(Q)$ by setting $h'_{\tilde{p}_i}(t) \coloneqq \frac{Q_i-Q^*_i}{\tEnd-\tStart}$ for all $t \in \planningInterval$ and $i \in I$. Using the same $\nu_i$ as for $h^*$ we then immediately have that the conditions in~\eqref{eq:de-volume} hold for $h'$ on all paths in $\Pc_i$. So, we only have to show that they also hold for the new paths $\tilde{p}_i$: If we have $h'_{\tilde{p}_i}(t)>0$ at any time $t$, we have $Q^*_i < Q_i$ and, therefore, $\Psi_{\tilde{p}_i}(h',t) = \Theta_i(Q^*_i) = \nu_i$. Otherwise, we still have $\Psi_{\tilde{p}_i}(h',t) = \Theta_i(Q^*_i) \geq \nu_i$ for all times $t$. Hence, in both cases \eqref{eq:de-volume} holds for the new paths~$\tilde p_i$ as well.
\end{proof}

We now want to characterize the first two types of dynamic equilibria using variational inequalities. For this we require that the effective walk delays are bounded in the following sense:

\begin{enumerate}[label=(A\arabic*),series=Assumptions]
	\item For any $p \in \Pc$ and $h \in \Lambda(Q)$ the function $\Psi_p(h,.)$ is bounded.\label[asmpt]{ass:PsiBounded}
\end{enumerate}

This then, in particular, implies that all effective walk delays are $L^2$-integrable, \ie $\Psi_p(h,.) \in L^2(\planningInterval)$. An example for where this assumption may be violated is a network-loading model involving spillback and an \effWalkDelay{} $\Psi$ that just measures the actual delay. Since spillback can lead to gridlock (\cf \cite[p~99]{SeringThesis}), the delay will be $\infty$ for particles caught in such a gridlock. However, even for these cases our model might still be applicable as long as particles always have the option to avoid joining a gridlock and achieving some bounded cost (\eg some alternative route with infinite capacity or a stay at home-option). This will be formalized in \Cref{lemma:RestrictToTruncatedPsi}.

If assumption \ref{ass:PsiBounded} holds, it is well known (\cf \eg \cite{Friesz93,ZhuM00}) that, as in the static case, both kinds of dynamic equilibria can be characterized by variational inequalities, namely
\begin{equation}\label{eqn:vi-fixed-inflow-uc}\tag{\ensuremath{\VI(\Psi,r,\planningInterval)}}
	\begin{aligned}
		\text{Find }h^* \in  \Lambda(r)  \text{ such that:}&\\
		\scalar{\Psi(h^*)}{h-h^*} &\geq 0 \text{ for all }h \in \Lambda(r)
	\end{aligned}
\end{equation}
and
\begin{equation}\label{eqn:vi-fixed-volume-uc}\tag{\ensuremath{\VI(\Psi,Q,\planningInterval)}}
	\begin{aligned}
		\text{Find }h^* \in  \Lambda(Q)  \text{ such that:}&\\
		\scalar{\Psi(h^*)}{h-h^*} &\geq 0 \text{ for all }h \in \Lambda(Q).
	\end{aligned}
\end{equation}
Here, $\scalar{.}{.}$ denotes the canonical scalar product on $L^2(\planningInterval)^\Pc$, i.e.
	\[\scalar{.}{.}: L^2(\planningInterval)^\Pc \times L^2(\planningInterval)^\Pc \to \IR, (f,g) \mapsto \scalar{f}{g} \coloneqq \sum_{p \in \Pc}\int_{\tStart}^{\tEnd}f_p(t)g_p(t)\diff t.\]

\begin{theorem}\label{thm:VICharOfDE}
	Assume that \ref{ass:PsiBounded} holds. Then, a walk inflow $h^* \in \Lambda(r)$ ($h^* \in \Lambda(Q)$) is a dynamic equilibrium with fixed inflow rates (with fixed flow volume) if and only if $h^*$ is a solution to \eqref{eqn:vi-fixed-inflow-uc} (to \eqref{eqn:vi-fixed-volume-uc}).
\end{theorem}

Conditions to guarantee the existence of such an element $h^*$ are given by Lions in \cite[Chapitre 2, Théorème 8.1]{Lions} which, following \citefulls{CominettiCL15}, can be restated as follows (see \cite[Proof of Theorem 4.2]{ZhuM00} for how to derive this version from Lions' result):
\begin{theorem}\label{thm:Lions} 
	Let $C$ be a non-empty, closed, convex and bounded subset of $L^2([a, b])^d$. Let $\A : C \rightarrow L^2([a, b])^d$ be a sequentially weak-strong-continuous mapping. Then, the following variational inequality has a solution $h^* \in C$:
		\begin{equation*}
			\begin{aligned}
				\text{Find }h^* \in  C  \text{ such that:}&\\
				\scalar{\A(h^*)}{h-h^*} &\geq 0 \text{ for all }h \in C.
			\end{aligned}
		\end{equation*}
\end{theorem}
Using this theorem together with \Cref{thm:VICharOfDE} one can derive existence of dynamic equilibria under suitable additional assumptions:

\begin{enumerate}[label=(A\arabic*),resume=Assumptions]
	\item The sets $\Pc_i$ are non-empty and finite.\label[asmpt]{ass:FinitelyManyWalks}\label[asmpt]{ass:nonEmptyPathset}
	\item The walk inflow bounds $B_p$ are chosen such that there exists at least one feasible flow $h \in \Lambda(Q)$.\label[asmpt]{ass:PathInflowBounds}
	\item For any $p \in \Pc$ the mapping $\Lambda(Q) \to L^2(\planningInterval), h \mapsto \Psi_p(h,.)$ is sequentially weak-strong continuous.\label[asmpt]{ass:PsiWScont}
\end{enumerate}

Here, a mapping $\mathcal{A}: K \to X$ from some subset $K \subseteq B$ of a reflexive Banach space $B$ (\eg $L^2(\planningInterval)^\Pc$) to a topological space $X$ is \emph{sequentially weak-strong continuous} if it maps weakly convergent sequences in $B$ to (strongly) convergent sequences in $X$. A sequence $f^n$ converges weakly in $B$ if for any $g \in B^*$ from the dual space the sequence $\scalar{f^n}{g}$ converges in $\IR$.

\begin{theorem}\label{thm:ExistenceUnconstrained}
	For any network and \effWalkDelay s satisfying \labelcref{ass:FinitelyManyWalks,ass:nonEmptyPathset,ass:PathInflowBounds,ass:PsiWScont,ass:PsiBounded} there exists a dynamic equilibrium (both with and without departure time choice).
\end{theorem}

As \Cref{thm:VICharOfDE,thm:ExistenceUnconstrained} can be proven in essentially the same way as analogous existence results for very similar models  (\cf \eg \cite{Han2013,ZhuM00}), we omit the proof here. We do, however, provide a proof for the exact setting used here in \Cref{app:VI-Existence}.


\section{A Counterexample}\label{sec:counter}

We now want to describe the dynamic flow model with hard edge capacities introduced by Zhong, Sumalee, Friesz and Lam~\cite{zhong11}. In order to formally define such edge capacities one first needs to be able to translate the individual particles' strategy choices (\ie a walk inflow~$h$) into the resulting dynamic flow on the edges of the network. This translation process is usually called \emph{network loading} and it typically uses some given edge delay functions $D_e(h,.): \IR_{\geq 0} \to \IR_{\geq 0}$ in order to model the flow dynamics on individual edges. Here, $D_e(h,\theta)$ denotes the travel time along edge $e$ when entering at time $\theta$ under the flow induced by $h$.

Using such edge delays, one can define the edge flow corresponding to a given walk inflow $h$: For this we define the set $\mathcal{R} \coloneqq \Set{(p,j) | p \in \Pc, j \in [\abs{p}]}$, \ie $\mathcal{R}$ contains exactly one element for every walk in $\Pc$ and each of its edges (counted with multiplicity if an edge occurs multiple times on the walk). An edge flow corresponding to a given walk inflow $h$ is then a tuple $(f^+,f^-)$ with $f^+,f^- \in L^2_+(\IR_{\geq 0})^\mathcal{R}$ satisfying the following conditions:
\begin{itemize}
	\item Correspondence to $h$: 
	\[f^+_{p,1}(\theta) = h_p(\theta) \text{ for all } p \in \Pc \text{ and almost all } \theta \in \IR_{\geq 0}.\]
	where we implicitly assume $h_p$ to be zero outside the planning horizon $\planningInterval$.
	\item Flow conservation at the nodes: 
	\[f^+_{p,j}(\theta) = f^-_{p,j-1}(\theta) \text{ for all } (p,j) \in \mathcal{R} \text{ with } j > 1 \text{ and almost all } \theta \in \IR_{\geq 0}.\]
	\item Flow conservation on the edges:
	\begin{align}\label{eq:FlowConservationOnEdges}
		\int_{0}^{\theta + D_e(h,\theta)} f^-_{p,j}(\zeta)\diff\zeta = \int_{0}^\theta f^+_{p,j}(\zeta)\diff\zeta \text{ \fa } (p,j) \in \mathcal{R} \text{ and all } \theta \in \IR_{\geq 0}.
	\end{align}
\end{itemize}
For any such edge flow we denote by 
	\[f^+_e(\theta) \coloneqq \sum_{\substack{(p,j) \in \mathcal{R}:\\e \text{ is } j\text{-th edge on } p}}f^+_{p,j}(\theta) \quad\text{ and }\quad f^-_e(\theta) \coloneqq \sum_{\substack{(p,j) \in \mathcal{R}:\\e \text{ is } j\text{-th edge on } p}}f^-_{p,j}(\theta)\]
the aggregated edge in- and outflow rates.

Two commonly used types of edge delay functions are the linear edge delays and the Vickrey queuing delays. In both models the flow dynamics of an edge $e$ are characterized by two values: The free flow travel time $\tau_e > 0$ and the edge's service rate $\nu_e > 0$. The linear edge delay functions are then defined as
	\[D_e(h,\theta) \coloneqq \tau_e + \frac{x_e(h,\theta)}{\nu_e},\]
where $x_e(h,\theta) \coloneqq \int_{0}^\theta f^+_e(\zeta)\diff\zeta -  \int_{0}^\theta f^-_e(\zeta)\diff\zeta$ is the \emph{flow volume} on edge $e$ at time $\theta$. The edge delays of the Vickrey point queue model are defined by
	\[D_e(h,\theta) \coloneqq \tau_e + \frac{q_e(h,\theta)}{\nu_e}\]
where $q_e(h,\theta) \coloneqq \int_{0}^\theta f^+_e(\zeta)\diff\zeta -  \int_{0}^{\theta+\tau_e} f^-_e(\zeta)\diff\zeta$ is the \emph{queue length} on edge $e$ at time $\theta$, together with the requirements that $f^-_e(\theta) \leq \nu_e$ and $q_e(h,\theta) \geq 0$ for almost all $\theta \in \IR$.

Given such edge delay functions we can then recursively define walk-delay-functions $D_p$ by setting
\begin{align*}
	D_p(h,t) \coloneqq \begin{cases}
		D_e(h,t), &\text{ if } p \text{ consists of a single edge } e \\
		D_e(h,t) + D_{p'}(h,t+D_e(h,t)), &\text{ if } p \text{ starts with edge } e \text{ followed by walk } p'.
	\end{cases}
\end{align*}
A typical choice for the effective walk delay is then $\Psi_p(h,t) \coloneqq D_p(h,t)$ or $\Psi_p(h,t) \coloneqq D_p(h,t) + P(t + D_p(h,t) - T_A)$ where $T_A$ is the desired arrival time and $P$ some penalty function for early/late arrival.

For both linear edge delays and the edge delays of the Vickrey point queue model it is known (\cf \eg \cite{CominettiCL15,Han2013,ZhuM00}) that for every walk inflow $h$, there exists a unique corresponding edge flow. Even more, the resulting mapping from $h$ over $(f^+,f^-)$ to $\Psi$ can be shown to be sequentially weak-strong continuous. Thus, unconstrained dynamic equilibria are guaranteed to exist with respect to both these edge delay functions.

\citefulls{zhong11} take the dynamic flow model with route and departure time choice and linear edge delays and augment it with additional constraints by associating with every edge $e$ a Lipschitz-continuous and bounded capacity function $c_e: \IR_{\geq 0} \to \IR_{\geq 0}$ and defining the set of all feasible flows as 
	\begin{align}\label{eq:ZhongDefS}
		S \coloneqq \Set{h \in \Lambda(Q) | x_e(h,\theta) \leq c_e(\theta) \text{ for all } \theta \in \IR_{\geq 0}, e \in E}.
	\end{align}
They then define side-constrained dynamic user equilibria with route and departure time choice as the solutions $h^*$ to the following variational inequality (\cf \cite[eq. (33)]{zhong11}):
\begin{align}\label{eq:ZhongVI}
	\begin{aligned}
		\text{Find }h^* \in  S  \text{ such that:}&\\
		\scalar{\Psi(h^*)}{h-h^*} &\geq 0 \text{ for all }h \in S
	\end{aligned}
\end{align}
Finally, they claim existence of such solutions (and, therefore, of such equilibria) under some mild additional assumptions (\cf \cite[Proposition 3.2]{zhong11}). As in the analogous existence results for unconstrained equilibria the proof proposed by \citeauthor*{zhong11} critically relies on the convexity of the set of feasible flows $S$. However, in contrast to the case of unconstrained equilibria, this convexity is not obvious for a \setS{} $S$ defined by \eqref{eq:ZhongDefS}. In fact, it turns out that in general this set need not be convex! This not only invalidates the existence proof but also calls into question whether a variational inequality of the form of~\eqref{eq:ZhongVI} is even the right way of defining such equilibria.

To prove our point we will now provide counterexamples both to the convexity of $S$ and the existence of solutions to the variational inequality~\eqref{eq:ZhongVI}. We first provide a simple example where $S$ is non-convex using infinite service rates and fixed network inflow rates. We then show how this example can be adapted and expanded to the exact model used by \citefulls{zhong11} (\ie using only finite service rates and allowing departure time choice). Finally, we show that in this example the variational inequality~\eqref{eq:ZhongVI} does not have a solution.

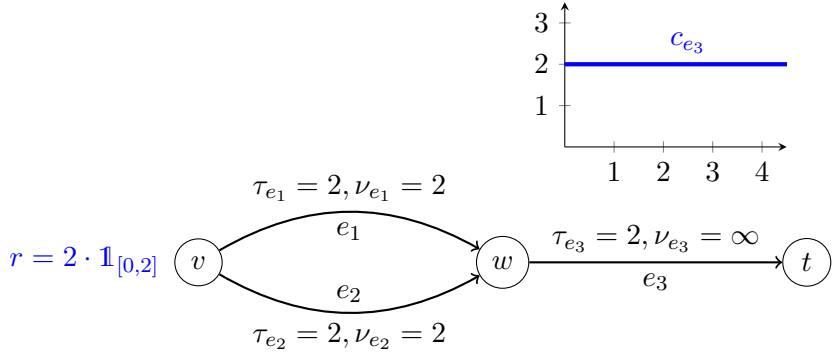
\begin{figure}
	\centering
	\begin{tikzpicture}
	\node[namedVertex] (v) at (0,0) {$v$};
	\node[namedVertex] (w) at (4,0) {$w$};
	\node[namedVertex] (t)  at (8,0) {$t$};
	
	\draw[edge] (v) to[bend left=30] node[below]{$e_1$} node[above]{$\tau_{e_1}=2, \nu_{e_1}=2$} (w);
	\draw[edge] (v) to[bend right=30] node[above]{$e_2$} node[below]{$\tau_{e_2}=2, \nu_{e_2}=2$} (w);
	\draw[edge] (w) -- node[below]{$e_3$} node[above](e3){$\tau_{e_3}=2, \nu_{e_3}=\infty$} (t);
	
	\node[above of=e3, anchor=south,node distance=.5cm] {
		\begin{tikzpicture}[scale=1,solid,black,
			declare function={
				c(\x)= 2;			
			}]

			\begin{axis}[xmin=0,xmax=4.5,ymax=3.5, ymin=0, samples=500,width=4.5cm,height=3.5cm,
				axis x line*=bottom, axis y line*=left, axis lines=middle, xtick={1,2,3,4}]
				\addplot[blue, ultra thick,domain=0:5] {c(x)} node[above,pos=.5]{$c_{e_3}$};
			\end{axis}
			
		\end{tikzpicture}
	};

	\node[left of=v,blue,node distance=1.5cm](){$r=2\cdot\CharF[{[0,2]}]$};
\end{tikzpicture}
	\caption{An instance with constant volume constraint on one edge where the constraint set $S$ defined by edge capacities is not convex when using linear edge delays.}\label{fig:CounterExampleConvexity}
\end{figure}

\begin{example}\label{ex:NonConvexitOfEdgeLoadConstraints}
	Consider the network depicted in \Cref{fig:CounterExampleConvexity} with a single commodity with source $v$, sink $t$ and a constant network inflow rate of $2$ during the planning interval $[0,2]$. If we use the linear edge delays defined by the values for $\tau_e$ and $\nu_e$ given in the figure, the following two walk inflows are feasible: Either sending flow at a rate of $2$ into the path $e_1,e_3$ (during the interval $[0,2]$) or sending flow at a rate of $2$ into the path $e_2,e_3$. In both cases, flow arrives at the intermediate node $w$ at a rate of $1$ during the interval $[2,6]$ as one can verify by a straightforward calculation (see \Cref{fig:FlowVolumeGraph}, left). Thus, the capacity constraint on edge $e_3$ is never violated and both flows are feasible. However, if we send flow at a rate of $1$ into each of the two paths during the interval $[0,2]$, this flow will arrive at node $w$ at a rate of $\frac{2}{3}$ during the interval $[2,5]$ from each of the two edges (see \Cref{fig:FlowVolumeGraph}, right). Thus, it will enter edge $e_3$ at a rate of $\frac{4}{3}$ and, therefore, violate its capacity constraint after time $\theta=2+\frac{3}{2}$. This shows that the set of feasible flows is not convex in this case. 
	
	\begin{figure}\centering
		\begin{tikzpicture}[scale=1,solid,black,
	declare function={
		vol(\x)= (\x <= 2)*2*\x + and(\x > 2, \x <= 6)*(6-\x);			
	}]

	\begin{axis}[xmin=0,xmax=6.5,ymax=4.5, ymin=0, samples=500,width=4.5cm,height=3.5cm,
		axis x line*=bottom, axis y line*=left, axis lines=middle, xtick={1,2,3,4,5,6}]
		\addplot[blue, ultra thick,domain=0:7] {vol(x)} node[right,pos=.5]{$x_{e_1}$};
	\end{axis}
	
\end{tikzpicture}\hspace{2cm}\begin{tikzpicture}[scale=1,solid,black,
	declare function={
		vol(\x)= (\x <= 2)*1*\x + and(\x > 2, \x <= 5)*(10/3-2/3*\x);			
	}]

	\begin{axis}[xmin=0,xmax=6.5,ymax=4.5, ymin=0, samples=500,width=4.5cm,height=3.5cm,
		axis x line*=bottom, axis y line*=left, axis lines=middle, xtick={1,2,3,4,5,6}]
		\addplot[blue, ultra thick,domain=0:7] {vol(x)} node[above,pos=.3]{$x_{e_1}$};
	\end{axis}
	
\end{tikzpicture}
		\caption{The flow volume $x_{e_1}(h,t)$ on edge $e_1$ given a constant inflow rate of $2$ (left) or $1$ (right) during the interval $[0,2]$. See \cite{CareyMcCartney} for more details on computing the flow dynamics with linear edge delays.}\label{fig:FlowVolumeGraph}
	\end{figure}
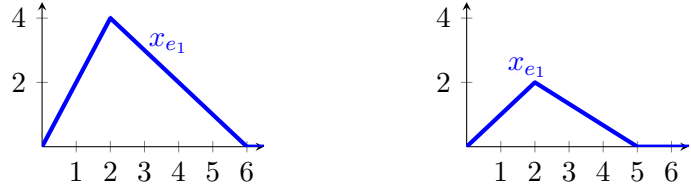
	
	Note, that a similar effect can be observed when using the Vickrey queuing model for the edge dynamics instead. Then, setting $\nu_e=1$ for all edges in the network gives us again an instance where sending all flow at a rate of $2$ into one of the two paths is feasible, while any convex combination of these two inflows violates the capacity constraint on edge~$e_3$.
\end{example}

Now, in order to adjust this example for the model used by \citeauthor*{zhong11} we make the following changes: We add an additional edge~$e_0$ with time-varying capacity constraint at the beginning forcing flow to arrive at node $v$ at a rate of $2$ even when departure time choice is enabled. Next, we replace the infinite service rate on edge $e_3$ by a large but finite service rate of $\frac{4}{\varepsilon}$ and slightly increase the capacity constraint on edge $e_3$ to $2+\varepsilon$. Finally, instead of a fixed inflow rate we now have a fixed flow volume $Q=4$. The resulting network is depicted in \Cref{fig:CounterExample}. 

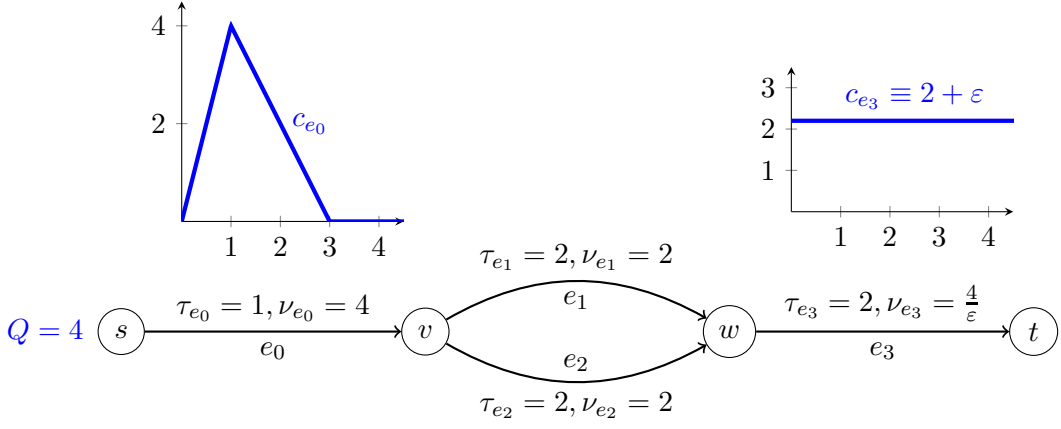
\begin{figure}
	\centering
	\begin{adjustbox}{max width=\textwidth}
		\begin{tikzpicture}
	\node[namedVertex] (s) at (0,0) {$s$};
	\node[namedVertex] (v) at (4,0) {$v$};
	\node[namedVertex] (w) at (8,0) {$w$};
	\node[namedVertex] (t)  at (12,0) {$t$};
	
	\draw[edge] (s) -- node[below](e0){$e_0$} node[above]{$\tau_{e_0}=1, \nu_{e_0}=4$} (v);
	\draw[edge] (v) to[bend left=30] node[below]{$e_1$} node[above]{$\tau_{e_1}=2, \nu_{e_1}=2$} (w);
	\draw[edge] (v) to[bend right=30] node[above]{$e_2$} node[below]{$\tau_{e_2}=2, \nu_{e_2}=2$} (w);
	\draw[edge] (w) -- node[below]{$e_3$} node[above](e3){$\tau_{e_3}=2, \nu_{e_3}=\frac{4}{\varepsilon}$} (t);
	
	\node[above of=e3, anchor=south,node distance=.5cm] {
		\begin{tikzpicture}[scale=1,solid,black,
			declare function={
				c(\x)= 2.2;			
			}]

			\begin{axis}[xmin=0,xmax=4.5,ymax=3.5, ymin=0, samples=500,width=4.5cm,height=3.5cm,
				axis x line*=bottom, axis y line*=left, axis lines=middle, xtick={1,2,3,4}]
				\addplot[blue, ultra thick,domain=0:5] {c(x)} node[above,pos=.5]{$c_{e_3}\equiv2+\varepsilon$};
			\end{axis}
			
		\end{tikzpicture}
	};

	\node[above of=e0, anchor=south,node distance=1cm] {
		\begin{tikzpicture}[scale=1,solid,black,
			declare function={
				c(\x)= and(\x >= 0, \x <= 1)*4*\x + and(\x > 1, \x <=3)*(6-2*\x);			
			}]

			\begin{axis}[xmin=0,xmax=4.5,ymax=4.5, ymin=0, samples=500,width=4.5cm,height=4.5cm,
				axis x line*=bottom, axis y line*=left, axis lines=middle, xtick={1,2,3,4}]
				\addplot[blue, ultra thick,domain=0:5] {c(x)} node[right,pos=.6]{$c_{e_0}$};
			\end{axis}
			
		\end{tikzpicture}
	};	

	\node[left of=s,blue](){$Q=4$};
\end{tikzpicture}
	\end{adjustbox}
	\caption{A counterexample to \cite[Proposition 3.2]{zhong11}.}\label{fig:CounterExample}
\end{figure}

We first show that the capacity constraint on edge $e_0$ does indeed reduce the case with departure time choice to the case of fixed inflow rates:

\begin{claim}\label{claim:FlowOverEdge0}
	Let $h \in S$ be any feasible flow in the network depicted in~\Cref{fig:CounterExample} and $(f^+,f^-)$ its corresponding edge flow. Then, this flow enters edge $e_0$ at a rate of $4$ during the interval $[0,1]$ and leaves it at a rate of $2$ during the interval $[1,3]$ \ie we have
		\[f^+_{e_0} \equiv 4\cdot\CharF[{[0,1]}] \text{ and } f^-_{e_0} \equiv 2\cdot\CharF[{[1,3]}] \text{ almost everywhere.}\]
\end{claim}

\begin{proofClaim}
	We first show that in a feasible flow all flow has entered edge $e_0$ by time $1$. So, let $\theta \coloneqq \min\set{\vartheta | \int_0^\vartheta f^+_{e_0}(\zeta)\diff\zeta = 4}$ be the time the last particle enters edge~$e_0$. We observe that due to the capacity constraint on edge $e_0$ all flow must have left edge $e_0$ by time $3$. Thus, we must have 
	\[3 \geq \theta + D_{e_0}(h,\theta) = \theta + 1 + \frac{1}{4}\flowVolume[e_0](h,\theta)\]
	and, therefore, 
	\begin{align}\label{eq:CounterExampleA}
		\flowVolume[e_0](h,\theta) \leq 8-4\theta
	\end{align}
	as well as $\theta \leq 3-1=2$. Furthermore, since the total flow volume of $4$ must traverse edge $e_0$, at time $\theta$ we must have
	\begin{align}\label{eq:CounterExampleB}
		\int_0^\theta f^-_{e_0}(\zeta)\diff\zeta + \flowVolume[e_0](h,\theta) = 4.
	\end{align}
	Now, let $\theta'$ be the time at which particles have to enter edge $e_0$ in order to leave it at time~$\theta$, \ie $\theta'$ satisfies 
	\begin{align}\label{eq:CounterExampleE}
		\theta = \theta' + D_{e_0}(h,\theta') = \theta' + 1 + \frac{1}{4}\flowVolume[e_0](h,\theta').
	\end{align}		
	Since the edge travel time along edge $e_0$ is always at least $1$ we must have $\theta' \leq \theta-1 \leq 2-1 = 1$ and, hence, no particle has left edge $e_0$ by time $\theta'$. Thus, we have 
	\begin{align*}
		\flowVolume[e_0](h,\theta') = \int_0^{\theta'} f^+_{e_0}(\zeta)\diff\zeta \overset{\text{\eqref{eq:FlowConservationOnEdges}}}{=} \int_0^{\theta'+D_{e_0}(h,\theta')}f^-_{e_0}(\zeta)\diff\zeta = \int_0^\theta f^-_{e_0}(\zeta)\diff\zeta.
	\end{align*}
	Together with \eqref{eq:CounterExampleB} this gives us
	\begin{align}\label{eq:CounterExampleD}
		\flowVolume[e_0](h,\theta') = 4 - \flowVolume[e_0](h,\theta).
	\end{align}
	From this, we can now deduce
	\begin{align*}
		4\theta' = c_{e_0}(\theta') \geq \flowVolume[e_0](h,\theta') \overset{\text{\eqref{eq:CounterExampleD}}}{=} 4 - \flowVolume[e_0](h,\theta) 
		\overset{\text{\eqref{eq:CounterExampleA}}}{\geq} 4 - 8 + 4\theta = 4\theta-4
	\end{align*}
	and, thus,
	\begin{align*}
		4 \geq 4\theta - 4\theta' \overset{\text{\eqref{eq:CounterExampleE}}}{=} 4\theta' + 4 + \flowVolume[e_0](h,\theta') - 4\theta' = 4 + \flowVolume[e_0](h,\theta'),
	\end{align*}
	which gives us $\flowVolume[e_0](h,\theta') = 0$. Plugging this back into \eqref{eq:CounterExampleD} gives us $\flowVolume[e_0](h,\theta) = 4$ and, finally, feasibility of $h$ together with the the capacity constraint on edge $e_0$ then implies $\theta=1$ and, therefore,
		\begin{align}\label{eq:AllFlowEnteredByTime1}
			\int_0^\vartheta f^+_{e_0}(\zeta)\diff\zeta = 4 \text{ for all } \vartheta \geq 1.
		\end{align}
	
	Now, take any time $\theta \in [0,1]$ and let $\bar\theta \coloneqq \theta+D_{e_0}(h,\theta)$ be the time particles entering edge $e_0$ at time $\theta$ arrive at node $v$. Since $\bar\theta \in [1,3]$, the feasibility of $h$ implies
		\begin{align*}
			4-2\theta-\frac{\flowVolume[e_0](h,\theta)}{2} &= 6-2\left(\theta+1+\frac{\flowVolume[e_0](h,\theta)}{4}\right) = 6-2\bar\theta = c_{e_0}(\bar\theta) \\
				&\geq \flowVolume[e_0](h,\bar\theta) = \int_0^{\bar\theta}f^+_{e_0}(\zeta)\diff\zeta - \int_0^{\bar\theta}f^-_{e_0}(\zeta)\diff\zeta \overset{\text{\eqref{eq:AllFlowEnteredByTime1}}}{=} 4 - \int_0^{\bar\theta}f^-_{e_0}(\zeta)\diff\zeta \\
				&\overset{\text{\eqref{eq:FlowConservationOnEdges}}}{=} 4-\int_0^\theta f^+_{e_0}(\zeta)\diff\zeta = 4 - \flowVolume[e_0](h,\theta).
		\end{align*}
	which implies $\flowVolume[e_0](h,\theta) \geq 4\theta$.
	At the same time, feasibility of $h$ implies $\flowVolume[e_0](h,\theta) \leq 4\theta$ and, thus, we have $\flowVolume[e_0](h,\theta)=4\theta$ for all $\theta \in [0,1]$. Since no flow leaves edge $e_0$ before time $\theta=1$, this implies $\int_0^\theta f^+_{e_0}(\zeta)\diff\zeta = 4\theta$ and, therefore, $f^+_{e_0}(\zeta) = 4$ for almost all $\zeta \in [0,1]$. A direct computation then shows $f^-_{e_0}(\zeta) = 2$ for almost all $\zeta \in [1,3]$ as well.
\end{proofClaim}

Next, we define three path inflows: Flow $h^1$ sends flow into the upper path $e_0,e_1,e_3$ at a rate of $4$ during the interval $[0,1]$. Flow $h^2$ sends flow into the lower path $e_0,e_2,e_3$ at a rate of $4$ during the same interval. And, finally, $h^3$ splits the flow equally between the two paths, \ie sends flow at a rate of $2$ into each of the two paths. We now claim that $h^1$ and $h^2$ are feasible while $h^3$ is not.

\begin{claim}
	In the network from \Cref{fig:CounterExample} the path inflows $h^1$ and $h^2$ are feasible. The path inflow $h^3$ is infeasible for $\varepsilon < \frac{2}{3}$.
\end{claim}

\begin{proofClaim}
	Both inflows $h^1$ and $h^2$ result in flow arriving at a rate of $1$ at node $w$ (without violating the capacity constraint on edge $e_0$ on the way). Since the total flow volume in the network is bounded by $4$, the travel time on edge $e_3$ is never larger than $2+\frac{4}{\frac{4}{\varepsilon}} = 2 + \varepsilon$. Thus, flow entering this edge at a rate of $1$ will never violate its capacity constraint. Therefore, both $h^1$ and $h^2$ are feasible.
	
	The equal split between the two paths in $h^3$ results in flow entering edge $e_1$ and $e_2$ at a rate of $1$ each during $[1,3]$. This flow exits these edges at a rate of $\frac{2}{3}$ during the interval $[3,6]$. Thus, the flow enters edge $e_3$ at a combined rate of $\frac{4}{3}$ during this interval. At time $\theta=3+\frac{3}{4}\cdot(2+\varepsilon)$ a flow volume of $2+\varepsilon$ has entered edge $e_3$. If $\varepsilon < \frac{2}{3}$, this happens before time $\theta=5$ and, in particular, before any flow has left edge $e_3$. Thus, the total flow volume on edge $e_3$ is $2+\varepsilon$ at this time, leading to a violation of the capacity constraint immediately after.
\end{proofClaim}

Since $h^3$ is a convex combination of $h^1$ and $h^2$, this claim already shows the non-convexity of the set $S$ of feasible flows. To show that the variational inequality~\eqref{eq:ZhongVI} has no solution, we need the following additional property of all feasible flows in the given instance.

\begin{claim}
	Let $h \in S$ be any feasible path inflow for the instance in \Cref{fig:CounterExample}. Then, one of the two paths has a total inflow volume of at most $4\varepsilon$.
\end{claim}

\begin{proofClaim}
	Define $\theta_1$ as the last time a particle can enter edge $e_1$ under the flow induced by $h$ and still arrive by time $5$ at the node $w$ and define $\theta_2$ analogous for edge $e_2$. That is, we choose $\theta_1$ and $\theta_2$ such that we have
		\[\theta_1 + D_{e_1}(h,\theta_1) = 5 \quad\text{ and }\quad \theta_2 + D_{e_2}(h,\theta_2) = 5\]
	or, equivalently
	\begin{align}\label{eq:0}
		\theta_1 + \frac{x_{e_1}(h,\theta_1)}{2} = 3 \quad\text{ and }\quad \theta_2 + \frac{x_{e_2}(h,\theta_2)}{2} = 3.
	\end{align}
	Without loss of generality we assume that $\theta_1 \leq \theta_2$. Since $\theta_1,\theta_2 \leq 3 = \tau_{e_0}+\tau_{e_1} = \tau_{e_0}+\tau_{e_2}$, no flow has left edges $e_1$ and $e_2$ by time $\theta_2$ and, thus
	\[2(\theta_1-1) \overset{\text{\Cref{claim:FlowOverEdge0}}}{=} \int_0^{\theta_1}f^-_{e_0}(\zeta)\diff \zeta \leq x_{e_1}(h,\theta_1) + x_{e_2}(h,\theta_2).\]
	At the same time we also have
	\begin{align}\label{eq:4}
		x_{e_1}(h,\theta_1) + x_{e_2}(h,\theta_2) \leq x_{e_3}(h,5) \leq c_{e_2}(5) = 2+\varepsilon
	\end{align}
	since all this flow has entered edge $e_3$ by time $5$ but no flow has left it. Together, this implies $2(\theta_1-1) \leq 2 + \varepsilon$ and, therefore,
	\begin{align}\label{eq:1}
		\theta_1 \leq 2 + \frac{\varepsilon}{2}.
	\end{align}
	Since flow arrives at node $v$ at a rate of $2$ during $[1,3]$ (\cf \Cref{claim:FlowOverEdge0}), we also have
	\[x_{e_1}(h,\theta_1) \leq 2(\theta_1-1)\]
	while \eqref{eq:0} implies
	\begin{align}\label{eq:3}
		x_{e_1}(h,\theta_1) = 6 - 2\theta_1.
	\end{align}
	Together, this gives us $\theta_1 \geq 2$ which, combined with \eqref{eq:1}, results in $\theta_1 \in [2,2+\frac{\varepsilon}{2}]$. Using~\eqref{eq:3} we then get $x_{e_1}(h,\theta_1) \in [2-\varepsilon,2]$ and, with \eqref{eq:4},
	\[x_{e_2}(h,\theta_2) \leq 2 + \varepsilon - x_{e_1}(h,\theta_1) \leq 2 + \varepsilon - (2 - \varepsilon) = 2\varepsilon.\]
	Plugging this back into \eqref{eq:0} and solving for $\theta_2$ gives
	\[\theta_2 = 3 - \frac{x_{e_2}(h,\theta_2)}{2} \geq 3- \frac{2\varepsilon}{2} = 3-\varepsilon.\]
	This, using again the fact that flow arrives at node $v$ at a rate of $2$ after time $\theta=1$, gives us
		\[x_{e_2}(h,3) \leq x_{e_2}(h,\theta_2) + 2\cdot(3-\theta_2) \leq 2\varepsilon + 2\varepsilon = 4\varepsilon.\]
	Since no flow leaves edge $e_2$ before time $3$ and no flow enters after time $3$, this proves our claim.
\end{proofClaim}

Now, let $h \in S$ be any feasible flow and $e_2$ the edge with a total inflow of at most $4\varepsilon$. We want to show that $\scalar{\Psi(h)}{h^2-h}<0$. Since the free flow travel times on both paths are equal, we only have to consider the additional flow dependent delays here (as induced by~$h$). Furthermore, as both $h$ and $h^2$ are feasible flows, they are the same on edge $e_0$ (by \Cref{claim:FlowOverEdge0}). Thus, the delays experienced on edge $e_0$ cancel out.
For flow $h^2$ edge $e_2$ contributes an additional delay of at most $\frac{4\varepsilon}{2}$ per particle and edge $e_3$ an additional delay of at most $\varepsilon$, leading to a total additional delay of at most $4\cdot(2\varepsilon+\varepsilon)=12\varepsilon$.
For flow $h$, on the other hand, a flow volume of at least $4-4\varepsilon$ traverses edge $e_1$ and is delayed there by all other particles having entered this edge before it (note, that by \Cref{claim:FlowOverEdge0} all particles enter edge $e_1$ during the interval $[1,3]$ and this edge has a travel time of at least $2$). Thus, the total additional delay here is at least
	\[\int_0^{4-4\varepsilon}\frac{x}{2}\diff x = 4-8\varepsilon+4\varepsilon^2.\]
This implies
	\[\scalar{\Psi(h)}{h^2-h} = \scalar{\Psi(h)}{h^2} - \scalar{\Psi(h)}{h} \leq 12\varepsilon - (4-8\varepsilon) = 20\varepsilon - 4\]
which is strictly smaller than $0$ for $\varepsilon<\frac{1}{5}$. Thus, for such $\varepsilon$ no feasible flow $h$ can be a solution to the variational inequality~\eqref{eq:ZhongVI}. 

Note, that whether or not the instances in \Cref{fig:CounterExampleConvexity,fig:CounterExample} \emph{should} have an equilibrium ought to depend on the behavioural model, i.e., which alternative strategies are available to individual particles under a given flow. However, a variational inequality of the form of~\eqref{eq:ZhongVI} cannot capture such subtleties as it essentially presents \emph{every} feasible flow as a potential alternative regardless of how different to the current flow it is. In particular, in the first example, the only existing alternative flow can be attained only by a \emph{collective} deviation of all particles and, therefore, should not be relevant for an individual particle's choice. If, on the other hand, we only consider small ``$\varepsilon$-deviations'', the two feasible flows may very well be equilibria as no such small deviation leads to another feasible flow and, thus, one can argue that individual particles have no potential alternative strategies here.

\begin{framed}
	This discussion leads to the following key questions:
	\begin{itemize}
		\item What is a reasonable and useful definition of side-constrained equilibria for dynamic flows?
		\item Under which conditions can those equilibria be characterized by a variational inequality?
		\item Under which conditions is existence of such equilibria guaranteed?
	\end{itemize}
\end{framed}
We will try to provide answers to all three questions in the following three sections.


\section{A General Framework for Side-Constrained Dynamic Equilibria}\label{sec:framework}

The general idea of dynamic equilibria is to find a walk inflow $h$ such that (almost) no particle has a better alternative. In the unconstrained case, ``alternative'' just means any other walk (and departure time) available to that particle's commodity. In particular, the set of all alternatives of any given particle is independent of the specific flow induced by $h$. This need not be the case anymore for flows with side-constraints as, for example, a certain walk may become unavailable depending on the flow volume already on any particular edge of this walk under $h$.

We will now introduce an abstract framework that will allow us to define different types of such side-constrained equilibria. This framework consists of two objects: First, we have a constraint set $S \subseteq \Lambda(Q)$ defining the set of feasible walk inflows. Second, for every commodity $i \in I$ and walk $p \in \Pc_i$, there is a correspondence
	\[A_p: S \to 2^{\Pc_i \times L_+^2(\planningInterval) \times \IR}, h \mapsto A_p(h),\]
mapping every walk inflow $h$ to the set of \emph{\addmEpsDev s} of commodity $i$ from walk $p$. Hereby, $2^X$ denotes the power set of $X$. An element $(q,\shiftN,\Delta) \in A_{p}(h)$ then denotes the following \addmEpsDev: Given the walk inflow $h$, commodity $i$ is allowed to shift flow from walk~$p$ to walk~$q$ by reducing the walk-inflow into~$p$ by~$\shiftN \in L_+^2(\planningInterval)$ and increasing the one into~$q$ using the same function~$\shiftN$ but with an additional time shift of~$\Delta$. 

For any such $(q,\shiftN,\Delta) \in A_{p}(h)$ we will denote the walk inflow obtained after this deviation by
\begin{equation}\label{eq:H-r}
	H_{p\to q}(h,\shiftN,\Delta) \coloneqq (h'_w)_{w \in \Pc} \text{ with } h'_w(t) \coloneqq \begin{cases} 
		h_p(t) - \shiftN(t), &\text{for } w=p\\
		h_q(t) + \shiftN(t-\Delta), &\text{for } w=q\\
		h_{w}(t), &\text{else}
	\end{cases}.
\end{equation}
\Cref{fig:ExampleAddmEpsDev} provides a graphical depiction for how $H_{p\to q}(h,\shiftN,\Delta)$ is obtained from $h$.
In order to ensure that we have $H_{p \to q}(h,\shiftN,\Delta) \in \Lambda(Q)$, we will always require that $A_p$ are defined such that any \addmEpsDev{} $(q,\shiftN,\Delta) \in A_p(h)$ satisfies the following three properties:

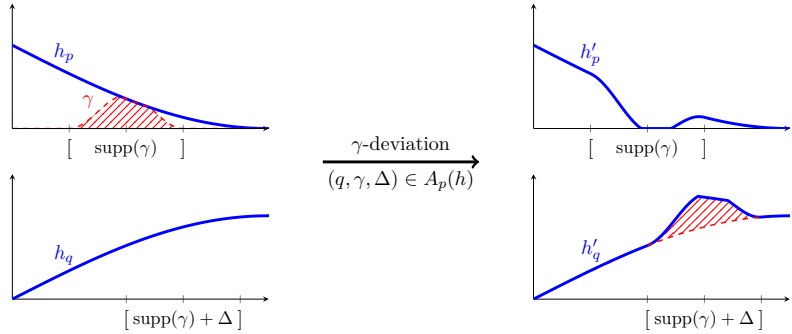
\begin{figure}[h]
	\centering
	\begin{adjustbox}{max width=.7\textwidth}
		\begin{tikzpicture}[scale=1,solid,black,
	declare function={
		hp(\x)= 2-2*sin(20*\x);	
		hq(\x)= 2*sin(20*\x);
		dev(\x) =  and(\x>1,\x<3)*min(hp(\x),.4*(sin(180*(\x-1.5))+1));
		zero(\x) = 0;		
	}]
	\node at(0,1.5){	
		\begin{tikzpicture}[scale=1,solid,black]
		\begin{axis}[xmin=0,xmax=4.5,ymax=3, ymin=0, samples=500,width=6.5cm,height=4cm,
			axis x line*=bottom, axis y line*=left, axis lines=middle, xtick={1,2,3},xticklabels={$[$,\small$\supp(\shiftN)$,$]$},ytick=\empty]
			\addplot[blue, ultra thick,domain=0:5] {hp(x)} node[above,pos=.2]{$h_p$};
			
			\addplot[red,dashed,thick,name path=F,domain=0:5] {dev(x)}  node[above=.1cm,pos=.25]{$\shiftN$};
			\addplot[name path=G] {zero(\x)};
			
			\addplot[pattern=north east lines, pattern color=red, thick]fill between[of=F and G];
		\end{axis}
		\end{tikzpicture}
	};
		
	\node at(0,-1.8){	
		\begin{tikzpicture}[scale=1,solid,black]
		\begin{axis}[xmin=0,xmax=4.5,ymax=3, ymin=0, samples=500,width=6.5cm,height=4cm,
			axis x line*=bottom, axis y line*=left, axis lines=middle, xtick={2,3,4},xticklabels={$[$,\small$\supp(\shiftN)+\Delta$,$]$},ytick=\empty]
			\addplot[blue, ultra thick,domain=0:5] {hq(x)} node[above,pos=.2]{$h_q$};
		\end{axis}
		\end{tikzpicture}
	};

	\node at(10,1.5) {
		\begin{tikzpicture}[scale=1,solid,black]
		\begin{axis}[xmin=0,xmax=4.5,ymax=3, ymin=0, samples=500,width=6.5cm,height=4cm,
			axis x line*=bottom, axis y line*=left, axis lines=middle, xtick={1,2,3},xticklabels={$[$,\small$\supp(\shiftN)$,$]$},ytick=\empty]
			\addplot[blue, ultra thick,domain=0:5] {hp(x)-dev(x)} node[above,pos=.2]{$h'_p$};
		\end{axis}
		\end{tikzpicture}
	};

	\node at(10,-1.8){
		\begin{tikzpicture}[scale=1,solid,black]
		\begin{axis}[xmin=0,xmax=4.5,ymax=3, ymin=0, samples=500,width=6.5cm,height=4cm,
			axis x line*=bottom, axis y line*=left, axis lines=middle, xtick={2,3,4},xticklabels={$[$,\small$\supp(\shiftN)+\Delta$,$]$},ytick=\empty]
			\addplot[blue, ultra thick,domain=0:5,name path=Gq] {hq(x)+dev(x-1)} node[above,pos=.2]{$h'_q$};
			
			\addplot[domain=2:4,red,dashed,thick,name path=Fq] {hq(x)};
			
			\addplot[domain=2:4,pattern=north east lines, pattern color=red, thick]fill between[of=Fq and Gq];
		\end{axis}
		\end{tikzpicture}
	};
	
	\draw[->,ultra thick] (3.5,0) -- node[above]{\epsDev} node[below]{$(q,\shiftN,\Delta) \in A_p(h)$} +(3,0);
\end{tikzpicture}
	\end{adjustbox}
	\caption{An example for how an \addmEpsDev{} $(q,\shiftN,\Delta) \in A_{p}(h)$ changes the path inflow rates on the involved paths $p$ and $q$ from the original flow $h$ (left) to the new flow $h' \coloneqq H_{p\to q}(h,\shiftN,\Delta)$ (right).}\label{fig:ExampleAddmEpsDev}
\end{figure}

\begin{itemize}
	\item $\shiftN \leq h_p$, \ie we do not remove too much flow. This ensures that the walk-inflow rates stay non-negative.
	\item $h_q(t) + \shiftN(t-\Delta) \leq B_q$ for almost all $t \in \planningInterval$, \ie we do not move too much flow onto walk~$q$.
	\item $\supp(\shiftN(\emptyarg-\Delta)) \subseteq \planningInterval$, \ie no flow is moved outside the planning horizon.\footnote{As $\shiftN \in L^2(\planningInterval)$ is formally an equivalence class of functions, its support is only defined up to a null-set. Hence, an inclusion $\supp(\shiftN) \subseteq A$ should be interpreted as $\supp(\shiftN)\setminus A$ having measure zero for any choice of representative of~$\shiftN$.}
\end{itemize}
This, in particular, ensures that $\Psi$ is well defined on any such flow $H_{p \to q}(h,\shiftN,\Delta)$.

The correspondence
\[
\begin{aligned}
	M_i: &S \to L^2_+ ([t_0,t_f])^{\Pc}\\ 
	& h \mapsto \set{h' \in L^2 ([t_0,t_f])^{\Pc} | \exists p \in \Pc_i, (q,\shiftN,\Delta) \in A_p(h): h' = H_{p \to q}(h,\shiftN,\Delta)}\end{aligned}\]
then returns for any given walk inflow $h$ the set of all possible walk inflows obtained by any of commodity $i$'s \addmEpsDev. Note that, in general, $M_i$ and $S$ can be completely independent of each other. In particular, a flow obtained by an \addmEpsDev{} might be infeasible (\ie $M_i(h) \not\subseteq S$) and not all feasible flows might be reachable by an \addmEpsDev{} even if they only differ by a single commodity's deviation (\ie $\Set{H_{p\to q}(h,\shiftN,\Delta) \in S} \not\subseteq M_i(h)$).

Now, any constraint set $S$ together with a family of deviation correspondences $A_p$ gives rise to the following informal equilibrium concept: A walk inflow $h$ is an equilibrium with respect to $S$ and $A_p$, if it is feasible and no particle can improve by an \addmEpsDev. 
To make this mathematically precise, 
we introduce the concept of \emph{\addmDev s} to some fixed walk $p$ of commodity $i$ with respect to some $h \in S$ at any fixed time $t \geq 0$ by defining the set
\begin{align}\label{eq:feasible-deviation} 
	\!U_{p}(h, t) 
	\coloneqq \Set{\!
		(q,\Delta)\in \Pc_i \times \IR |\!\! \begin{array}{l}
			\forall \delta > 0, \varepsilon > 0: \exists \shiftN \in L^2_+(\planningInterval): \int_{\tStart}^{\tEnd}\shiftN(\zeta)\diff\zeta > 0, \\ \shiftN \leq \varepsilon, \supp(\shiftN) \subseteq [t-\delta,t+\delta], \text{ and } (q,\shiftN,\Delta) \in A_p(h)
		\end{array}\!\!\!
	}.
\end{align}
In words: A walk $q$ and a shift $\Delta$ form an admissible alternative to $p$ at time $t$, if we can shift arbitrarily small amounts of flow in space from $p$ to $q$ and in time from small neighbourhoods of $t$  to small neighbourhoods of $t+\Delta$. 
This concept can be interpreted as a dynamic extension of the methodology of Bernstein and Smith~\cite{BernsteinS94}.

Note that, if there is no inflow into a walk $p$ in some neighbourhood of $t$, then $U_{p}(h,t)$ will always be empty (regardless of the state of the alternative walks). If we assume that the trivial deviation  (\ie shifting flow from~$p$ to~$p$ at any rate but without time shift) is always admissible then the converse also holds, \ie we have: $U_{p}(h,t)$ is non-empty if and only if there is inflow into $p$ near $t$. 

From now on, we will often assume that the \effWalkDelay{} for any given walk inflow $h$ is continuous.

\begin{enumerate}[label=(A\arabic*),resume=Assumptions]
	\item For any $h \in S$ and $p \in \Pc$, the function $\Psi_{p}(h,.): [t_0,t_f] \to \IR$ is continuous.\label[asmpt]{ass:EffectivePathDelayContinuous}
\end{enumerate}

We can then formally define the concept of a side-constrained dynamic equilibrium in our model as follows:
\begin{framed}
\begin{definition}\label{def:DCE}
	Given a graph $G$, a set of commodities $I$, a set of feasible walks $\Pc$, an \effWalkDelay{} operator $\Psi$, a constraint set $S$ and for every walk $p \in \Pc$ a correspondence $A_p$. Then, a feasible flow $h^*\in S$ is a \emph{\SCDE[full] (\SCDE)} \wrt $S$ and $A_p$, if for all $p \in \Pc$ and all $t \in [t_0,t_f]$ the following condition holds:
	\begin{equation}\label{eq:CDE} 
		\Psi_p(h^*,t) \leq  \Psi_q(h^*,t+\Delta) \text{ for all } (q,\Delta) \in U_{p}(h^*, t).
	\end{equation}
\end{definition}
\end{framed}

Note that the above model encompasses both (side-constrained) dynamic equilibria with and without departure time choice. In particular, to disable departure time choice we just choose some $S \subseteq \Lambda(r)$ and define $A_p$ such that it only contains deviations of the form~$(q,\shiftN,0)$.

\subsection{Basic Properties of \texorpdfstring{\SCDE}{SCDE}}

We observe that an alternative (negative) definition of side-constrained dynamic equilibria is as follows: A flow is \emph{not} an equilibrium if there exist two walk-time pairs $(p,t)$ and $(q,t')$ such that the cost of walk $q$ at time $t'$ is strictly lower than that of walk $p$ at time $t$ and we can shift flow in arbitrarily small neighbourhoods of $t$ from $p$ to neighbourhoods of $t'$ on $q$. This is formalized in the following \namecref{lemma:NegativeCharacterization}:
\begin{lemma}\label{lemma:NegativeCharacterization}
	A flow $h \in S$ is \emph{not} an \SCDE{} if and only if there exists a commodity $i \in I$, walks $p,q \in \Pc_i$, times $t,t' \in \planningInterval$, a sequence of functions $\shiftN_n \in L^2_+(\planningInterval)$ such that
	\begin{enumerate}[label=\alph*)]
		\item $\lim_n \esssup \shiftN_n = 0$,
		\item $\lim_n \sup\set{\theta \in \planningInterval | \supp(\shiftN_n) \cap [\tStart,\theta] \text{ has measure zero}} \geq t$ and\\ $\lim_n \inf\set{\theta \in \planningInterval | \supp(\shiftN_n) \cap [\theta,\tEnd] \text{ has measure zero}} \leq t$, 
		\item $\Psi_p(h,t) > \Psi_q(h,t')$ and
		\item $\int_{\tStart}^{\tEnd} \shiftN_n(\theta) \diff\theta > 0$ and $(q,\shiftN_n,\Delta) \in A_p(h)$ for all $n \in \IN$.
	\end{enumerate}
\end{lemma}

\begin{proof}
	If the \namecref{lemma:NegativeCharacterization}'s conditions are satisfied then a), b) and d) together show that we have $(q,t'-t) \in U_{p}(h,t)$. Then, c) implies that $h$ is not an \SCDE.
	
	If, on the other hand, we know that $h$ is not a \SCDE[full], then there must be some $t \in [t_0,t_f]$, $i \in I$, $p,q \in \Pc_i$ and $\Delta \in \IR$ such that both $\Psi_p(h,t) > \Psi_q(h,t+\Delta)$ and $(q,\Delta) \in U_{p}(h,t)$ hold. The latter then implies the existence of $\shiftN_n \in L_+^2(\planningInterval)$ satisfying a), b) and d) while the former is exactly c) with $t'\coloneqq t + \Delta$.
\end{proof}

The next \namecref{lemma:ConditiononApForStrongerEquilibria} provides a condition for two sets of \addmEpsDev s under which one of the corresponding \SCDE{} is stronger than the other. Two direct consequences of this \namecref{lemma:ConditiononApForStrongerEquilibria} are:
\begin{enumerate}[label=\alph*)]
	\item Smaller sets of \addmEpsDev s lead to weaker forms of \SCDE{} (we will make extensive use of this in \Cref{prop:RelationshipsOfCDE}).
	\item In many cases, one can restrict the set of \addmEpsDev s to some (simpler) subset without changing the resulting equilibria. For example, by only allowing functions of the form $\shiftN=\varepsilon\CharF[J]$ where $\CharF[J]$ denotes the characteristic function of a measurable subset $J \subseteq \planningInterval$ (as it has been done previously in \cite[Definition~4]{GHP22}). 
\end{enumerate}

\begin{lemma}\label{lemma:ConditiononApForStrongerEquilibria}
	Let $S$ be a \setS{} and $A_p$ and $A'_p$ two correspondences defining \addmEpsDev s. Furthermore, let $h^* \in S$ be a feasible flow such that the sets of \addmEpsDev s satisfy the following condition for every walk $p \in \Pc$:
	\[\forall (q,\shiftN,\Delta) \in A_p(h^*), \shiftN \neq 0: \exists (q,\shiftN',\Delta) \in \A'_p(h^*): \shiftN' \neq 0, \shiftN' \leq \shiftN.\]
	If this $h^*$ is an \SCDE{} \wrt $S$ and $A'_p$, then it is also an \SCDE{} \wrt $S$ and $A_p$.
\end{lemma}

\begin{proof}
	We prove this \namecref{lemma:ConditiononApForStrongerEquilibria} by showing that the given condition implies $U'_p(h^*,t) \supseteq U_p(h^*,t)$ for all $t \in [t_0,t_f]$ (where $U'_p$ denotes the \addmDev s defined by $A'_p$):
	
	Take any \addmDev{} $(q,\Delta) \in U_p(h^*,t)$. Then, for any $\delta, \varepsilon > 0$ there must exist some $(q,\shiftN,\Delta) \in A_p(h^*)$ with $\shiftN \leq \varepsilon$, $\int_{\tStart}^{\tEnd}\shiftN(t')\diff t' > 0$ and $\supp(\shiftN) \subseteq [t-\delta,t+\delta]$. The \namecref{lemma:ConditiononApForStrongerEquilibria}'s assumption now guarantees the existence of some $(q,\shiftN',\Delta) \in A'_p(h^*)$ with $\int_{\tStart}^{\tEnd}\shiftN'(t')\diff t' > 0$, $\shiftN' \leq \gamma \leq \varepsilon$ and $\supp(\shiftN') \subseteq \supp(\shiftN) \subseteq [t-\delta,t+\delta]$. As this holds for any $\delta,\varepsilon > 0$, this shows that we have $(q,\Delta) \in U'_p(h^*,t)$.
	
	Hence, being an \SCDE{} \wrt $A'_p$ is a stronger condition than being an \SCDE{} \wrt $A_p$, which proves the \namecref{lemma:ConditiononApForStrongerEquilibria}.
\end{proof}

Finally, the following two lemmas show that in many cases, we can \wlofg restrict ourselves to smaller sets of feasible walks $\Pc_i$ (\eg in order to make them finite and satisfy \ref{ass:FinitelyManyWalks}) or assume that the effective walk-delay operators $\Psi_p$ are bounded (\ie satisfy \ref{ass:PsiBounded}). For the first lemma we make use of the concept of dominating walks as introduced in \cite{GHP22}.

\begin{definition}
	We call a subset $\Pc' \subseteq \Pc$ a \emph{dominating walk set}, if for every walk $p \in \Pc$, $h \in S$, $t \in [t_0,t_f]$ and $(q,\Delta) \in U_p(h,t)$, there exists some $(q',\Delta') \in U_p(h,t)$ such that $q' \in \Pc'$ and $\Psi_{q'}(h,t+\Delta') \leq \Psi_q(h,t+\Delta)$.
\end{definition}

\begin{lemma}\label{lemma:RestrictToDominatingWalkSet}
	Let $\Pc' \subseteq \Pc$ a dominating walk-set and define 
		\[A'_p(h,t) \coloneqq \set{(q,\shiftN,\Delta) \in A_p(h,t) | q \in \Pc'}.\]
	Then, any $h^* \in S' \coloneqq \Set{h \in S | h_p \equiv 0 \text{ for all } p \in \Pc\setminus\Pc'}$ is an \SCDE{} \wrt $S'$ and $A'_p$ if and only if it is an \SCDE{} \wrt $S$ and $A_p$.
\end{lemma}

\begin{proof}
	The `if'-part follows directly from \Cref{lemma:ConditiononApForStrongerEquilibria}. For the `only if'-part, take any walk $p \in \Pc_i$, time $t \in \planningInterval$ and \addmDev{} $(q,\Delta) \in U_p(h^*,t)$. Since $\Pc'$ is a dominating walk set, there must be some $(q',\Delta') \in U_p(h^*,t)$ with $\Psi_{q'}(h^*,t+\Delta') \leq \Psi_q(h^*,t+\Delta)$ and $q' \in \Pc'$. Since $h^*$ is an \SCDE{} \wrt $A'_p$, this gives us
		\[\Psi_p(h^*,t) \leq \Psi_{q'}(h^*,t+\Delta') \leq \Psi_q(h^*,t+\Delta).\] 
	Thus, $h^*$ is an \SCDE{} \wrt $A_p$ as well.
\end{proof}

\begin{definition}\label{def:truncatedPsi}
	For any $M \in \IR,$ we define the \emph{truncated effective walk delay operator} $\truncated{\Psi}{M}$ of $\Psi$ by setting
		\[\truncated{\Psi}{M}_p(h,t) \coloneqq \min\set{\Psi_p(h,t), M} \text{ for all } p \in \Pc, h \in \Lambda(Q), t \in \planningInterval.\]
\end{definition}

\begin{lemma}\label{lemma:RestrictToTruncatedPsi}
	Given $M \in \IR$ such that for all $h \in S$, $t \in \planningInterval$ and $p \in \Pc$ with $U_p(h,t) \neq \emptyset$ there exists some $(q,\Delta) \in U_p(h,t)$ such that $\Psi_q(h,t+\Delta) < M$,
	then any walk inflow $h^*$ is an \SCDE{} with respect to $\Psi$ if and only if it is an \SCDE{} with respect to $\truncated{\Psi}{M}$.
\end{lemma}

\begin{proof}
	First, let $h^*$ be an \SCDE{} with respect to $\Psi$. Then, for any $p \in \Pc$, $i \in I$ and $(q,\Delta) \in U_p(h^*,t)$, we have
		\[\truncated{\Psi}{M}_p(h^*,t) = \min\set{\Psi_p(h^*,t),M} \leq \min\set{\Psi_q(h^*,t+\Delta),M} = \truncated{\Psi}{M}_p(h^*,t+\Delta).\]
	Therefore, $h^*$ is also an \SCDE{} with respect to $\truncated{\Psi}{M}$.
	
	Now, let $h^*$ be an \SCDE{} with respect to $\truncated{\Psi}{M}$. Again, take any $p \in \Pc$, $i \in I$ and $(q,\Delta) \in U_p(h^*,t)$. Then, by the \namecref{lemma:RestrictToTruncatedPsi}'s assumption, we have some $(q',\Delta') \in U_p(h^*,t)$ such that $\Psi_{q'}(h^*,t+\Delta') < M$ implying
		\[\truncated{\Psi}{M}_p(h^*,t) \leq \truncated{\Psi}{M}_{q'}(h^*,t+\Delta') = \Psi_{q'}(h^*,t+\Delta') < M\]
	and, therefore, $\truncated{\Psi}{M}_p(h^*,t) = \Psi_p(h^*,t)$. This, in turn, implies
		\[\Psi_p(h^*,t) = \truncated{\Psi}{M}_p(h^*,t) \leq \truncated{\Psi}{M}_q(h^*,t+\Delta) \leq \Psi_q(h^*,t+\Delta).\]
	Thus, $h^*$ is an \SCDE{} with respect to $\Psi$.		
\end{proof}

\subsection{Some Special Cases of \texorpdfstring{\SCDE}{SCDE}}\label{sec:SCDESpecialCases}

We first observe that the standard \emph{un}constrained dynamic equilibrium concepts are special cases of our model:

\begin{lemma}
	Assume that \ref{ass:EffectivePathDelayContinuous} holds. We define $S \coloneqq \Lambda(Q)$ and 
		\[A_p(h) \coloneqq \Set{(q,\shiftN,\Delta) | \begin{array}{l}
				q \in \Pc_i, \shiftN \in L_+^2(\planningInterval), \shiftN \leq h_p, \supp(\shiftN(\emptyarg-\Delta)) \subseteq \planningInterval \text{ and }\\
				h_q(t) + \shiftN(t-\Delta) \leq B_q \text{ for almost all } t \in \planningInterval
			\end{array}}\]
	for all $h \in S$, $p \in \Pc_i$ and $i \in I$. Then, any flow $h^* \in S$ is a \SCDE[full] in the sense of \eqref{eq:CDE} if and only if it is a dynamic equilibrium with fixed flow volumes and departure time choice in the sense of \eqref{eq:de-volume}.
\end{lemma}

\begin{proof}
	First, let $h^*$ be an equilibrium in the sense of \eqref{eq:de-volume} with corresponding values $\nu_i$. Now, take any time $t \in [t_0,t_f]$, shift $\Delta$ and walks $p,q \in \Pc_i$ such that $\Psi_p(h^*,t) > \Psi_q(h^*,t+\Delta)$. We then have to show that $(q,\Delta)$ is not an \addmDev, \ie that $(q,\Delta) \notin U_p(h^*,t)$. For this, we distinguish two cases: If we have $\Psi_q(h^*,t+\Delta) < \nu_i$, then \ref{ass:EffectivePathDelayContinuous} guarantees the existence of some neighbourhood $[t-\delta,t+\delta]$ of $t$ such that $\Psi_q(h^*,t'+\Delta) < \nu_i$ holds for all $t' \in [t-\delta,t+\delta]$. From \eqref{eq:de-volume} we then get that $h^*_q(t'+\Delta) \geq B_q$ for almost all these $t' \in [t-\delta,t+\delta]$ which, in turn, implies $(q,\Delta) \notin U_p(h^*,t)$. If, on the other hand, we have $\Psi_q(h^*,t+\Delta) \geq \nu_i$, then we have $\Psi_p(h^*,t) > \nu_i$ by our initial assumption. \Cref{ass:EffectivePathDelayContinuous} then again guarantees the existence of some $\delta > 0$ with $\Psi_p(h^*,t') > \nu_i$ for all $t' \in [t-\delta,t+\delta]$. From this \eqref{eq:de-volume} allows us to deduce $h^*_p(t')=0$ for almost all those $t'$ and, hence, $U_p(h^*,t) = \emptyset$. 
	
	Now, let $h^*$ be an equilibrium in the sense of \eqref{eq:CDE} and assume for contradiction that it is not an equilibrium in the sense of \eqref{eq:de-volume}. Then, there must exist some commodity $i \in I$ such that there exists no $\nu_i \in \IR$ satisfying \eqref{eq:de-volume}. In particular, we must have $Q_i > 0$, and therefore 
		\[\nu_i \coloneqq \esssup\set{\Psi_p(h^*,t) | p \in \Pc_i, t \in [t_0,t_f], h^*_p(t) > 0} > -\infty\]
	where $\esssup$ denotes the essential supremum, \ie 
		\[\esssup\Set{g(x) | x \in X} \coloneqq \inf\Set{b \in \IR | \abs{\set{x \in X | g(x) > b}} = 0}\]
	for any measurable function $g: X \to \IR$ where $\abs{.}$ denotes the measure on $X$. Since this $\nu_i$ cannot satisfy \eqref{eq:de-volume}, there must be a walk $q \in \Pc_i$, some $\varepsilon, \varepsilon' > 0$ and a set $J_q \subseteq [t_0,t_f]$ of positive measure such that we have 
		\[h^*_q(t) \leq B_q - \varepsilon \text{ and } \Psi_q(h^*,t) \leq \nu_i - \varepsilon' \text{ for all } t \in J_q.\]
	At the same time, the definition of $\nu_i$ guarantees the existence of some walk $p \in \Pc_i$ and set $J_p \subseteq [t_0,t_f]$ of positive measure such that
		\[h^*_p(t) > 0 \text{ and } \Psi_p(h^*,t) > \nu_i - \varepsilon' \text{ for all } t \in J_p.\]
	Now, there must exist some $\Delta \in \IR$ such that $(J_p + \Delta) \cap J_q$ has positive measure. Defining $J \coloneqq J_p \cap (J_q-\Delta)$ gives us a set containing at least one $t \in J$ such that for all $n \in \IN$ the set $J_n \coloneqq [t-\frac{1}{n},t+\frac{1}{n}] \cap J$ has positive measure. Defining 
		\[\shiftN_n \coloneqq \min\Set{\tfrac{\varepsilon}{n}\cdot\CharF[J_n],h_p}: \planningInterval \to \IR, t' \mapsto \begin{cases}\min\set{\tfrac{\varepsilon}{n},h_p(t')}, &\text{if } t' \in J_n\\0,&\text{else}\end{cases}\]
	for all $n \in \IN$ then gives us a sequence of $L^2$-functions satisfying
	\begin{itemize}
		\item $\lim_n \esssup \shiftN_n = \lim_n \tfrac{\varepsilon}{n} = 0$,
		\item $\lim_n \sup\set{\theta \in \planningInterval | \supp(\shiftN_n) \cap \planningInterval \text{ has measure zero}}$\\\null$\quad= \lim_n \sup\set{\theta \in \planningInterval | J_n \cap \planningInterval \text{ has measure zero}} \geq \lim_n (t-\tfrac{1}{n}) = t$ and, analogously, $\lim_n \inf\set{\theta \in \planningInterval | \supp(\shiftN_n) \cap [\theta,\tEnd] \text{ has measure zero}} \leq t$,
		\item $\Psi_p(h^*,t) > \nu_i - \varepsilon' \geq \Psi_q(h^*,t+\Delta)$ since $t \in J \subseteq J_p$ and $t + \Delta \in J + \Delta \subseteq J_q$,
		\item $\int_{J_n}\shiftN(t')\diff t' > 0$ for all $n \in \IN$ as $J_n \subseteq J_p$ has positive measure and
		\item $(q,\shiftN_n,\Delta) \in A_{p}(h^*)$ for all $n \in \IN$ since $\shiftN_n \leq h_p$, $\supp(\shiftN_n(\emptyarg-\Delta))= J_n+\Delta \subseteq J_q \subseteq \planningInterval$ and $h^*_q (t') + \shiftN_n(t'-\Delta) \leq h^*_q (t') + \varepsilon \leq B_q$ for all $t' \in J_n + \Delta \subseteq J_q$.
	\end{itemize}
	By \Cref{lemma:NegativeCharacterization} this implies that $h^*$ is not a \SCDE[full] -- a contradiction to our initial assumption.
\end{proof}

\begin{lemma}
	Assume that \ref{ass:EffectivePathDelayContinuous} holds. We choose $S \coloneqq \Lambda(r)$ and 
		\[A_p(h) \coloneqq \set{(q,\shiftN,0) | q \in \Pc_i, \shiftN \in L_+^2(\planningInterval), \shiftN \leq h_p}\]
	for all $h \in S$, $p \in \Pc_i$ and $i \in I$. Then, any flow $h^* \in S$ is a \SCDE[full] in the sense of \eqref{eq:CDE} if and only if it is a dynamic equilibrium with fixed inflow rate in the sense of \eqref{eq:de-rate}.
\end{lemma}

\begin{proof}
	First, let $h^*$ be an equilibrium in the sense of \eqref{eq:de-rate}. Then, for any point $t \in [t_0,t_f]$ and walks $p,q \in \Pc_i$ with $\Psi_p(h^*,t) > \Psi_q(h^*,t)$ \cref{ass:EffectivePathDelayContinuous} guarantees the existence of some neighbourhood $[t-\delta,t+\delta]$ of $t$ with $\Psi_p(h^*,t') > \Psi_q(h^*,t')$ for all $t' \in [t-\delta,t+\delta]$. \Cref{eq:de-rate} then implies that $h^*_p(t') = 0$ for almost all $t'$ from this interval and, thus, $U_{p}(h^*,t) = \emptyset$. In particular, we have $q \notin U_{p}(h^*,t)$ which shows that \eqref{eq:CDE} is satisfied by $h^*$.
	
	Now, let $h^*$ be an equilibrium in the sense of \eqref{eq:CDE}. Then, for every time $t \in [t_0,t_f]$ and each pair of walks $p,q$ with $\Psi_p(h^*,t) > \Psi_q(h^*,t)$, we have $q \notin U_{p}(h^*,t)$. As we are in the unconstrained case, this can only be because there exists some $\delta_t > 0$ such we have $h^*_p(t') = 0$ for almost all $t' \in [t-\delta_t,t+\delta_t]$. Since this is true for all such times $t$, we clearly have
		\[\set{t \in [t_0,t_f] | \Psi_p(h^*,t) > \Psi_q(h^*,t)} \subseteq \bigcup_{t \in [t_0,t_f]:\Psi_p(h^*,t) > \Psi_q(h^*,t)}[t-\delta_t,t+\delta_t].\]
	But then, there also exists a countable such covering, \ie a countable set 
		\[K \subseteq \set{t \in [t_0,t_f] | \Psi_p(h^*,t) > \Psi_q(h^*,t)}\]
	such that 
		\[\set{t \in [t_0,t_f] | \Psi_p(h^*,t) > \Psi_q(h^*,t)} \subseteq \bigcup_{t \in K}[t-\delta_t,t+\delta_t].\]
	Therefore, for almost all $t \in \set{t \in [t_0,t_f] | \Psi_p(h^*,t) > \Psi_q(h^*,t)}$ we have $h^*_p(t') = 0$ which shows that $h^*$ satisfies \eqref{eq:de-rate}.
\end{proof}

We conclude this \namecref{sec:SCDESpecialCases} by introducing a first type of \SCDE{} with actual side-constraints: In general, our framework requires us to specify two objects for such a definition: The \setS~$S$ and the \addmEpsDev s~$A_p$. We will make use of this flexibility later on (in particular in \Cref{sec:SCviaNL}) but we can also just take any \setS{} $S \subseteq \Lambda(Q)$ and then derive \addmEpsDev s from it as follows: We say that a potential \epsDev{} is admissible if and only if it leads to another feasible flow, \ie 
	\begin{align}\label{eq:ApglobalSCDE}
		A_{p}(h) \coloneqq \set{(q,\shiftN,\Delta) | H_{p \to q}(h,\shiftN,\Delta) \in S}
	\end{align}
or, equivalently,
	\[M_i(h) \coloneqq \set{h' \in L^2 (\planningInterval)^{\Pc} | \exists p, q, i, \shiftN,\Delta: h' = H_{p \to q}(h,\shiftN,\Delta) \in S}.\]
This imposes a global feasibility constraint on the possible deviations: That is, particles are only allowed to deviate if the resulting alternative flow is feasible again for \emph{all} particles (not just the ones deviating). This is a generalization of the capacitated dynamic equilibria defined in \cite{GHP22}. Here, we will call these types of side-constrained dynamic equilibria \emph{\globalSCDE[fulls]{}}.

\begin{definition}\label{def:strongCDE}
	A \emph{\globalSCDE[full] (\globalSCDE)} with respect to some set $S \subseteq \Lambda(Q)$ is a side-constrained dynamic equilibrium \wrt $S$ and \addmEpsDev s $A_{p}(h)$ defined by \eqref{eq:ApglobalSCDE}.
\end{definition}


\section{Characterization of \texorpdfstring{\SCDE{}}{SCDE} via (Quasi-)Variational Inequalities}\label{sec:characterization}

Similar to the characterizations of unconstrained dynamic equilibria with variational inequalities (\cf \Cref{thm:VICharOfDE}) we now want to characterize side-constrained dynamic equilibria using \emph{quasi}-variational inequalities (QVIs). This is not only of theoretical interest but there are also algorithms solving  them, see \eg \citefull{KanzowSteckQVI}, \citefull{Shehu20} and references therein. Note, however, that the convergence guarantees given by \citeauthor*{Shehu20} require a certain strong monotonicity property for the mapping $h \mapsto \Psi(h,.)$ which, in general, is not satisfied for dynamic flows whereas \citeauthor{KanzowSteckQVI} only need weak-strong continuity but instead require some convexity assumptions on the sets $S$ and $M(h)$.

\subsection{Quasi-Variational Inequality}

Let us define the following tangent cone at $h \in S$ with respect to $A_p$:
\begin{equation}
	T(A_p,h):=\Set{v\in L^2(\planningInterval)^{\Pc} | \begin{array}{l}
			\exists (h^n)_{n \in \IN} \subset M(h), (t_n)_{n \in \IN} \subset \IR_{>0}: \\ 
			\lim_{n\rightarrow\infty} t_n=0,  \lim_{n\rightarrow\infty} \frac{h^n-h}{t_n}=v
		\end{array}}.
\end{equation}
where $M(h) \coloneqq \bigcup_{i \in I}M_i(h)$ is the set of all walk inflows reachable by a single \addmEpsDev{} from $h$. Intuitively, this set $T(A_p,h)$ contains all directions (within the space $L^2(\planningInterval)^\Pc$) in which, according to the sets of \addmEpsDev s, one can move some (infinitesimally) small step starting from~$h$.

We introduce the following assumptions on the correspondences $A_p$:
\begin{definition}
	The correspondences $(A_p)_{p \in \Pc}$ are \emph{closed under rate-reduction} at $h \in S$, if for all $i \in I$, $q,p \in \Pc_i$, $\Delta \in \IR$, $\lambda \in [0,1]$ and $\shiftN \in L_+^2(\planningInterval)$, we have:
	\[(q,\shiftN,\Delta) \in A_p(h) \implies (q,\lambda\shiftN,\Delta) \in A_{p}(h).\]
\end{definition}

\begin{definition}
	The correspondences $(A_p)_{p \in \Pc}$ are \emph{closed under time-restriction} at $h \in S$, if for all $i \in I$, $q,p \in \Pc_i$, $\shiftN \in L_+^2(\planningInterval)$ and measurable $J \subseteq \planningInterval$, we have:
	\[(q,\shiftN,\Delta) \in A_{p}(h) \implies (q,\shiftN\cdot\CharF[J],\Delta) \in A_{p}(h).\]
	Here, $\CharF[J]$ denotes the characteristic function of the set~$J$, \ie we have $\CharF[J](t)=1$ if $t \in J$ and $\CharF[J](t)=0$ otherwise.
\end{definition}

\begin{enumerate}[label=(A\arabic*),resume=Assumptions]
	\item The \addmEpsDev s are closed under rate-reduction at all $h \in S$.\label[asmpt]{ass:closedSpace}
	\item The \addmEpsDev s are closed under time-restriction at all $h \in S$.\label[asmpt]{ass:closedTime}
\end{enumerate}

The first assumption states that whenever flow is allowed to deviate at a certain rate, it is also allowed to deviate at any lower rate. The second assumption states that whenever flow is allowed to deviate during a certain neighbourhood, the same deviation is also allowed during any subset of that neighbourhood.
It will turn out that several well-motivated equilibrium concepts fulfil both
assumptions (see \Cref{obs:CDEcharbyQVI} for more details).

\begin{obs}\label{obs:ClosedSpaceTimeConvex}
	If the trivial deviation is always admissible (\ie $(p,0,0) \in A_p(h)$), then \cref{ass:closedSpace} is a weaker assumption than convexity, \ie if $M_i(h)$ is convex (at $h$), then \ref{ass:closedSpace} holds. This is true because for any $\lambda \in [0,1]$, the walk inflow $H_{p\to q}(h,\lambda\shiftN,\Delta)$ is a convex combination of $h=H_{p\to p}(h,0,0)$ and $H_{p\to q}(h,\shiftN,\Delta)$:
	\[H_{p\to q}(h,\lambda\shiftN,\Delta) = (1-\lambda)\cdot H_{p\to p}(h,0,0) + \lambda\cdot H_{p\to q}(h,\shiftN,\Delta).\]
	\Cref{ass:closedTime}, on the other hand, is independent of convexity, \ie there exist sets $M_i(h)$ which are convex but do not satisfy \ref{ass:closedTime} and sets $M_i(h)$ which satisfy \ref{ass:closedTime} but are not convex. 
	E.g., consider a network with a single commodity with a fixed inflow rate $r = \CharF[{[0,2]}]$, two nodes $s$ and $t$ and two parallel edges $e_1$ and $e_2$ connecting $s$ and $t$ (these are then also the only $s$,$t$-paths $p$ and $q$). We define the following sets $S \subseteq \Lambda(r)$ with corresponding sets $A_{p}(h) \coloneqq \set{(q,\shiftN,0) | H_{p \to q}(h,\shiftN,0) \in S}$:
	\begin{itemize}
		\item $S_1$ contains all walk inflows $h \in \Lambda(r)$ which for (almost) every point in $[0,2]$ sent all flow in exactly one of the two paths. Clearly, the resulting \addmEpsDev s satisfy \ref{ass:closedTime}. The set $S$ is, however, not convex as it contains the walk inflow $h^1$ which sends all flow into $p$ as well as the walk inflow $h^2$ which sends all flow into $q$ but not any of their non-trivial convex combinations.
		\item $S_2 \coloneqq \mathrm{conv}(h^1,h^2)$ is a convex set, but the corresponding sets $A_p(h)$ do not satisfy~\ref{ass:closedTime} since $H_{p\to q}(h^1,\CharF[[0,2]],0) = h^2 \in S_2$ is a feasible walk inflow while the flow $H_{p\to q}(h^1,\CharF[[0,1]],0) \notin S_2$ is not.
	\end{itemize}
\end{obs}

We now consider the following quasi-variational inequality:

\begin{equation}\label{eq:QVI-SCDE}\tag{\ensuremath{\QVI(\Psi,S,A_p)}}
	\begin{aligned}
		\text{Find }h^* \in  S  \text{ such that:}&\\
		\scalar{\Psi(h^*)}{v} &\geq 0 \text{ for all } v \in T(A_p,h^*).\end{aligned}
\end{equation}

\begin{theorem}\label{thm:VI-fixed-inflow:sufficient}
	Let $h^* \in S$ be given such that $S$ and $(A_p)_{p \in \Pc}$ satisfy \ref{ass:closedSpace} at $h^*$ and $\Psi$ such that \ref{ass:EffectivePathDelayContinuous} holds. If $h^*$ is a solution to the quasi-variational inequality \eqref{eq:QVI-SCDE} then it is an \SCDE{} \wrt $S$ and $(A_p)_{p \in \Pc}$.
\end{theorem}

To prove this \namecref{thm:VI-fixed-inflow:sufficient} we observe that whenever some flow $h^*$ is not an \SCDE{}, then, due to the continuity of $\Psi(h,\emptyarg)$, there must also be a small \addmEpsDev{} such that \emph{all} particles involved in it would improve by deviating this way. Closedness under rate-reduction then ensures that any such improving \epsDev{} induces an improving direction in the tangent cone at $h^*$, which shows that $h^*$ cannot be a solution to \eqref{eq:QVI-SCDE}.

\begin{proof}
	Let $h^* \in  S$ be a solution to \eqref{eq:QVI-SCDE} and assume that  $h^*$ is not an \SCDE. Then there exists a commodity $i$, walks $p,q \in \Pc_i$, a shift $\Delta$ and some time $t \in \planningInterval$ such that $\Psi_p(h^*,t) > \Psi_q(h^*,t+\Delta)$ and $(q,\Delta) \in U_{p}(h^*,t)$. Since $\Psi_q(h^*,.)$ and $\Psi_p(h^*,.)$ are continuous (by \ref{ass:EffectivePathDelayContinuous}) there must be some  constants $\delta, \varepsilon > 0$ such that $\Psi_p(h^*,t') - \Psi_q(h^*,t'+\Delta) \geq \varepsilon$ holds for all $t' \in [t-\delta,t+\delta]$. From $(q,\Delta) \in U_{p}(h^*,t)$ we then get some function $\shiftN \in L^2_+(\planningInterval)$ with $\shiftN \leq h_p$, $\int_{\tStart}^{\tEnd}\shiftN(\zeta)\diff\zeta >0$, $\supp(\shiftN) \subseteq [t-\delta,t+\delta]$ and $\bar h \coloneqq H_{p\to q}(h^*,\shiftN,\Delta) \in M_i(h^*)$. Because of \ref{ass:closedSpace}, we now have $\bar h- h^* = \lim_{\lambda \searrow 0}\frac{H_{p\to q}(h^*,\lambda\shiftN,\Delta)-h^*}{\lambda} \in T(A_p,h^*)$. But, at the same time we have
	\begin{align*}
		&\scalar{\Psi(h^*)}{\bar h-h^*} 
		= \sum_{w \in \Pc}\int_{\tStart}^{\tEnd} \Psi_w(h^*(t'))\cdot\left(\bar h_w(t')-h^*_w(t')\right)\diff t'\\
		&\quad\,=\int_{t-\delta}^{t+\delta}\Psi_p(h^*,t')\cdot\left(\bar h_p(t')-h_p^*(t')\right)\diff t' + \int_{t+\Delta-\delta}^{t+\Delta+\delta}\Psi_q(h^*,t')\cdot\left(\bar h_q(t')-h_q^*(t')\right)\diff t' \\
		&\quad\,=\int_{t-\delta}^{t+\delta}\Psi_p(h^*,t')\cdot\left(\bar h_p(t')-h_p^*(t')\right) + \Psi_q(h^*,t'+\Delta)\cdot\left(\bar h_q(t'+\Delta)-h_q^*(t'+\Delta)\right)\diff t' \\
		&\quad\,=\int_{t-\delta}^{t+\delta}\left(\Psi_q(h^*,t'+\Delta)-\Psi_p(h^*,t')\right)\cdot\shiftN(t')\diff t' \\
		&\quad\,\leq -\varepsilon \cdot \int_{t-\delta}^{t+\delta}\shiftN(t')\diff t' < 0,
	\end{align*}
	which is a contradiction to $h^*$ being a solution to \eqref{eq:QVI-SCDE}.
\end{proof}

\begin{remark}
	To see why we need \cref{ass:closedSpace} in the statement of \Cref{thm:VI-fixed-inflow:sufficient}, consider again the \setS{} $S_1$ from \Cref{obs:ClosedSpaceTimeConvex} and define 
		\[A_{p}(h^1) \coloneqq \Set{(q,\tfrac{\abs{J}}{2}\cdot\CharF[J],0) | J \subseteq [0,2] \text{ measurable}}.\]
	Then, we have $T(A_p,h^1) = \set{0}$ and, thus, $h^1$ is a solution to \eqref{eq:QVI-SCDE} (regardless of the choice of \effWalkDelay{} operators) while it is not necessarily a \SCDE{} as we have $U_{p}(h^1,t) = \set{(p,0),(q,0)}$ for all $t \in [0,2]$ here.
\end{remark}

\begin{theorem}\label{thm:VI-fixed-inflow:nessecary}
	Let $h^* \in S$ be given such that $S$ and $(A_p)_{p \in \Pc}$ satisfy \labelcref{ass:closedTime,ass:closedSpace} at $h^*$. If $h^*$ is an \SCDE{} \wrt $S$ and $(A_p)_{p \in \Pc}$ then it is also a solution to~\eqref{eq:QVI-SCDE}.
\end{theorem}

The proof of this \namecref{thm:VI-fixed-inflow:nessecary} is essentially the reverse of the previous one's: That is, we use an improving direction to deduce the existence of an improving \addmEpsDev{} which, due to closedness under rate-reduction and time-restriction, gives us an improving \addmDev.

\begin{proof}
	We show this by contradiction. So, let $h^*$ be an \SCDE{} and assume that there is some $v\in T(A_p,h^*)$ with $\scalar{\Psi(h^*)}{v} < 0$, \ie we have sequences $(h^n) \subseteq M(h^*)$ and $(t_n) \subseteq \IR_{>0}$ with $\scalar{\Psi(h^*)}{\lim_n\tfrac{h^n-h^*}{t_n}}<0$. By continuity of $\scalar{\emptyarg}{\emptyarg}$, this implies
	$\scalar{\Psi(h^*)}{\tfrac{h^n-h^*}{t_n}}<0$ or, equivalently, $\scalar{\Psi(h^*)}{h^n-h^*}<0$ for large enough~$n$.
	Rewriting and using that $h^n$ is of the form $h^n = H_{p \to q}(h^*,\shiftN,\Delta)$ for some $(q,\shiftN,\Delta) \in A_{p}(h^*)$, yields
	\begin{align*}
		&0 > \scalar{\Psi(h^*)}{h^n-h^*}
		= \!\int_{\tStart}^{\tEnd} \Psi_p(h^*,t)  (h_p^n(t)-h_p^*(t)) \diff t + \!\int_{\tStart}^{\tEnd}\Psi_q(h^*,t) (h_q^*(t)-h_q^n(t)) \diff t  \\
		&\quad\quad= \int_{\tStart}^{\tEnd} \Psi_p(h^*,t)  (h_p^n(t)-h_p^*(t)) + \Psi_q(h^*,t+\Delta) (h_q^*(t+\Delta)-h_q^n(t+\Delta)) \diff t  \\
		&\quad\quad= \int_{\tStart}^{\tEnd} \left(\Psi_q(h^*,t+\Delta)- \Psi_p(h^*,t)\right) \cdot \shiftN(t) \diff t.
	\end{align*}
	This implies that there is some subset $J \subseteq \supp(\shiftN)$ of positive measure with 
	\[\left(\Psi_q(h^*,t+\Delta)- \Psi_p(h^*,t)\right) \cdot \shiftN(t) < 0 \text{ for all } t \in J.\]
	Since $\shiftN$ is non-negative, this implies $\Psi_q(h^*,t+\Delta) < \Psi_p(h^*,t)$ as well as $\shiftN(t)>0$ for 
	all $t \in J$. As $J$ has positive measure, it must contain a point $t \in J$ such that the intersection of any neighbourhood of $t$ with $J$ also has positive measure. Defining $J_n \coloneqq [t-\frac{1}{n},t+\frac{1}{n}] \cap J$ then results in a sequence of subsets of $J$ of positive measure satisfying $\liminf_n J_n = \limsup_n J_n = t$. Since we have $H_{p\to q}(h^*,\shiftN,\Delta) \in M_i(h^*)$, \cref{ass:closedTime,ass:closedSpace} ensure that $H_{p \to q}(h^*,\frac{1}{n}\shiftN\cdot\CharF[J_n],\Delta) \in M(h^*)$ holds for all $n \in \INs$ as well. Furthermore, we have $\int_{\tStart}^{\tEnd}\frac{1}{n}\shiftN(t')\cdot\CharF[J_n](t')\diff t' = \frac{1}{n}\int_{J_n}\shiftN(t')\diff t' > 0$ for all $n$ since $\shiftN(t') > 0$ for all $t' \in J$ and $J_n \subseteq J$ has positive measure. Altogether, this shows that $h^*$ is not an \SCDE{} by \Cref{lemma:NegativeCharacterization}.
\end{proof}

\begin{remark}
	To see why we need \cref{ass:closedTime} in the statement of \Cref{thm:VI-fixed-inflow:nessecary}, consider the \setS{} $S_2$ from \Cref{obs:ClosedSpaceTimeConvex} and define the set of \addmEpsDev s 
		\[A_{p}(h^1) \coloneqq \set{(q,\varepsilon\cdot\CharF[[0,2]],0) | \varepsilon \geq 0} \cup \set{(p,\varepsilon\cdot\CharF[J],0) | J \subseteq [0,2] \text{ measurable}, \varepsilon \geq 0}\]
	as in the \namecref{obs:ClosedSpaceTimeConvex}. Clearly, this set $A_{p}(h^1)$  is not closed under time restriction. Furthermore, we have $U_{p}(h^1,t) = \set{(p,0)}$ for all $t \in [0,2]$ and, therefore, $h^1$ is a \SCDE. However, it is also easy to see that for certain choices of the \effWalkDelay{} operators $\Psi_p$ and $\Psi_q$ the flow $h^1$ is not a solution to the quasi-variational inequality \eqref{eq:QVI-SCDE} (\eg choose constant flow independent delays $\Psi_p \equiv 2$ and $\Psi_q \equiv 1$).
\end{remark}

\subsection{Variational Inequality}

Quasi-variational inequalities may be much harder to solve compared to standard variational inequalities since the feasible search space depends on the  solution itself. However, under two additional assumptions, we can also use the following variational inequality to characterize \SCDE{}:s
\begin{equation}\label{eq:VI-SCDE}\tag{\ensuremath{\VI(\Psi,S)}}
	\begin{aligned}
		\text{Find }h^* \in  S  \text{ such that:}&\\
		\scalar{\Psi(h^*)}{h-h^*} &\geq 0 \text{ for all }h \in S.
	\end{aligned}
\end{equation}
Note, that this variational inequality is of exactly the form of~\eqref{eq:ZhongVI} used by \citefulls{zhong11} to define their version of side-constrained dynamic equilibria.

For the sufficiency part, we need the following additional assumption stating that small enough \addmEpsDev s from a feasible flow lead to another feasible flow:
\begin{enumerate}[label=(A\arabic*),resume=Assumptions]
	\item For any $h \in S$ there exists some neighbourhood $V_h$ of $h$ such that $M(h) \cap V_h \subseteq S$.\label[asmpt]{ass:addmDevLeadToFeasFlow}
\end{enumerate}
Note that this assumption is trivially satisfied for \globalSCDE{} as, in this case, by definition any \addmEpsDev{} leads to a feasible flow.

\begin{theorem}\label{thm:VI-sufficient}
	Take any constraint set $S$ and \addmEpsDev s $(A_p)_{p \in \Pc}$ satisfying \labelcref{ass:closedSpace,ass:addmDevLeadToFeasFlow}. Furthermore, assume that $\Psi$ satisfies \ref{ass:EffectivePathDelayContinuous}.
	Then, any solution $h^*$ to the variational inequality~\eqref{eq:VI-SCDE} is an \SCDE{} \wrt $S$ and $(A_p)_{p \in \Pc}$.
\end{theorem}

\begin{proof}
	Let $h^* \in S$ be a solution to the variational inequality~\eqref{eq:VI-SCDE}. We claim that $h^*$ is then also a solution to the quasi-variational inequality~\eqref{eq:QVI-SCDE}. In order to show this, take any $v \in T(S,A_p,h^*)$. Then there exist sequences of $(h^n)_{n \in \IN} \subset M(h^*)$ and $(t_n)_{n \in \IN} \subset \IR_{>0}$ such that $\lim_{n\rightarrow\infty}\frac{h^n-h^*}{t_n} = v$. Using the continuity of $\scalar{.}{.}$, we get
	\begin{align*}
		\scalar{\Psi(h^*)}{v} = \scalar{\Psi(h^*)}{ \lim_{n\rightarrow\infty}\tfrac{h^n-h^*}{t_n}} =  \lim_{n\rightarrow\infty}\tfrac{1}{t_n}\scalar{\Psi(h^*)}{h^n-h^*} \geq 0.
	\end{align*}
	The inequality at the end holds since $h^n \in M(h)$ implies $h^n \in S$ for large enough $n$ due to \ref{ass:addmDevLeadToFeasFlow} and since $h^*$ is a solution to \eqref{eq:VI-SCDE}. Thus, we can now apply \Cref{thm:VI-fixed-inflow:sufficient} to conclude that $h^*$ is indeed an \SCDE.
\end{proof}

Now, whenever in addition to the assumptions of \Cref{thm:VI-sufficient} we know that \eqref{eq:VI-SCDE} has a solution, we get a first existence result for \SCDE{} generalizing the existence theorem for unconstrained dynamic equilibria (\Cref{thm:ExistenceUnconstrained}):

\begin{corollary}\label{cor:ExistenceViaVI}
	Let $S$ be a convex, non-empty, closed and bounded set and $A_p$ satisfying \labelcref{ass:closedSpace,ass:addmDevLeadToFeasFlow}. Furthermore, assume that \ref{ass:FinitelyManyWalks} holds and $\Psi$ satisfies \labelcref{ass:PsiBounded,ass:EffectivePathDelayContinuous,ass:PsiWScont}. Then, there exists an \SCDE{} \wrt $S$ and $A_p$.
\end{corollary}

\begin{proof}
	By \Cref{thm:Lions} the variational inequality \eqref{eq:VI-SCDE} has a solution, which, by \Cref{thm:VI-sufficient}, is an \SCDE.
\end{proof}

\begin{remark}
	Note that for \globalSCDE{} with convex feasibility set~$S$ both \cref{ass:addmDevLeadToFeasFlow,ass:closedSpace} hold automatically. In particular, the existence result for capacitated dynamic equilibria in \cite[Theorem 6]{GHP22} is a special case of the above \namecref{cor:ExistenceViaVI}.
	
	On the other hand, the counterexamples from \Cref{sec:counter} shows not only that $S$ defined by volume-constraints can be non-convex, but also that \cref{ass:closedSpace} does not necessarily hold for \globalSCDE{} with such feasibility sets. Thus, even in cases where the variational inequality~\eqref{eq:VI-SCDE} has a solution, it is not clear whether such a solution is also an \SCDE.
\end{remark}

For the necessity part we only consider the case of fixed network inflow rates and require the following additional property of the constraint set $S$ and the \addmEpsDev s~$A_p$.

\begin{definition}\label{def:elementary}
	The set $S \subseteq \Lambda(r)$ is called \emph{closed with respect to elementary directions}, 
	if for all $h,h'\in S$ the following holds true:
	Whenever there exist $i\in I, p,q\in \Pc_i$ and $J \subset \planningInterval$ with positive measure such that $h_p(t) - h'_p(t) > 0$ and $h'_q(t) - h_q(t) > 0$
	for all $t\in J$, we have $(q,0) \in U_{p}(h,t)$ for some $t \in J$.
\end{definition}

This property states that in any walk inflow $h \in S$ particles are allowed to switch from some walk $p$ to another walk $q$ if there exists another feasible walk inflow $h'$ which has a lower inflow rate into $p$ and (during the same time) a higher inflow rate into $q$. 

\begin{theorem}\label{thm:VI-necessary}
	Take any constraint set $S \subseteq \Lambda(r)$ with fixed inflow rates and \addmEpsDev s $(A_p)_{p \in \Pc}$ such that $S$ is closed with respect to elementary directions. Then, every \SCDE{} $h^*\in S$ \wrt $S$ and $(A_p)_{p \in \Pc}$ is a solution to the variational inequality~\eqref{eq:VI-SCDE}.
\end{theorem}

\begin{proof}
	Let $h^* \in S$ be an \SCDE. We have to show that $h^*$ is a solution to~\eqref{eq:VI-SCDE}, \ie that for any $h\in S$ we have $\scalar{\Psi(h^*)}{h-h^*} \geq 0$. So, assume for contradiction that this is not the case, \ie there exists some $h \in S$ with $\scalar{\Psi(h^*)}{h-h^*} < 0$.
	
	We now define for any pair of walks $p$ and $q$ a function $g_{p \to q}: \planningInterval \to \IR_{\geq 0}$ by 
	\begin{align*}
		g_{p \to q}(t) \coloneqq \begin{cases}
			\left(h_p(t)-h^*_p(t)\right)\cdot\frac{h^*_q(t)-h_q(t)}{\sum_{\substack{q' \in \Pc:\\ h_{q'}(t) < h^*_{q'}(t)}}\left(h^*_{q'}(t)-h_{q'}(t)\right)}, &\text{ if } h_p(t) > h^*_p(t), h_q(t) < h^*_q(t) \\
			0,												&\text{ else }.
		\end{cases}
	\end{align*}
	First, we observe that these functions are non-negative, bounded, well-defined and measurable. Now we define for each of these functions a corresponding function $h_{p \to q} \in L^2(\planningInterval)^\Pc$ by setting
	\begin{equation}
		(h_{p\rightarrow q})_w(t) \coloneqq 
		\begin{cases} 
			0, 								&\text{ if }w\not\in \{p,q\}\\
			g_{p\rightarrow q}(t)\in \R_+, 	&\text{ if }w=p\\
			-g_{p\rightarrow q}(t)\in \R_-, &\text{ if }w=q.
		\end{cases}
	\end{equation}
	We now claim that these functions add up to precisely the difference between $h$ and $h^*$:
	
	\begin{claim}\label{claim:TranshipmentSolution}
		We have $h-h^* = \sum_{p,q \in \Pc}h_{p\to q}$.
	\end{claim}
	
	\begin{proofClaim}
		Let $w \in \Pc$ be any walk and $t \in \planningInterval$ be any time. Then, we distinguish three cases:
		\begin{proofbycases}
			\caseitem{$h_w(t) = h^*_w(t)$} In this case, we have $g_{p \to w}(t) = 0 = g_{w \to q}(0)$ for all $p, q \in \Pc$ and, thus,
			\begin{align*}
				\left(\sum_{p,q \in \Pc}h_{p\to q}\right)_w 
				&= \left(\sum_{p \in \Pc}h_{p\to w}\right)_w + \left(\sum_{q \in \Pc}h_{w\to q}\right)_w \\
				&= \sum_{p \in \Pc}-g_{p\to w} + \sum_{q \in \Pc}g_{w\to q} = 0 = (h_w(t) - h^*_w(t)).\\
			\end{align*}
			
			\caseitem{$h_w(t) > h^*_w(t)$} In this case, we have $g_{p \to w}(t) = 0$ for all $p \in \Pc$ and $g_{w \to q}(t) = 0$ for all $q \in \Pc$ with $h_q(t) \geq h^*_q(t)$. For all other $q$, we have $g_{w \to q}(t) = \left(h_w(t)-h^*_w(t)\right)\cdot\frac{h^*_q(t)-h_q(t)}{\sum_{q' \in \Pc: h_{q'}(t) < h^*_{q'}(t)}\left(h^*_{q'}(t)-h_{q'}(t)\right)}$ and, thus,
			\begin{align*}
				\left(\sum_{p,q \in \Pc}h_{p\to q}\right)_w 
				&= \sum_{p \in \Pc}-g_{p\to w} + \sum_{q \in \Pc}g_{w\to q} \\
				&= 0 + \sum_{\substack{q \in \Pc:\\ h_q(t) < h^*_q(t)}}\left(h_w(t)-h^*_w(t)\right)\cdot\frac{h^*_q(t)-h_q(t)}{\sum_{q' \in \Pc: h_{q'}(t) < h^*_{q'}(t)}\left(h^*_{q'}(t)-h_{q'}(t)\right)} \\
				&= h_w(t)-h^*_w(t).
			\end{align*}
			Note, that we need fixed inflow rates (\ie $S \subseteq \Lambda(r)$) here to ensure that there exists at least one walk $q$ with $h_q(t)-h^*_q(t)$. This is used in the last equality.
			
			\caseitem{$h_w(t) < h^*_w(t)$} In this case, we have $g_{w \to q}(t) = 0$ for all $q \in \Pc$ and $g_{p \to w}(t) = 0$ for all $p \in \Pc$ with $h_p(t) \leq h^*_p(t)$. For all other $p$, we have $g_{p \to w}(t) = \left(h_p(t)-h^*_p(t)\right)\cdot\frac{h^*_w(t)-h_w(t)}{\sum_{q' \in \Pc: h_{q'}(t) < h^*_{q'}(t)}\left(h^*_{q'}(t)-h_{q'}(t)\right)}$ and, thus,
			\begin{align*}
				\left(\sum_{p,q \in \Pc}h_{p\to q}\right)_w 
				&= \sum_{p \in \Pc}-g_{p\to w} + \sum_{q \in \Pc}g_{w\to q}& \\
				&= -\sum_{\mathclap{\substack{p \in \Pc:\\ h_p(t) > h^*_p(t)}}}\left(h_p(t)-h^*_p(t)\right)\cdot\frac{h^*_w(t)-h_w(t)}{\sum_{q' \in \Pc: h_{q'}(t) < h^*_{q'}(t)}\left(h^*_{q'}(t)-h_{q'}(t)\right)} + 0& \\
				&= -\left(h^*_w(t)-h_w(t)\right). &\qedhere
			\end{align*}
		\end{proofbycases}
	\end{proofClaim}
	
	Now, using our initial assumption $\scalar{\Psi(h^*)}{h-h^*} < 0$, we get
	\begin{align*}
		0 &> \scalar{\Psi(h^*)}{h-h^*} 
		\overset{\text{\Cref{claim:TranshipmentSolution}}}{=} \scalar{\Psi(h^*)}{\sum_{p,q\in \Pc}h_{p\rightarrow q}}\\
		&=\sum_{p,q\in \Pc} \scalar{\Psi(h^*)}{h_{p\rightarrow q}}
		=\sum_{p,q\in \Pc}\sum_{w \in \Pc} \int_{\tStart}^{\tEnd} \Psi_w(h^*,t)\cdot (h_{p\rightarrow q})_w(t) dt \\
		&= \sum_{p,q\in \Pc} \int_{\tStart}^{\tEnd} \Psi_p(h^*,t)\cdot g_{p\rightarrow q}(t) dt + \int_{\tStart}^{\tEnd} \Psi_q(h^*,t)\cdot (-g_{p\rightarrow q}(t)) dt \\
		&= \sum_{p,q\in \Pc} \int_{\tStart}^{\tEnd}g_{p\rightarrow q}(t) (\Psi_p(h^*,t)-\Psi_q(h^*,t))\diff t.
	\end{align*}
	Since  $g_{p\rightarrow q}\geq 0$, there must be a subset $J\subset \planningInterval$ of positive measure
	with $g_{p\rightarrow q}(t)>0 $ and $\Psi_p(h^*,t)-\Psi_q(h^*,t)<0$ for all $t\in J$. 	
	As $g_{p\rightarrow q}(t)>0 $ implies $h^*_p(t)-h_p(t)<0$ and $h^*_q(t)-h_q(t)>0$ and $S$ is  closed with respect to elementary directions, we get $(p,0) \in U_{q}(h^*,t)$ for some time $t\in J$ which, together with $\Psi_p(h^*,t)-\Psi_q(h^*,t)<0$, contradicts that $h^*$ is an \SCDE.
\end{proof}

\begin{obs}
	For any fixed time $t \in \planningInterval$ the values $g_{p\rightarrow q}(t), p,q\in \Pc$ defined in the proof above solve a transshipment problem defined as follows:
	We create a complete bipartite graph $G=(V_1(t)\cup V_2(t), E(t))$,
	where nodes in $V_1(t) \subseteq \Pc$ are surplus nodes, that is, 
	they fulfil $b_p(t):=h_p(t)-h^*_p(t)>0$, and nodes in $ V_2(t)\subseteq \Pc$
	are deficit nodes fulfilling $b_q(t):=h_q(t)-h^*_q(t)<0$.
	Note that obviously $V_1(t)\cap V_2(t)=\emptyset$ for all $t\in \planningInterval$.
	For every arc $(p,q)\in E(t):=V_1(t)\times V_2(t)$ we define
	capacities $c_{(p,q)}(t):=\min\{b_p(t), b_q(t)\}$.
\end{obs}

\begin{remark}
	Clearly $S=\Lambda(r)$ and $A_p$ defined by \eqref{eq:ApglobalSCDE} also satisfy the assumptions of \Cref{thm:VI-necessary}. Thus, \Cref{thm:VI-sufficient,thm:VI-necessary} together are a generalization of the characterization of unconstrained dynamic equilibria with fixed inflow rates by variational inequalities (\ie \Cref{thm:VICharOfDE}).
\end{remark}

Another interesting example for which \Cref{def:elementary} applies is the case of monotone box-constraints.
\begin{example}
	Consider continuous and non-decreasing functions $z_p:\R_+\rightarrow\R_+, p\in\Pc$ and continuous functions 	$v_p:\planningInterval\rightarrow\R_+, p\in\Pc$.
	We get that the set
		\[ S:=\Set{ h\in \Lambda(r) | z_p(h_p(t)) \leq v_p(t) \text{ for all } p\in\Pc \text{ and almost all } t \in \planningInterval}\]
	is closed with respect to elementary directions (for \addmEpsDev s defined by~\eqref{eq:ApglobalSCDE}).
	To see this, let $h,h'\in S$ with $h_p(t)-h'_p(t) > 0$ and $h'_q(t)-h_q(t) > 0$
	for all $t\in J$, where $J$ is some set of positive measure. 
	
	Then, there must also exist some $\varepsilon > 0$ and a subset $J' \subseteq J$ of positive measure with $h_p(t)-h'_p(t) \geq \varepsilon$ and $h'_q(t)-h_q(t) \geq \varepsilon$ for all $t \in J'$. We can then construct a sequence of sets $J_n \subseteq J'$ of positive measure satisfying $J_{n+1} \subseteq J_n$ and $\lim_n \inf J_n = \lim_n \sup J_n = t$ for some $t \in J'$.	
	Then, for any $n \in \INs$ we have
		\[ H_{p \to q}(h,\tfrac{\varepsilon}{n}\cdot\CharF[J_n],,0)_w(t) \leq \max\set{h_w(t),h'_w(t)}\]
	for all $t \in J_n$ and $w\in \Pc$. Since both $h$ and $h'$ are from $S$, this then implies
		\begin{align*}
			&z_w\left(H_{p \to q}(h,\tfrac{\varepsilon}{n}\cdot\CharF[J_n],0)_w(t)\right) \leq z_w\left(\max\set{h_w(t),h'_w(t)}\right) \\
			&\quad\quad\quad= \max\set{z_w(h_w(t)),z_w(h'_w(t))} \leq v_w(t)
		\end{align*}
	for almost all $t \in J_n$ and, thus, $H_{p \to q}(h,\tfrac{\varepsilon}{n}\cdot\CharF[J_n],0) \in S$. Since $\int_{J_n}\frac{\varepsilon}{n}\diff t' > 0$ and $t \in [\inf J_n,\sup J_n]$ hold as well, this is enough to show $(q,0) \in U_{p}(h,t)$.
\end{example}


\section{Edge Volume Constraints and Network-Loading}\label{sec:SCviaNL}

In this section, we will now come back to more concrete \SCDE[fulls], where flows are feasible whenever they obey certain ``capacity constraints'' on the edges of the network. For the following definitions we will only consider the case of fixed flow volume (and free departure time choice) but note that all definitions can be easily transferred to the case of fixed network inflow rates by choosing $S \cap \Lambda(r)$ instead of the sets $S$ defined here and only allowing \addmEpsDev s with $\Delta=0$. Furthermore, to simplify the notation we will not explicitly state the condition $H_{p \to q}(h,\shiftN,\Delta) \in \Lambda(Q)$ on elements of $A_p(h)$ whenever we define such sets here, \ie this condition should always be implicitly added when reading a definition of $A_p(h)$ in this section. Finally, statements of the form ``$\forall t \in \supp(\shiftN)\dots$'' for some $\shiftN \in L^2(\planningInterval)$ should always be interpreted as: ``There exists some representative of~$\shiftN$ such that $\forall t \in \supp(\shiftN)\dots$''.

In order to formally define \SCDE{} with capacity constraints, we assume that for every edge $e$, we are given a capacity function $c_e: \IR_{\geq 0} \to \IR$ and that the flow model is equipped with a weak form of network-loading associating with every walk inflow $h$ two types of functions:
\begin{itemize}
	\item \emph{Arrival time} functions $\tau^j_p(h,.): [t_0,t_f] \to \IR_{\geq 0}$ such that for every time $t \in \planningInterval$, walk $p =(v_1,v_2,\dots,v_{\abs{p}+1}) \in \Pc$ and $j \in \set{1,\dots,\abs{p}+1}$, the value $\tau^j_p(h,t)$ denotes the time at which a particle entering walk $p$ at time $t$ will arrive at the $j$-th node on walk $p$ (or, equivalently leaves the $(j-1)$-th edge/enters the $j$-th edge of walk $p$). In particular, $\tau^1_p(h,t)$ denotes the time at which a particle starts its journey and  $\tau^{\abs{p}+1}_p(h,t)$ denotes the time at which the particle arrives at the end of walk $p$ (\ie the sink).
	\item \emph{Edge-load} functions $f_e(h,.): \IR_{\geq 0} \to \IR$ such that for every edge $e$ and time $\theta \in \IR_{\geq 0}$, the value $f_e(h,\theta)$ denotes some measure of the flow induced by $h$ on edge $e$ at time~$\theta$ and which, in a feasible flow, has to be bounded by the edge capacity $c_e(\theta)$. 
\end{itemize}

\begin{remark}
	Given a model with a full network-loading (as described in \Cref{sec:counter}), the arrival time functions would be defined using the corresponding walk-delay functions $D_p$, i.e.
		\[\tau_p^j(h,t) \coloneqq t + D_{p|_{j-1}}(h,t),\]
	where $p|_{j-1}$ is the prefix of $p$ of length $j-1$. The edge-load function could then, for example, be the flow volume $\flowVolume[e](h,\theta)$, the queue length $q_e(h,\theta)$, the cumulative inflow $\int_0^\theta f^+_e(\zeta)\diff\zeta$ or the current edge inflow rate $f^+_e(\theta)$.
\end{remark}

We now define the set of all walk inflows resulting in network flows obeying these edge capacities by
\begin{align}\label{eq:FeasibilitySetforEdgeLoad}
	S \coloneqq \Set{h \in \Lambda(Q) | f_e(h,\theta) \leq c_e(\theta) \text{ \fa } e \in E \text{ and } \theta \in \IR_{\geq 0}}.
\end{align}
Using this \setS{} in order to define a \globalSCDE{} (\cf \Cref{def:strongCDE}) results in an equilibrium that we call \emph{\globalEL[full] (\globalEL)}. 
Since, as discussed in \Cref{sec:counter}, the \setS{} defined by \eqref{eq:FeasibilitySetforEdgeLoad} need not be convex, we can, in general, not use the corresponding variational inequality~\eqref{eq:VI-SCDE} to characterize those types of equilibria. Furthermore, the \addmDev s need not even satisfy \ref{ass:closedTime} or \ref{ass:closedSpace} (see \Cref{ex:NonConvexitOfEdgeLoadConstraints}), so we also cannot apply the quasi-variational inequality.
Another problem of \globalEL{} is that this definition is rather restrictive in what counts as an \addmEpsDev{} (and, in turn, leads to a rather weak type of equilibrium). Namely, particles are only allowed to deviate, if their deviation leads to a new flow which is again feasible for \emph{all} particles -- in particular also for the particles which are not themselves involved in the deviation. In \Cref{ex:DifferenceBetweenLocalAndGlobalFeas}, we provide an instance with a \globalSCDE{} in which particles seem to have a better alternative to their current route choice but are not allowed to deviate because that would lead to infeasibility for \emph{other} particles not involved in the deviation. 

\subsection{Dynamic Equilibria of Type LP, MNS and BS}

In the following,  we want to allow for deviations in which the deviating particles themselves will not violate the capacity constraints while ignoring potential violations by particles not directly involved in the deviation. To formalize this in terms of \addmEpsDev s, we need to answer two questions: At which (time) points during its potential alternative journey must a particle check whether it would violate some capacity constraint and how does a particle check whether it would violate a capacity constraint? For the first question we will consider two potential answers 
\begin{enumerate}[label=\alph*)]
	\item Whenever it enters a new edge, \ie at all points $\tau^j_q(h,t)$ for $j = 1, \dots, \abs{q}$ or
	\item Whenever a particle travels along an edge, \ie at all points $\theta \in [\tau^j_q(h,t),\tau^{j+1}_q(h,t)]$ for $j = 1, \dots, \abs{q}$. 
\end{enumerate}
Since a) clearly allows more deviations than b), we will call equilibria using a) \emph{strong} and equilibria using b) \emph{weak}. For the second question we will propose three different answers. First, we can follow the approach of Larsson and Patriksson in the static model (\cf \Cref{def:weakWE}) and require that alternative walks are truly unsaturated at the time of deviation. In other words, at any time at which a deviating particle would arrive at an edge/travel along an edge of the alternative walk (according to the travel times of the current flow), there must be some additional room left on this edge (i.e., $f_e < c_e$ must hold). We can formalize this by
\begin{align}\label{eq:FeasibleDeviationsAlwaysAdditionalSpaceEnter}
	A_{p}(h) \coloneqq \Set{(q,\shiftN,\Delta) | \begin{array}{l}
			\forall t \in \supp(\shiftN)+\Delta, e = (v_j,v_{j+1}) \in q: \\
			f_e(h,\tau^j_{q}(h,t)) < c_e(\tau^j_{q}(h,t))
		\end{array}}
\end{align}
and
\begin{align}\label{eq:FeasibleDeviationsAlwaysAdditionalSpaceTravelling}
	A_{p}(h) \coloneqq \Set{(q,\shiftN,\Delta) | \begin{array}{l}
			\forall t \in \supp(\shiftN)+\Delta, e = (v_j,v_{j+1}) \in q: \\
			f_e(h,\theta) < c_e(\theta) \text{ \fa } \theta \in [\tau^j_q(h,t),\tau^{j+1}_q(h,t)]
		\end{array}}.
\end{align}
We call the resulting equilibrium a \emph{\sCDEu[full] (\sCDEu)} or a  \emph{\wCDEu[full] (\wCDEu)}, respectively. While this equilibrium notion seems quite intuitive, it has the same drawback as noted by \citeauthor*{Marcotte04} for LP-equilibria in the static model: Namely, one can consider a network consisting of two paths of different length that share their first edge. Then, a flow which only sends flow over the longer path can still be an equilibrium if this flow completely uses the available capacity on the shared first edge. 

A more lenient definition, thus, would allow $f_e \leq c_e$ to be tight on any common prefix of $p$ and $q$, \ie 
\begin{align}\label{eq:FeasibleDeviationsAdditionalSpaceExceptCommonPrefixEnter}
	A_{p}(h) &\coloneqq \Set{\!(q,\shiftN,0) | \!\!\begin{array}{l}
			\exists \text{ a common prefix } w \text{ of } p \text{ and } q \text{ with } q = w\tilde{q} \text{ and} \\
			\forall t \in \supp(\shiftN), e = (v_j,v_{j+1}) \in \tilde{q}: f_e(h,\tau^j_{q}(h,t)) < c_e(\tau^j_{q}(h,t)))
	\end{array}\!\!\!} \\
	& \cup \Set{(q,\shiftN,\Delta) | \!\begin{array}{l}
			\forall t \in \supp(\shiftN)+\Delta, e = (v_j,v_{j+1}) \in q: f_e(h,\tau^j_{q}(h,t)) < c_e(\tau^j_{q}(h,t))
	\end{array}\!}\notag
\end{align}
and 
\begin{align}
	A_{p}(h) &\coloneqq \Set{(q,\shiftN,0) | \begin{array}{l}
			\exists \text{ a common prefix } w \text{ of } p \text{ and } q \text{ with } q = w\tilde{q} \text{ and} \\
			\forall t \in \supp(\shiftN), e = (v_j,v_{j+1}) \in \tilde{q}: \\ f_e(h,\theta) < c_e(\theta) \text{ for all } \theta \in [\tau^j_q(h,t),\tau^{j+1}_q(h,t)]
	\end{array}\!\!\!} \notag\\
	& \cup \Set{(q,\shiftN,\Delta) | \begin{array}{l}
			\forall t \in \supp(\shiftN)+\Delta, e = (v_j,v_{j+1}) \in q: \\ f_e(h,\theta) < c_e(\theta) \text{ for all } \theta \in [\tau^j_q(h,t),\tau^{j+1}_q(h,t)]
	\end{array}}.\label{eq:FeasibleDeviationsAdditionalSpaceExceptCommonPrefixTravelling}
\end{align}
Here, $w\tilde{q}$ denotes the walk obtained by concatenating the walks $w$ and $\tilde{q}$. As this definition is inspired by the drawbacks of the LP-equilibrium noted by \citeauthor*{Marcotte04}, we call the resulting equilibrium a \emph{\sCDEuP[full] (\sCDEuP)} and a \emph{\wCDEuP[full] (\wCDEuP)}, respectively.

Finally, we can also require  -- similarly to the \BS-equilibrium from the static model -- that  the new flow (obtained after a potential deviation) must satisfy the capacity constraints at all points relevant for the deviating particles. This is formalized in the following two types of \addmEpsDev s:
\begin{align}\label{eq:UnsaturatedPathsFeasibleAfterDeviationEnter}
	A_{p}(h) \coloneqq \Set{\!(q,\shiftN,\Delta) |\!\! \begin{array}{l}
				\forall t \in \supp(\shiftN)+\Delta, e = (v_j,v_{j+1}) \in q: \\ 	
				f_e(h',\tau_q^j(h',t)) \leq c_e(\tau_q^j(h',t)) \text{ where } h' = H_{p \to q}(h,\shiftN,\Delta)
			\end{array}\!\!\!}
\end{align}
and
\begin{align}\label{eq:UnsaturatedPathsFeasibleAfterDeviationTravelling}
	A_{p}(h) \coloneqq \Set{\!(q,\shiftN,\Delta) |\!\! \begin{array}{l}
				\forall t \in \supp(\shiftN)+\Delta, e = (v_j,v_{j+1}) \in q, \theta \in [\tau^j_q(h',t),\tau^{j+1}_q(h',t)]: \\
				f_e(h',\theta) \leq c_e(\theta) \text{ where } h' = H_{p \to q}(h,\shiftN,\Delta)
			\end{array}\!\!\!}.
\end{align}
We call this type of equilibrium a \emph{\sCDEdf[full] (\sCDEdf)} or a \emph{\wCDEdf[full] (\wCDEdf)}, respectively, as the \BS[full] in the static model also requires that the costs cannot decrease by switching to any other walk whenever such a switch would result in another feasible (static) flow (\cf \Cref{def:BS}). Note, however, that in the static model, feasibility for the particles involved in the deviation typically (if the cost functions are separable and nondecreasing) also ensures feasibility for all other particles -- thus, the difference between \globalEL{} and \sCDEdf{}/\wCDEdf{} vanishes there, whereas \globalEL{} are strictly weaker than both \sCDEdf{} and \wCDEdf{} in the dynamic setting.

The following definition summarizes all the different types of equilibria defined in this section:
\begin{definition}\label{def:TypesOfCDE}
	Let the \setS{} $S$ be defined in \eqref{eq:FeasibilitySetforEdgeLoad}. Then, a side-constrained dynamic equilibrium is 
	\begin{itemize}
		\item a \emph{\globalEL[full] (\globalEL)} if the \addmEpsDev{}s are defined by \eqref{eq:ApglobalSCDE},
		\item a \emph{\sCDEu[full] (\sCDEu)} if the \addmEpsDev{}s are defined by \eqref{eq:FeasibleDeviationsAlwaysAdditionalSpaceEnter},
		\item a \emph{\wCDEu[full] (\wCDEu)} if the \addmEpsDev{}s are defined by \eqref{eq:FeasibleDeviationsAlwaysAdditionalSpaceTravelling},
		\item a \emph{\sCDEuP[full] (\sCDEuP)} if the \addmEpsDev{}s are defined by \eqref{eq:FeasibleDeviationsAdditionalSpaceExceptCommonPrefixEnter} and
		\item a \emph{\wCDEuP[full] (\wCDEuP)} if the \addmEpsDev{}s are defined by \eqref{eq:FeasibleDeviationsAdditionalSpaceExceptCommonPrefixTravelling},
		\item a \emph{\sCDEdf[full] (\sCDEdf)} if the \addmEpsDev{}s are defined by \eqref{eq:UnsaturatedPathsFeasibleAfterDeviationEnter},
		\item a \emph{\wCDEdf[full] (\wCDEdf)} if the \addmEpsDev{}s are defined by \eqref{eq:UnsaturatedPathsFeasibleAfterDeviationTravelling}.
	\end{itemize}
\end{definition}

The relationships between these different equilibrium concepts are captured in the following \namecref{prop:RelationshipsOfCDE}:

\begin{proposition}\label{prop:RelationshipsOfCDE}
	Denoting the sets of equilibria by their respective names we have the following relations between them:
		\begin{center}
			\begin{tikzpicture}[node distance=3cm,every node/.append style={text depth=0.25ex}]
				\node(sL){\sCDEu};
				\node(wL)[right of=sL]{\wCDEu};
				\node(sM)[below of=sL, node distance=1cm]{\sCDEuP};
				\node(wM)[right of=sM]{\wCDEuP};
				
				\path(sL) --node[sloped]{$\subseteq$} (wL);
				\path(sM) --node[sloped]{$\subseteq$} (sL);
				\path(sM) --node[sloped]{$\subseteq$} (wM);
				\path(wM) --node[sloped]{$\subseteq$} (wL);
			\end{tikzpicture}
		\end{center}
	and
		\begin{center}
			\begin{tikzpicture}[node distance=3cm,every node/.append style={text depth=0.25ex}]
				\node(sB){\sCDEdf};
				\node(wB)[right of=sB]{\wCDEdf};
				\node(sC)[right of=wB]{\globalEL};
				
				\path(sB) --node[sloped]{$\subseteq$} (wB)
							--node[sloped]{$\subseteq$} (sC);
			\end{tikzpicture}
		\end{center}
	If both $f_e$ and $\tau^j_p$ depend continuously on $h$ while $f_e(h,.), \tau^j_p(h,.)$ and $c_e$ are all continuous functions, then we additionally have
		\begin{center}
			\begin{tikzpicture}[node distance=5cm,every node/.append style={text depth=0.25ex}]
				\node(sL){\sCDEu};
				\node(wL)[right of=sL]{\wCDEu};
				\node(sB)[below of=sL,node distance=1cm]{\sCDEdf};
				\node(wB)[right of=sB]{\wCDEdf};
				
				\path(sB) --node[sloped]{$\subseteq$} (sL);
				\path(wB) --node[sloped]{$\subseteq$} (wL);
				\path(sB) --node{and} (wL);
			\end{tikzpicture}
		\end{center}
	Furthermore, in general, all these inclusions are proper and the concepts of \globalEL{} and \sCDEu{} are independent of each other, \ie neither includes the other.
\end{proposition}

\begin{proof}
	The inclusions in the first two diagrams all follow directly from \Cref{lemma:ConditiononApForStrongerEquilibria}. E.g., \sCDEuP{} allow more \epsDev s than \sCDEu{} and, therefore, \sCDEuP{} is a stronger equilibrium concept than \sCDEu.
	
	For the inclusion \sCDEdf{} $\subseteq$ \sCDEu{}, we  cannot compare \addmEpsDev s but have to compare \addmDev s instead. That is, assume that $(q,\Delta) \in U_p(h,t)$ is an \addmDev{} according to \sCDEu. Then, due to the continuity of $f_e(h,.), \tau^j_p(h,.)$ and $c_e$, there must be some neighbourhood $N \subseteq \planningInterval$ of $t+\Delta$ such that for all $t' \in N$ and $e=(v_j,v_{j+1}) \in q$, we have $f_e(h,\tau_q^j(h,t')) < c_e(\tau_q^j(h,t'))$. Since all $f_e$ and $\tau_q^j$ also depend continuously on $h$, we then have $f_e(h^\varepsilon,\tau_q^j(h^\varepsilon,t')) \leq c_e(\tau_q^j(h^\varepsilon,t'))$ for all $t'$ in some smaller neighbourhood $N' \subseteq N$ and $h^\varepsilon \coloneqq H_{p \to q}(h,\varepsilon\CharF[N-\Delta],\Delta)$ for small enough $\varepsilon$. This then implies $(q,\varepsilon\CharF[N-\Delta],\Delta) \in A_p(h)$ according to \sCDEdf{} and, consequently, $(q,\Delta) \in U_p(h,t)$ according to \sCDEdf. Thus, all \addmDev s according to \sCDEu{} are \addmDev s according to \sCDEdf{} as well and, therefore, the latter is a stronger equilibrium concept than the former. The inclusion \wCDEdf{} $\subseteq$ \wCDEu{} can be shown in the same way.
	
	For the independence of \globalEL{} and \sCDEu{}, we note that the corresponding notion of \addmDev s can be both stricter for \globalEL{}  (a deviation might only be inadmissible because of infeasibility for particles not directly involved in the deviation -- see \Cref{ex:DifferenceBetweenLocalAndGlobalFeas}) and stricter for \sCDEu{}  (a deviation might still be possible even if $f_e \leq c_e$ is tight at relevant times because an increased inflow into the alternative walk $q$ does not necessarily increase $f_e$ on all of its edges). The former part also shows that there can exist \globalEL{} which are not a \sCDEdf{} while the latter part shows that there can exist \sCDEu{} which are not a \sCDEdf. 
\end{proof}

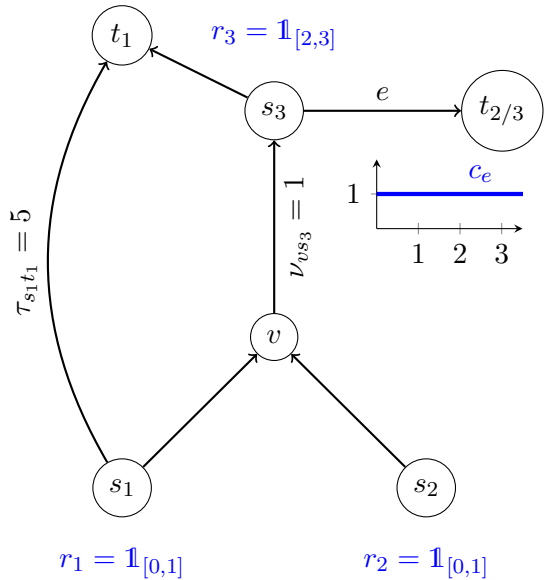
\begin{figure}[h]
	\centering
	\begin{tikzpicture}
	\node[namedVertex] (s1)at(0,0) {$s_1$};
	\node[namedVertex] (s2)at(4,0) {$s_2$};
	\node[namedVertex] (v)at(2,2) {$v$};
	\node[namedVertex] (s3)at(2,5) {$s_3$};
	\node[namedVertex] (t1)at(0,6) {$t_1$};
	\node[namedVertex] (t23)at(5,5) {$t_{2/3}$};
	
	\draw[edge] (s1) -- (v);
	\draw[edge] (s1) to[bend left=30] node[sloped,above,rotate=180]{$\tau_{s_1 t_1}=5$} (t1);
	\draw[edge] (s2) -- (v);
	\draw[edge] (v) -- node[sloped,below]{$\nu_{vs_3}=1$} (s3);
	\draw[edge] (s3) -- (t1);
	\draw[edge] (s3) -- node[above]{$e$} node[pos=.8](e){} (t23);
	
	\node[below of=e, anchor=north,node distance=.5cm] {
		\begin{tikzpicture}[scale=1,solid,black,
			declare function={
				c(\x)= 1;			
			}]

			\begin{axis}[xmin=0,xmax=3.5,ymax=2, ymin=0, samples=500,width=3.5cm,height=2.5cm,
				axis x line*=bottom, axis y line*=left, axis lines=middle, xtick={1,2,3,4}, ytick={1}]
				\addplot[blue, ultra thick,domain=0:5] {c(x)} node[above,pos=.5]{$c_{e}$};
			\end{axis}
			
		\end{tikzpicture}
	};

	\node[below of=s1,blue](){$r_1=\CharF[{[0,1]}]$};
	\node[below of=s2,blue](){$r_2=\CharF[{[0,1]}]$};
	\node[above of=s3,blue](){$r_3=\CharF[{[2,3]}]$};
\end{tikzpicture}
	\caption{A three commodity network with fixed network inflow rates. All values of $\tau_e$ not explicitly given in the figure are $1$ and all $\nu_e$ not given are infinity. Using the Vickrey point queue model for the edge dynamics and the capacity constraint on edge $e$ as volume or inflow rate constraint, this network has a unique feasible flow which is a \globalEL{} but neither a \wCDEdf{} nor a \wCDEu.}\label{fig:NonExistence}
\end{figure}

\begin{example}\label{ex:DifferenceBetweenLocalAndGlobalFeas}
	Consider the three commodity network with fixed network inflow rates given in \Cref{fig:NonExistence}. We use the Vickrey point queue model for the edge dynamics and the edge capacity function of edge $e=(s_3,t_{2/3})$ as volume or inflow rate constraint. Then, this network has a unique feasible flow: Commodity~$1$ sends all its flow via the direct edge towards $t_1$ and the other two commodities send all flow along their only available path. This is the only feasible flow (up to changes on a subset of measure zero) since, if commodity~$1$ were to send any of its flow along the alternative path via $v$ and $s_3$, this would result in a congestion on edge $vs_3$. This, in turn, would lead to some of commodity~$2$'s particles arriving at $s_3$ after time $2$ and, thus, entering edge $e$ at the same time as the particles of commodity~$3$. This then leads to  a violation of the capacity constraint on edge~$e$.
	
	Consequently, this unique flow is a \globalEL{} (since there are no \addmDev s). However, it is neither a \wCDEdf{} nor a \wCDEu{} as particles of commodity~$1$ could deviate to the shorter path without violating any capacity constraint \emph{themselves}. In particular, this network does not have any \wCDEdf{} or \wCDEu{}.
\end{example}

\begin{obs}\label{obs:CDEcharbyQVI}
	Admissible $\gamma$-deviations defined by either \eqref{eq:FeasibleDeviationsAlwaysAdditionalSpaceEnter}, \eqref{eq:FeasibleDeviationsAlwaysAdditionalSpaceTravelling}, \eqref{eq:FeasibleDeviationsAdditionalSpaceExceptCommonPrefixEnter} or \eqref{eq:FeasibleDeviationsAdditionalSpaceExceptCommonPrefixTravelling} clearly satisfy both \labelcref{ass:closedSpace,ass:closedTime}. Thus, \sCDEu, \wCDEu, \sCDEuP{} and \wCDEuP{} with continuous \effWalkDelay s (\ie satisfying \cref{ass:EffectivePathDelayContinuous}) are all characterized by their corresponding quasi-variational inequality \eqref{eq:QVI-SCDE} (\cf \Cref{thm:VI-fixed-inflow:nessecary,thm:VI-fixed-inflow:sufficient}).
\end{obs}


\subsection{Existence Results}

As shown in \Cref{sec:counter}, feasibility sets defined by capacity constraints need not be convex. Thus, it is not clear whether \eqref{eq:VI-SCDE} has a solution (at least, we cannot apply \Cref{thm:Lions}). Furthermore, in general it is not guaranteed that small enough \addmEpsDev s from a feasible flow lead to another feasible flow (\ie \cref{ass:addmDevLeadToFeasFlow} need not hold) and, therefore, even if we had a solution to the variational inequality, this would not be guaranteed to also be an \SCDE. See \Cref{ex:DifferenceBetweenLocalAndGlobalFeas} for such an instance.

To still be able to show existence of \sCDEuP{} and \wCDEuP{} (under certain additional assumptions), we will therefore employ a different approach using an augmented Lagrangian relaxation of the hard capacity constraints. Before we come to the existence theorem for \sCDEuP, we introduce the additional assumptions we will need there. Note that some of the assumptions have to be stated slightly differently for the setting with and without departure time choice (DTC). In particular, we will use $\Lambda$ to refer to $\Lambda(Q)$ in the case with DTC and to $\Lambda(r)$ in the case without DTC.

We start with the assumptions on the underlying network:
\begin{enumerate}[label=(A\arabic*),resume=Assumptions]
	\item For any edge $e$ and walk inflow $h$, we have $f_e(h,0) \leq c_e(0)$.\label[asmpt]{ass:FeasibleAtStart}
	\item With DTC: For every $h\in \Lambda(Q)$ and $i\in I$, there exists some $p\in \Pc_i$ and $J \subseteq [t_0,t_f]$ of positive measure such that for all $t \in J$, we have $h_p(t) < B_p$ and $f_e(h, \tau^j_p(h,t))\leq c_e(\tau^j_p(h,t))$ for all $e = (v_j,v_{j+1})\in p$.\label[asmpt]{ass:ExistenceOfUnsaturatedPath} 
	
	Without DTC: For every $h\in \Lambda(r)$ and $i\in I$ and $t \in \planningInterval$, there exists some $p\in \Pc_i$ such that we have $f_e(h, \tau^j_p(h,t))\leq c_e(\tau^j_p(h,t))$ for all $e = (v_j,v_{j+1})\in p$.
	\item For any edge $e$, the capacity function $c_e$ is non-negative.\label[asmpt]{ass:CapacityNN} 
	\item For any edge $e$, the capacity function $c_e$ is non-decreasing.\label[asmpt]{ass:CapacityNonDecreasing}
\end{enumerate}
Here, \cref{ass:FeasibleAtStart} states that there can be no capacity violation at the start (\ie in the empty network). \Cref{ass:ExistenceOfUnsaturatedPath} essentially states that there is always an outside option, \ie the edge capacities are chosen large enough such that there is always at least one walk which is not oversaturated. Note, that this can be easily accomplished either by having one walk with large enough capacities or by using elastic demands. \Cref{ex:DifferenceBetweenLocalAndGlobalFeas} shows why this assumption is necessary for the existence of \sCDEu. To see why \cref{ass:CapacityNonDecreasing} is important for the existence of equilibria, we refer to \Cref{ex:NonExistenceForNonConstantCapacity}.

Next, we state the required assumptions on the flow dynamics, \ie the network-loading, effective walk-delay and edge-load functions (note that, for many flow models these assumptions are known to always be satisfied -- \cf \Cref{lemma:VickreyModelSatisfiesAssumptions}):

\begin{enumerate}[label=(A\arabic*),resume=Assumptions]
	\item For any edge $e$ and walk inflow $h$, the function $f_e(h,\emptyarg)$ is continuous.\label[asmpt]{ass:fContinuous}
	\item For any edge $e$ and any weakly convergent sequence $h^n$, the sequence $f_e(h^n,\emptyarg)$ converges uniformly. \label[asmpt]{ass:fConvergesUniform}
	\item For any $h \in \Lambda$, we have that $h$ and $f(h)$ satisfy the \emph{principle of causation}: For any edge $e\in E$ and interval $[a,b] \subseteq \IR_{\geq 0}$ with $0 \leq f_e(h,a) < f_e(h,b)$ there exists some walk $p \in  \Pc$ and subset $J^{-1}\subseteq[t_0,t_f]$ of positive measure such that we have $e = (v_j,v_{j+1}) \in p$, $h_p(t')>0$ for all $t'\in J^{-1}$ and $\tau_p^j(h,t')\in [a,b]$ for all $t'\in J^{-1}$. \label[asmpt]{ass:StrongPrincipleOfCausation}
	\item For any walk $p$ and $j \in \set{1, \dots, \abs{p}+1}$ and any weakly convergent sequence $h^n$, the sequence $\tau_p^j(h^n,\emptyarg)$ converges uniformly. \label[asmpt]{ass:tauConvergesUniform}
	\item For any two walks $p$ and $q$ that share a common prefix $w$, we have $\tau^j_p(h,t) = \tau^j_q(h,t)$ for all $e=(v_j,v_{j+1}) \in w$, $h \in S$ and $t \in \planningInterval$.\label[asmpt]{ass:tEqualOnCommonPrefix}
	\item The mapping $\Lambda \to C(\planningInterval)^\Pc, h \mapsto \Psi(h,\emptyarg)$ is sequentially weak-strong continuous, \ie for any walk $p \in \Pc$ and any weakly convergent sequence $h^n$, the sequence $\Psi_p(h^n,\emptyarg)$ converges uniformly.\label[asmpt]{ass:PsiConvergesUniform}
	\item There exists some $M \in \IR$ such that for all $h \in S$, $p \in \Pc$ and almost all $t \in \planningInterval$, we have $\Psi_p(h,t) \leq M$.\label[asmpt]{ass:PsiBoundedFlowIndependent}
\end{enumerate}
Here, \cref{ass:fContinuous,ass:fConvergesUniform,ass:tauConvergesUniform} are standard continuity assumptions. \Cref{ass:StrongPrincipleOfCausation} states that, whenever the edge load increases during some interval, then there must exist earlier times at which strictly positive flow is injected into some walk containing~$e$ and contributing to this increased edge load. 
\Cref{ass:tEqualOnCommonPrefix} requires that intermediate travel times only depend on the part of the walk already traversed and not on future route choices. Finally, \labelcref{ass:PsiConvergesUniform} is a strengthening of \labelcref{ass:PsiWScont} as convergence with respect to the uniform norm implies convergence with respect to the $L^2$-norm while \labelcref{ass:PsiBoundedFlowIndependent} is a strengthening of \labelcref{ass:PsiBounded} as we are now asking for the effective walk-delay operator to be uniformly bounded instead of just being bounded for any fixed walk-inflow~$h$.

\begin{theorem}\label{thm:ExistenceFDAddSpaceExCP}
	Under \cref{ass:PsiBoundedFlowIndependent,ass:FinitelyManyWalks,ass:PathInflowBounds,ass:PsiConvergesUniform,ass:EffectivePathDelayContinuous,ass:StrongPrincipleOfCausation,ass:fContinuous,ass:fConvergesUniform,ass:CapacityNonDecreasing,ass:CapacityNN,ass:ExistenceOfUnsaturatedPath,ass:tauConvergesUniform,ass:tEqualOnCommonPrefix,ass:FeasibleAtStart} there always exist \sCDEu{} and \sCDEuP{} (both with and without DTC).
\end{theorem}

The main steps for proving this \namecref{thm:ExistenceFDAddSpaceExCP} are as follows: We start by relaxing the capacity constraints and instead include a (proportional) penalty for violating the edge capacities in the \effWalkDelay{} operator. Next, we show that we may assume \wlofg that this adjusted walk-delay operator is still bounded, as the existence of an outside option (\ie \ref{ass:ExistenceOfUnsaturatedPath}) guarantees that in an (unconstrained) equilibrium particles will never pay high penalties. We can then apply the usual existence results for unconstrained dynamic equilibria to get a weakly convergent sequence of such equilibria with respect to higher and higher penalties. It then remains to show that the limit point of this sequence is both feasible \wrt $S$ and an equilibrium \wrt the given $A_p$. The former is relatively straightforward as the increasing penalties together with the existence of an outside option imply that fewer and fewer particles may violate a capacity constraint at any point in time. The latter is more technical but, intuitively, we just have to show that, due to our continuity assumptions on the flow dynamics, an improving \addmDev{} in the limit would imply the existence of an improving deviation within the sequence of (unconstrained) equilibria.

\begin{proof}
	As every \sCDEuP{} is also a \sCDEu{} (see \Cref{prop:RelationshipsOfCDE}), it suffices to show the existence of the former. 
	We start by defining penalty functions $\xi_e(h,\theta):=\max\{0,f_e(h,\theta)-c_e(\theta)\},e\in E,\theta\in \IR_{\geq 0}$ and write $\xi_p(h,t):=\sum_{e = (v_j,v_{j+1})\in p} \xi_e(h,\tau^j_p(h,t))$ for all $t \in \planningInterval$.
	By \cref{ass:PsiConvergesUniform,ass:fConvergesUniform,ass:tauConvergesUniform}, the new $\lambda$-parametrized \effWalkDelay{} operator
		\[ \Psi_p^{\lambda}(h,t) \coloneqq \Psi_p(h,t)+\lambda \xi_p(h,t)\text{ for all }t\in [t_0, t_f], p\in  \Pc,\]
	satisfies \ref{ass:PsiWScont} for any $\lambda>0$ as well. 
	
	\begin{claim}\label{claim:BoundOnEssInf}
		There exists some $M\in \R_+$ such that we have 
			\[\min_{p\in \Pc_i}\essinf\set{\Psi^{\lambda}_p(h,t) | t \in \planningInterval: h_p(t) < B_p} \leq M\]
		for all $\lambda \geq 0$, $i \in I$ and $h \in \Lambda(Q)$ for the case with DTC and 
			\[\sup_{t \in \planningInterval}\min_{p\in \Pc_i}{\Psi^{\lambda}_p(h,t)} \leq M\]
		for all $\lambda \geq 0$, $i \in I$ and $h \in \Lambda(r)$ for the case without DTC.
	\end{claim}
	
	\begin{proofClaim}
		We start with the case with DTC: 
		By assumption~\ref{ass:ExistenceOfUnsaturatedPath}, for every $h$ and $i \in I$ there exist $p\in \Pc_i$ and $J \subseteq \planningInterval$ with positive measure such that $h_p(t) < B_p$ and $f_e(h,\tau^j_p(h,t)) \leq c_e(\tau^j_p(h,t))$ 
		and, hence, $\xi_e(h,\tau^j_p(h,t)) = 0$ for all $e = (v_j,v_{j+1}) \in p$ and $t \in J$. Thus, using \ref{ass:PsiBoundedFlowIndependent}, we have $\Psi^{\lambda}_p(h,t) = \Psi_p(h,t) + 0 \leq M$ for all $t \in J$. Since $J$ has positive measure, this implies the claim's first part.
		
		Now for the case without DTC take any time $t \in \planningInterval$. By \cref{ass:ExistenceOfUnsaturatedPath} there exists some walk $p \in \Pc_i$ with $f_e(h,\tau^j_p(h,t)) \leq c_e(\tau^j_p(h,t))$ 
		and, hence, $\xi_e(h,\tau^j_p(h,t)) = 0$ for all $e = (v_j,v_{j+1}) \in p$. Thus, using \ref{ass:PsiBoundedFlowIndependent} we have $\Psi^{\lambda}_p(h,t) = \Psi_p(h,t) + 0 \leq M$, which proves the claim's second part.
	\end{proofClaim} 
	
	Using this claim we can now, in the setting without side-constraints, apply \Cref{lemma:RestrictToTruncatedPsi} to replace $\Psi^{\lambda}$ by the truncated effective walk delay operator $\truncated{\Psi^\lambda}{M+1}$ which is, in particular, bounded (\ie satisfies \ref{ass:PsiBounded}) while yielding only dynamic equilibria which are also equilibria \wrt $\Psi^\lambda$. As $\truncated{\Psi^\lambda}{M+1}$ still satisfies \ref{ass:PsiWScont} while \ref{ass:FinitelyManyWalks} and \ref{ass:PathInflowBounds} are part of the theorem's assumption, we can apply \Cref{thm:ExistenceUnconstrained} to obtain for any $\lambda \geq 0$ an unconstrained dynamic equilibrium with respect to $\Psi^\lambda$ (with or without DTC).
	
	Let us now take a sequence of strictly positive numbers $(\lambda_n)_{n\in \IN}$ with $\lambda_n\rightarrow\infty$ and a corresponding sequence $h^n\in \Lambda$ of unconstrained dynamic equilibria with respect to $\Psi^{\lambda_n}$.
	By taking subsequences, $h^n$ weakly converges to some $h^*\in \Lambda$ as $\Lambda$ is bounded, closed and convex (\cf \cite[Corollary~7.32]{FunctionalAnalysisHundertmark}). We will now show that this $h^*$ is a \sCDEuP{} by first showing that it is feasible and then that it also satisfies the \SCDE-condition~\eqref{eq:CDE} for \addmEpsDev s defined by \eqref{eq:FeasibleDeviationsAdditionalSpaceExceptCommonPrefixEnter}.
	
	\begin{claim}\label{claim:hStarInSr}
		$h^*$ is feasible, i.e.\ $h^*\in S$.
	\end{claim}

	\begin{proofClaim}
		Suppose that $h^* \notin S$, \ie there exists some $e\in E$ and some time $\theta \in \IR_{\geq 0}$ with $\xi_e(h^*,\theta)>0$. 	Now, by assumption~\ref{ass:fConvergesUniform}, $f(h^n)$ converges uniformly to $f(h^*)$ and, in particular, there exists some $\delta > 0$ such that for large enough $n \in \IN$ we have 
			\[f_e(h^n,\theta) - c_e(\theta) = \xi_e(h^n,\theta) \geq \delta.\]
		
		Using \cref{ass:FeasibleAtStart,ass:fContinuous,ass:CapacityNN,ass:CapacityNonDecreasing}, there must be some proper interval $[\theta',\theta]$ such that $f_e(h^n,\theta') < f_e(h^n,\theta)$ and $f_e(h^n,\vartheta) - c_e(\vartheta) \geq \frac{\delta}{2}$ for all $\vartheta \in [\theta',\theta]$. By~\ref{ass:StrongPrincipleOfCausation} this gives us some $p_n \in \Pc$ with $e = (v_{j_n},v_{j_n+1}) \in p_n$ and some set $J^{-1}_{n} \subseteq [t_0,t_f]$ of positive measure, such that for all $t' \in J^{-1}_{n}$ we have
			\[h^n_{p_n}(t')>0 \text{ and } \tau^{j_n}_{p_n}(t',h^n) \in [\theta',\theta]\]
		and, thus, 
			\[\Psi^{\lambda_n}_{p_n}(h^n,t') \geq \lambda_n\xi_e(h^n,\tau^{j_n}_{p_n}(t',h^n)) \geq \lambda_n\tfrac{\delta}{2}.\]
		Since $\lambda_n \to \infty$, there exists some $n \in \IN$ with $\lambda_n\frac{\delta}{2} > M$ and, hence, for this $n$ we have a set $J^{-1}_{n}$ of positive measure such that for all $t' \in J^{-1}_{n}$ we have $h^n_{p_n}(t') > 0$ as well as $\Psi^{\lambda_n}_{p_n}(h^n,t') > M$. At the same time we have, by \Cref{claim:BoundOnEssInf}, \[M \geq \min_{p\in \Pc_i}\essinf\set{\Psi^{\lambda^n}_{p}(h^n,t) | t \in \planningInterval: h^n_p(t) < B_p}  = \nu_i\] or \[M \geq \sup_{t \in \planningInterval}\min_{p \in \Pc_i}\set{\Psi^{\lambda^n}_{p}(h^n,t)}\] for the case with or without DTC, respectively. But this is now a contradiction to $h^n$ being an unconstrained dynamic equilibrium.
	\end{proofClaim}

	\begin{claim}\label{claim:ExistenceFDAddSpaceExCPisEq}
		$h^*$ is an \SCDE{} \wrt the \addmEpsDev s defined in \eqref{eq:FeasibleDeviationsAdditionalSpaceExceptCommonPrefixEnter}.
	\end{claim}

	\begin{proofClaim}
		Suppose that $h^*\in S$ is not an equilibrium, that is, there
		is some $i\in I,p,q\in \Pc_i$ and $t \in [t_0,t_f]$ with $(q,\Delta)\in U_{p}(h, t)$ and $\Psi_p(h^*,t)>\Psi_q(h^*,t+\Delta)$. We distinguish two cases: If $\Delta = 0$, we define $w$ as the maximal common prefix of $p$ and $q$, \ie $p = w\tilde{p}$ and $q = w\tilde{q}$ for some subwalks $\tilde{p}$ and $\tilde{q}$. If $\Delta\neq 0$, we define $w$ as the empty walk and $\tilde p \coloneqq p$ and $\tilde q \coloneqq q$.
			
		\begin{itemize}
			\item With \ref{ass:EffectivePathDelayContinuous} and \ref{ass:PsiConvergesUniform}, \ie continuity of $\Psi_p(h^*,\emptyarg)$ and uniform convergence  of $\Psi_p(h^n,\emptyarg)$ to $\Psi_p(h^*,\emptyarg)$, we get 
			the existence of some $\delta > 0$ such that
				\begin{align}\label{claim:ExistenceFDAddSpaceExCPisEq:Step:BetterAlternative}
					\Psi_p(h^n,t')>\Psi_q(h^n,t'+\Delta) \text{ for all } t' \in [t-\delta,t+\delta] \text{ for $n$ large enough.}
				\end{align}
			
			\item From $(q,\Delta)\in U_{p}(h^*, t)$, we get a function $\shiftN$ with $\supp(\shiftN) \subseteq [t-\delta,t+\delta]$, $\int_{\tStart}^{\tEnd}\shiftN(t')\diff t' > 0$ and $(q,\shiftN,\Delta) \in A_{p}(h^*)$. By the definition of \addmEpsDev s (\ie \eqref{eq:FeasibleDeviationsAdditionalSpaceExceptCommonPrefixEnter}) this implies 	
				\[f_e(h^*,\tau^j_{q}(h^*,t')) < c_e(\tau^j_{q}(h^*,t')) \text{ for all } e = (v_j,v_{j+1}) \in \tilde{q}, t' \in \supp(\shiftN)+\Delta.\]
			
			\item With \ref{ass:fConvergesUniform} and \ref{ass:tauConvergesUniform}, \ie the uniform convergence of $f_e(h^n,\emptyarg)$ to $f_e(h^*,\emptyarg)$ and $\tau^j_p(h^n,\emptyarg)$ to $\tau^j_p(h^*,\emptyarg)$, we have the existence of some $N \in \IN$ and $J \subseteq \supp(\shiftN)$ such that $\int_J \shiftN(t')\diff t' > 0$ and 
		 		\[f_e(h^n,\tau^j_{q}(h^n,t')) < c_e(\tau^j_{q}(h^n,t')) \text{ for all } n \geq N, e = (v_j,v_{j+1}) \in \tilde{q} \text{ and } t' \in J+\Delta.\]
			
			\item With weak convergence of $h^n\rightarrow h^*$ we now get the existence of some $J_n\subseteq J$ of positive measure with $h_p^n(t')>0$ for all $t' \in J_n$ for $n$ large enough:
			
			To see this, let $\CharF[J,p] \in L^2(\planningInterval)^\Pc$ be the characteristic function of the measurable set $J$ for walk $p$ and the zero function for all other walks. Then the weak convergence of $h^n$ to $h^*$ implies
				\[\lim_n \int_{J} h^n_{p}(t') \diff t' = \lim_n \scalar{h^n}{\CharF[J,p]} = \scalar{h^*}{\CharF[J,p]} = \int_{J} h^*_{p}(t')\diff t'\geq \int_J\shiftN(t')\diff t' > 0.\]
			This implies $\int_{J} h^n_{p}(t') \diff t' > 0$ for large enough $n$ and, thus, for any such $n$ there exists a subset $J_n \subseteq J$ of positive measure with $h^n_{p}(t') > 0$ for all $t' \in J_n$. 
		\end{itemize}
	
		Taking all this together we can now show a contradiction as follows: Choose $n \in \IN$ large enough such that all the previous statements hold. Then, in particular, we have $f_e(h^n,\tau^j_{q}(h^n,t')) < c_e(\tau^j_{q}(h^n,t'))$ and, thus, 
		\begin{align}\label{claim:ExistenceFDAddSpaceExCPisEq:Step:NoPenalty}
			\xi_e(h^n,\tau^j_{q}(h^n,t')) = 0 \text{ for all } e = (v_j,v_{j+1}) \in \tilde{q} \text{ and } t' \in J_n + \Delta.
		\end{align}
		
		This then implies
			\begin{align*}
				\Psi_p^{\lambda_n}(h^n,t') 
					&= \Psi_p(h^n,t') + \lambda_n \xi_p(h^n,t') \\
					&= \Psi_p(h^n,t') + \lambda_n \sum_{e=(v_j,v_{j+1}) \in w\tilde{p}}\xi_e(h^n,\tau^j_p(h^n,t')) \\
					&\geq \Psi_p(h^n,t') + \lambda_n \sum_{e=(v_j,v_{j+1}) \in w}\xi_e(h^n,\tau^j_p(h^n,t')) \\
					\overset{\text{\eqref{claim:ExistenceFDAddSpaceExCPisEq:Step:BetterAlternative}}}&{>} \Psi_q(h^n,t'+\Delta) + \lambda_n \sum_{e=(v_j,v_{j+1}) \in w}\xi_e(h^n,\tau^j_p(h^n,t')) \\
					\overset{\ref{ass:tEqualOnCommonPrefix}}&{=} \Psi_q(h^n,t'+\Delta) + \lambda_n \sum_{e=(v_j,v_{j+1}) \in w}\xi_e(h^n,\tau^j_q(h^n,t'+\Delta)) \\
					\overset{\text{\eqref{claim:ExistenceFDAddSpaceExCPisEq:Step:NoPenalty}}}&{=} \Psi_q(h^n,t'+\Delta) + \lambda_n \sum_{e=(v_j,v_{j+1}) \in w\tilde{q}}\xi_e(h^n,\tau^j_q(h^n,t'+\Delta)) \\
					&= \Psi_q(h^n,t'+\Delta) + \lambda_n \xi_q(h^n,t') = \Psi_q^{\lambda_n}(h^n,t'+\Delta) 
			\end{align*}
		for all $t' \in J_n$. For the fourth to last equality, note that $w$ is the empty walk whenever we have $\Delta \neq 0$.
		
		Now, at the same time, $t' \in J_n$ also implies $h_p^n(t')>0$. Together with the fact that $J_n$ has positive measure, these two statements are now a contradiction to $h^n$ being an unconstrained dynamic equilibrium.
	\end{proofClaim}
	Finally, \Cref{claim:hStarInSr,claim:ExistenceFDAddSpaceExCPisEq} together imply that $h^*$ is a \sCDEuP.
\end{proof}

Note, that the same type of proof cannot be applied to show existence of \sCDEdf{}. While everything up to and including \Cref{claim:hStarInSr} still holds (since the constraint set $S$ remains unchanged), \Cref{claim:ExistenceFDAddSpaceExCPisEq} is not guaranteed to hold any more for \addmEpsDev s defined by \eqref{eq:UnsaturatedPathsFeasibleAfterDeviationEnter}. 
In fact there exists an instance wherein the sequence $h^n$ from the proof of \Cref{thm:ExistenceFDAddSpaceExCP} converges to some $h^*$ which is not a \sCDEdf:

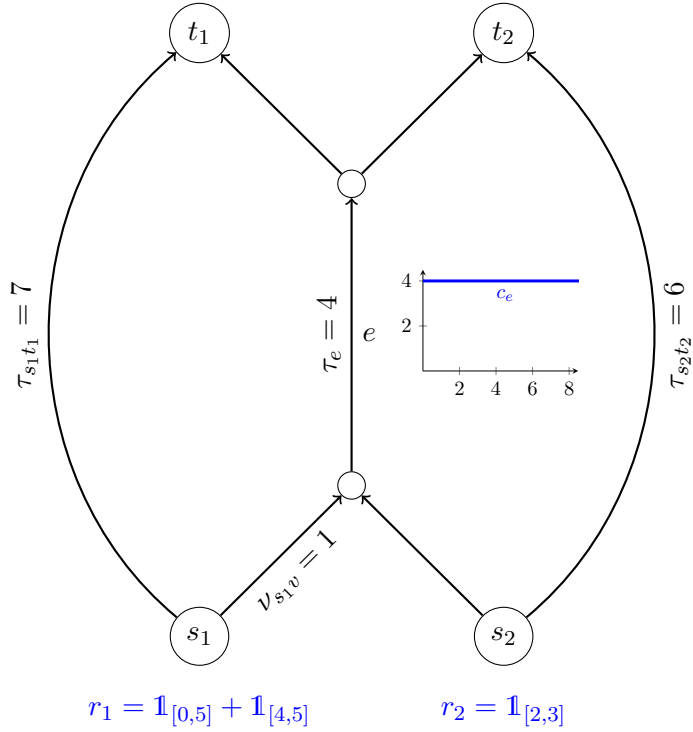
\begin{figure}
	\centering
	\begin{tikzpicture}
	\node[namedVertex] (s1) at (0,0) {$s_1$};
	\node[namedVertex] (s2) at (4,0) {$s_2$};
	\node[namedVertex] (v) at (2,2) {};
	\node[namedVertex] (w) at (2,6) {};
	\node[namedVertex] (t1) at (0,8) {$t_1$};
	\node[namedVertex] (t2) at (4,8) {$t_2$};
	
	\draw[edge] (s1) to[bend left=50] node[above,sloped,rotate=180]{$\tau_{s_1t_1}=7$} (t1);
	\draw[edge] (s1) -- node[below,sloped]{$\nu_{s_1v} = 1$} (v);
	\draw[edge] (v) -- node[right](e){$e$} node[sloped,above]{$\tau_e=4$} (w);
	\draw[edge] (w) -- (t1);
	\draw[edge] (s2) -- (v);
	\draw[edge] (w) -- (t2);
	\draw[edge] (s2) to[bend right=50] node[below,sloped]{$\tau_{s_2t_2}=6$} (t2);
		
	\node[right of=e, anchor=west,node distance=.2cm] {
		\begin{tikzpicture}[scale=.7,solid,black,
			declare function={
				c(\x)= 4;			
			}]

			\begin{axis}[xmin=0,xmax=8.5,ymax=4.5, ymin=0, samples=500,width=4.5cm,height=3.5cm,
				axis x line*=bottom, axis y line*=left, axis lines=middle, xtick={2,4,6,8},ytick={2,4}]
				\addplot[blue, ultra thick,domain=0:9] {c(x)} node[below,pos=.5]{$c_{e}$};
			\end{axis}
			
		\end{tikzpicture}
	};
	
	\node[below of=s1,blue](){$r_1=\CharF[{[0,5]}] + \CharF[{[4,5]}]$};
	\node[below of=s2,blue](){$r_2=\CharF[{[2,3]}]$};
\end{tikzpicture}
	\caption{An instance with volume constraints, fixed network inflow rates and the Vickrey queueing model for the edge dynamics. All values of $\tau$ not explicitly given are $1$ and all missing values of $\nu$ are infinite. In this network there exists a sequence of equilibrium flows for the instance with increasing prices (instead of strict volume constraints) which converges to a feasible flow which is a \sCDEuP{} but not a \sCDEdf.}\label{fig:CounterExamplePrices}
\end{figure}

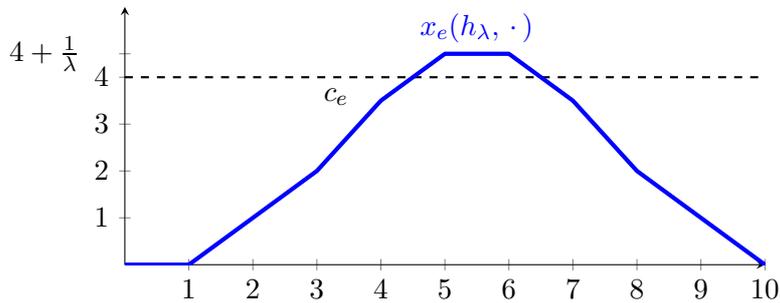
\begin{figure}
	\centering
	\begin{tikzpicture}[scale=1,solid,black,
	declare function={
		c(\x)= 4;	
		vol(\x)= 0 + and(\x>1,\x <= 5)*(\x-1) + and(\x > 5, \x <= 6)*4 + and(\x > 6, \x <=10)*(4-(\x-6)) + and(\x > 3,\x<4)*(.5*(\x-3)) + and(\x >= 4, \x <= 7)*.5 + and(\x > 7, \x <8)*(.5-.5*(\x-7));		
	}]

	\begin{axis}[xmin=0,xmax=10,ymax=5.5, ymin=0, samples=500,width=10cm,height=5cm,
		axis x line*=bottom, axis y line*=left, axis lines=middle, xtick={1,2,3,4,5,6,7,8,9,10},ytick={1,2,3,4,4.5},yticklabels={$1$,$2$,$3$,$4$,$4+\tfrac{1}{\lambda}\quad$}]
		\addplot[thick,domain=0:11,dashed] {c(x)} node[below,pos=.3]{$c_{e}$};
		\addplot[ultra thick,domain=0:11,blue] {vol(x)} node[above,pos=.5]{$x_{e}(h_\lambda,\emptyarg)$};
	\end{axis}

\end{tikzpicture}
	\caption{The flow volume on edge~$e$ of the network from \Cref{fig:CounterExamplePrices} induced by the flow split described in \Cref{ex:CounterExamplePrices}.}\label{fig:CounterExamplePricesVolGraph}
\end{figure}

\begin{example}\label{ex:CounterExamplePrices}
	The two-commodity network with fixed network inflow rates in \Cref{fig:CounterExamplePrices} shows why existence of \sCDEdf{} cannot be shown by relaxing the constraints and replacing them with fees for violation (as done in the proof of \Cref{thm:ExistenceFDAddSpaceExCP} for \sCDEu{} and \sCDEuP). The given network uses the Vickrey point queue model for the edge dynamics and has a constant volume constraint of $4$ on the central edge. For any price $\lambda$ the following is an equilibrium: Commodity~$1$ uses the central path at a rate of $1$ during the interval $[0,5]$ and (additionally) the direct edge towards $t_1$ at a rate of $1$ during the interval $[4,5]$. Commodity~$2$ uses the central path at a rate of $1/\lambda$ and sends everything else over the direct edge towards $t_2$. The flow volume on edge $e$ induced by such a flow split is depicted in \Cref{fig:CounterExamplePricesVolGraph}. Note, that during the interval $[3,4]$ this flow volume is less than $4$ and, thus, the particles of commodity~$2$ starting during $[2,3]$ are indifferent between the central and the right path. During the interval $[5,6]$, on the other hand, the capacity constraint on edge~$e$ is violated and particles entering during this interval have to pay a penalty of $\lambda\cdot\frac{1}{\lambda}=1$. In particular, particles of commodity~$1$ starting during $[4,5]$ are indifferent between the left and the central path. Therefore, the flow split described above is indeed an unconstrained equilibrium for price~$\lambda$. 
	
	Letting $\lambda$ go to $\infty$ these equilibria converge to a flow where commodity~$2$ sends everything along the direct edge while commodity~$1$ uses the same split as for every $\lambda$. This is clearly a feasible flow and also a \sCDEuP{} since the capacity constraint on edge $e$ is tight during $[5,7]$ and, thus, the central path is not an \addmDev{} during $[4,5]$. However, it is not a \sCDEdf{} as, due to the edge dynamics on edge $s_1v$, sending more flow into the central path during $[4,5]$ would not actually increase the flow volume on edge $e$ and, therefore, not violate the volume constraint on this edge. Thus, the central path is an \addmDev{} for \addmEpsDev s defined by \eqref{eq:UnsaturatedPathsFeasibleAfterDeviationEnter}.
\end{example}

The following example shows why it is important for the previous existence proof that the capacity functions $c_e$ are non-decreasing (\ie why we need \cref{ass:CapacityNonDecreasing}). 

\begin{figure}
	\centering
	\begin{tikzpicture}
	\node[namedVertex] (s1)at(0,0) {$s_1$};
	\node[namedVertex] (t1)at(4,0) {$t_1$};
	
	\draw[edge] (s1) --node(e1)[above]{$e_1$} node[sloped,below]{$\tau_{e_1}=2$} (t1);
	\draw[edge] (s1) to[bend right=60] node[above]{$e_2$} node[sloped,below]{$\tau_{e_2}=5$} (t1);
	
	\node[above of=e1, anchor=south,node distance=.15cm] {
		\begin{tikzpicture}[scale=1,solid,black,
			declare function={
				c(\x)= 0 + (\x <= 1)*1 + and(\x > 1, \x<=2)*(2-\x);	
			}]

			\begin{axis}[xmin=0,xmax=3.5,ymax=1.5, ymin=0, samples=500,width=4.5cm,height=3cm,
				axis x line*=bottom, axis y line*=left, axis lines=middle, xtick={1,2}, ytick={1}]
				\addplot[blue, ultra thick,domain=0:5] {c(x)} node[above,pos=.5]{$c_{e_1}$};
			\end{axis}
			
		\end{tikzpicture}
	};
	
	\node[left of=s1,blue, node distance=1.5cm](){$r_1=\CharF[{[0,1]}]$};
\end{tikzpicture}
	\caption{A single-commodity instance without a \sCDEu. The two edges have a fixed (flow independent) travel time of $\tau_e$. The edge $e_1$ has an edge capacity limiting the total flow volume on the edge.}\label{fig:NonExistenceForNonConstantCapacity}
\end{figure}
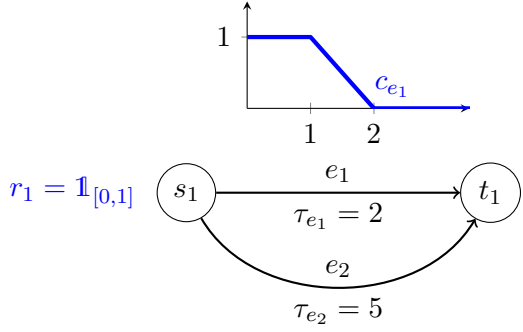

\begin{example}\label{ex:NonExistenceForNonConstantCapacity}
	In the network with a decreasing capacity function depicted in \Cref{fig:NonExistenceForNonConstantCapacity} the only feasible flow is to send all flow over the edge $e_2$ as sending any flow into edge $e_1$ will lead to a positive flow volume on this edge during the interval $[2,3]$ (which the edge capacity does not allow for). However, switching from edge $e_2$ to edge $e_1$ is an \addmDev{} according to both \eqref{eq:FeasibleDeviationsAlwaysAdditionalSpaceEnter} and \eqref{eq:UnsaturatedPathsFeasibleAfterDeviationEnter} during $[0,1]$ as there is still capacity left on edge $e_1$ at the time these particles would enter edge $e_1$. Thus, this network has no \sCDEdf{} or \sCDEu. At the same time, deviating from edge $e_2$ to edge $e_1$ is not an \addmDev{} according to \eqref{eq:FeasibleDeviationsAdditionalSpaceExceptCommonPrefixTravelling} or \eqref{eq:UnsaturatedPathsFeasibleAfterDeviationTravelling} and, thus, sending all flow via edge $e_2$ is both a \wCDEdf{} and a \wCDEuP.
\end{example}

As this example already suggests, for the existence of \wCDEuP{}, we may remove the assumption that $c_e$ is non-decreasing and allow more general capacity functions. On the flip side, we have to impose stronger assumptions on the underlying network and the flow dynamics then. We start again with the assumptions on the network:

\begin{enumerate}[label=(A\arabic*),resume=Assumptions]
	\item For any edge $e$ and walk inflow $h$, we have $f_e(h,\theta) \leq c_e(\theta)$ for all $\theta \in [0,D_e(h,0)]$.\label[asmpt]{ass:FeasibleAtStartInterval}
	\item For the case with DTC: For every $h\in \Lambda(Q)$ and $i\in I$, there exists some $p\in \Pc_i$ and $J \subseteq [t_0,t_f]$ of positive measure such that for all $t \in J$, we have $h_p(t) < B_p$ and $f_e(h, \theta)\leq c_e(\theta)$ for all $e = (v_j,v_{j+1})\in p$ and $\theta \in [\tau^j_p(h,t),\tau^{j+1}_p(h,t)]$.\label[asmpt]{ass:ExistenceOfUnsaturatedPathInterval}
	
	For the case without DTC: For every $h\in \Lambda(r)$, $t \in \planningInterval$ and $i\in I$, there exists some $p\in \Pc_i$ such that we have $f_e(h, \theta)\leq c_e(\theta)$ for all $e = (v_j,v_{j+1})\in p$ and $\theta \in [\tau^j_p(h,t),\tau^{j+1}_p(h,t)]$.
	\item For any edge $e$ the capacity function $c_e$ is continuous on $\IR_{\geq 0}$.\label[asmpt]{ass:CapacityContinuous} 
\end{enumerate}
Similar to \ref{ass:FeasibleAtStart}, \cref{ass:FeasibleAtStartInterval} ensures that feasibility is always guaranteed at the start. If no particle can reach a capacitated edge ``right away'' (\ie before time $\max_e D_e(h,0)$), this is again equivalent to having no capacity-violations on empty edges. Next, \cref{ass:ExistenceOfUnsaturatedPathInterval} ensures that there is always an outside option (similar to \ref{ass:ExistenceOfUnsaturatedPath}). Finally, the continuity of the capacity function (\ie \ref{ass:FeasibleAtStartInterval}) is a new type of assumption as, previously, we only required it to be non-negative (and non-decreasing).

Next, we require that the walk travel times are induced by edge-delay functions (as is the case, e.g., in the Vickrey point queue and the linear edge delay model):
\begin{enumerate}[label=(A\arabic*),resume=Assumptions]
	\item The walk travel times are induced by edge delay functions $D_e(h,t)$, \ie for any walk $p$ and edge $e = (v_j,v_{j+1}) \in p$, we have $\tau^{j+1}_p(h,t) = \tau^j_p(h,t) + D_e(h, \tau^j_p(h,t))$.\label[asmpt]{ass:tauInducedByEdgeDelays}
	\item The edge delays are non-negative, \ie $D_e(h,\theta) \geq 0$ for all $h,\theta$ and $e$.\label[asmpt]{ass:EdgeDelaysNonNegative}
	\item For any edge~$e$ and walk inflow~$h$, the edge delay $D_e(h,\emptyarg)$ is continuous.\label[asmpt]{ass:EdgeDelaysContinuous}
	\item For any edge $e$ and any weakly convergent sequence $h^n$, the sequence $D_e(h^n,\emptyarg)$ converges uniformly.\label[asmpt]{ass:EdgeDelaysConvergeUniform}
\end{enumerate}
\Cref{ass:tauInducedByEdgeDelays} essentially imposes a certain edge-separability of the walk travel times (and, therefore, strengthens~\ref{ass:tEqualOnCommonPrefix}) while \labelcref{ass:EdgeDelaysContinuous,ass:EdgeDelaysConvergeUniform} are typical continuity assumptions on the edge travel times. 
Finally, we give the adjusted assumptions on the flow dynamics:

\begin{enumerate}[label=(A\arabic*),resume=Assumptions]
	\item For any edge $e$ and walk inflow $h$, the function $f_e(h,\emptyarg)$ is uniformly continuous.\label[asmpt]{ass:fUniformlyContinuous}
	\item The edge load function satisfies the \emph{FIFO-compatible principle of causation}, \ie for any two time $a < b$ with $f_e(h,b) > 0$ and $b > a + D_e(h,a)$, there exists some walk $p \in \Pc$ with $e=(v_j,v_{j+1}) \in p$ and some set $J^{-1} \subseteq \planningInterval$ of positive measure such that $h_p(t) > 0$ and $\tau^j_p(h,t) \in [a,b]$ for all $t \in J^{-1}$.\label[asmpt]{ass:FIFOcompatiblePrincipleOfCausation}
\end{enumerate}
Here, \cref{ass:fUniformlyContinuous} is a strengthening of~\Cref{ass:fContinuous}. The FIFO-compatible principle of causation~\ref{ass:FIFOcompatiblePrincipleOfCausation} states that, whenever the edge-load function is positive, there must be particles that have entered but not left the edge before that time (assuming that the edge dynamics satisfy the FIFO principle).

\begin{theorem}\label{thm:ExistenceFDAddSpaceExCPInterval}
	Under \cref{ass:PsiBoundedFlowIndependent,ass:FinitelyManyWalks,ass:PathInflowBounds,ass:PsiConvergesUniform,ass:EffectivePathDelayContinuous,ass:tauInducedByEdgeDelays,ass:EdgeDelaysConvergeUniform,ass:FIFOcompatiblePrincipleOfCausation,ass:fUniformlyContinuous,ass:fConvergesUniform,ass:CapacityNN,ass:CapacityContinuous,ass:FeasibleAtStartInterval,ass:EdgeDelaysNonNegative,ass:EdgeDelaysContinuous,ass:ExistenceOfUnsaturatedPathInterval}, there always exist \wCDEu{} and \wCDEuP{} (both with and without DTC).
\end{theorem}

To prove this \namecref{thm:ExistenceFDAddSpaceExCPInterval} we replace the given edge-load function by a new one that incorporates the capacity function and has the property that it is larger than~$0$ at time~$\theta$ for an edge~$e$ if and only if there will be a violation of the (original) capacity function during the time interval particles entering edge~$e$ at time~$\theta$ spend travelling on this edge. Hence, we can replace all capacity functions by the constant zero-function and are back in the case of the previous \namecref{thm:ExistenceFDAddSpaceExCP}. The remaining, technical aspects of the proof are a), showing that strong equilibria \wrt the adjusted edge-loads and capacities are weak equilibria \wrt the original ones and vice versa, and b), showing that the assumptions of this \namecref{thm:ExistenceFDAddSpaceExCPInterval} imply the assumptions of \Cref{thm:ExistenceFDAddSpaceExCP} for the adjusted functions.

\begin{proof}
	By \Cref{prop:RelationshipsOfCDE} we only need to show the existence of a \wCDEuP. To do that we define a new edge-load function 
		\[\tilde{f}_e(h,\theta) \coloneqq \max\Set{f_e(h,\zeta)-c_e(\zeta) | \zeta \in [\theta,\theta+D_e(h,\theta)]}\]
	and new constant edge-capacity functions $\tilde{c}_e(\theta) \coloneqq 0$ for all $\theta \in \IR_{\geq 0}$.
	
	\begin{claim}\label{claim:EquivaleneOfAddmEpsSetsWithAndWithoutInterval}
		An \addmEpsDev{} with respect to $\tilde{f}_e$, $\tilde{c}_e$ and \eqref{eq:FeasibleDeviationsAdditionalSpaceExceptCommonPrefixEnter} is an \addmEpsDev{} with respect to $f_e$, $c_e$ and \eqref{eq:FeasibleDeviationsAdditionalSpaceExceptCommonPrefixTravelling} and vice versa.
	\end{claim}
	
	\begin{proofClaim}
		First, let $(q,\shiftN,\Delta)$ be an \addmEpsDev{} from $p$ with respect to $\tilde{f}_e$, $\tilde{c}_e$ and \eqref{eq:FeasibleDeviationsAdditionalSpaceExceptCommonPrefixEnter} and $\tilde q$ the suffix of~$q$ starting after the maximal common prefix of$~p$ and~$q$. Then, for any $t \in \supp(\shiftN)+\Delta$ and $e = (v_j,v_{j+1}) \in \tilde{q}$, we have 
		\begin{align*}
			0 &= \tilde{c}_e(\tau^j_q(h,t)) > \tilde{f}_e(h,\tau^j_q(h,t)) \\
			&= \max\Set{f_e(h,\theta)-c_e(\theta) | \theta \in [\tau^j_q(h,t),\tau^j_q(h,t)+D_e(h,\tau^j_q(h,t))]} \\
			\overset{\text{\ref{ass:tauInducedByEdgeDelays}}}&{=} \max\Set{f_e(h,\theta)-c_e(\theta) | \theta \in [\tau^j_q(h,t),\tau^{j+1}_q(h,t)]} 
		\end{align*}
		and, therefore $f_e(h,\theta) < c_e(\theta)$ for all $\theta \in [\tau^j_q(h,t),\tau^{j+1}_q(h,t)]$ which implies that $(q,\shiftN,\Delta)$ is an \addmEpsDev{} from $p$ with respect to $f_e$, $c_e$ and \eqref{eq:FeasibleDeviationsAdditionalSpaceExceptCommonPrefixTravelling}.
		
		Now, if $(q,\shiftN,\Delta)$ is an \addmEpsDev{} from $p$ with respect to $f_e$, $c_e$ and \eqref{eq:FeasibleDeviationsAdditionalSpaceExceptCommonPrefixTravelling} and $\tilde q$ as before, then, we have $f_e(h,\theta) > c_e(\theta)$ for all $t \in \supp(\shiftN)+\Delta$, $\theta \in [\tau^j_q(h,t),\tau^{j+1}_q(h,t)]$ and $e = (v_j,v_{j+q}) \in \tilde{q}$ implying -- by the same calculation as before -- that $\tilde{c}_e(\tau^j_q(h,t)) = 0 > \tilde{f}_e(h,\tau^j_q(h,t))$ for all such $t$ and $e$. Thus, $(q,J,\varepsilon,\Delta)$ is an \addmEpsDev{} from $p$ with respect to $\tilde{f}_e$, $\tilde{c}_e$ and~\eqref{eq:FeasibleDeviationsAdditionalSpaceExceptCommonPrefixEnter}.
	\end{proofClaim}
	
	Now, using assumption~\ref{ass:EdgeDelaysNonNegative} and the same calculations as in the proof of the previous claim, we observe that the constraint set~$S$ defined using $\tilde{f}$ and $\tilde{c}$ is the same as the one defined by $f$ and $c$. Thus, \Cref{claim:EquivaleneOfAddmEpsSetsWithAndWithoutInterval} implies that any \sCDEuP{} for $\tilde{f}$ and $\tilde{c}$ is also a \wCDEuP{} for $f$ and $c$. Consequently, to show existence of such an equilibrium, it suffices to show that we can apply \Cref{thm:ExistenceFDAddSpaceExCP} to the instance with $\tilde{f}$ and $\tilde{c}$. 

	\begin{claim}
		The new edge-load function $\tilde{f}$ satisfies \cref{ass:StrongPrincipleOfCausation,ass:fContinuous,ass:fConvergesUniform,ass:ExistenceOfUnsaturatedPath}.
	\end{claim}

	\begin{proofClaim}
		\Cref{ass:CapacityContinuous,ass:fUniformlyContinuous,ass:EdgeDelaysContinuous} together with the definition of $\tilde{f}$ directly imply that the new edge-load function is again continuous, \ie satisfies \ref{ass:fContinuous}.
		
		To see that \ref{ass:fConvergesUniform} holds for $\tilde{f}$, take any sequence of walk inflows $h^n$ converging weakly to some walk inflow $h^*$. Then for any edge $e$ we have
			\begin{align*}
				&\norm{\tilde{f}_e(h^n,\emptyarg)-\tilde{f}_e(h^*,\emptyarg)}_\infty = \max_{\theta \in \IR_{\geq 0}}\abs{\tilde{f}_e(h^n,\theta)-\tilde{f}_e(h^*,\theta)} \\
				&\quad= \max_{\theta \in \IR_{\geq 0}}\abs{\max_{\vartheta \in [t,t+D_e(h^n,t)]}\left(f_e(h^n,\vartheta)-c_e(\vartheta)\right)-\max_{\vartheta \in [t,t+D_e(h^*,t)]}\left(f_e(h^*,\vartheta)-c_e(\vartheta)\right)} \\
				&\quad\leq \max_{\theta \in \IR_{\geq 0}}\abs{\max_{\vartheta \in [t,t+D_e(h^n,t)]}\left(f_e(h^n,\vartheta)-c_e(\vartheta)\right)-\max_{\vartheta \in [t,t+D_e(h^n,t)]}\left(f_e(h^*,\vartheta)-c_e(\vartheta)\right)} \\
					&\quad\quad\quad+ \max_{\theta \in \IR_{\geq 0}}\abs{\max_{\vartheta \in [t,t+D_e(h^n,t)]}\left(f_e(h^*,\vartheta)-c_e(\vartheta)\right)-\max_{\vartheta \in [t,t+D_e(h^*,t)]}\left(f_e(h^*,\vartheta)-c_e(\vartheta)\right)}\\
				&\quad\leq \max_{\theta \in \IR_{\geq 0}}\max_{\vartheta \in [t,t+D_e(h^n,t)]}\abs{f_e(h^n,\vartheta) - f_e(h^*,\vartheta)} \\
					&\quad\quad\quad+ \hspace{-2em}\max_{\substack{\vartheta,\vartheta' \in \IR_{\geq 0}:\\\abs{\vartheta-\vartheta'} \leq \norm{D_e(h^n,\emptyarg)-D_e(h^*,\emptyarg)}_\infty}}\hspace{-2em}\Set{\abs{(f_e(h^*,\vartheta)-c_e(\vartheta))-(f_e(h^*,\vartheta')-c_e(\vartheta'))}}\\	
				&\quad\leq \norm{f_e(h^n,\emptyarg)-f_e(h^*,\emptyarg)}_\infty \\
					&\quad\quad\quad+ \hspace{-2em}\max_{\substack{\vartheta,\vartheta' \in \IR_{\geq 0}:\\\abs{\vartheta-\vartheta'} \leq \norm{D_e(h^n,\emptyarg)-D_e(h^*,\emptyarg)}_\infty}}\hspace{-2em}\Set{\abs{(f_e(h^*,\vartheta)-c_e(\vartheta))-(f_e(h^*,\vartheta')-c_e(\vartheta'))}}.
			\end{align*}
		Now the two terms at the end go to zero as $n$ goes to infinity by \cref{ass:fUniformlyContinuous,ass:fConvergesUniform,ass:EdgeDelaysConvergeUniform,ass:CapacityContinuous} for $f_e$ and $c_e$. Thus, \ref{ass:fConvergesUniform} holds for $\tilde{f}_e$ as well.
		
		Assumption~\ref{ass:ExistenceOfUnsaturatedPathInterval} for $f$ implies that \ref{ass:ExistenceOfUnsaturatedPath} holds for $\tilde{f}$.
		
		Finally, to show that $\tilde{f}$ satisfies the principle of causation (\ie \ref{ass:StrongPrincipleOfCausation}), take any walk inflow $h$, edge $e$ and two times $a < b$ with $0 \leq \tilde{f}_e(h,a) < \tilde{f}_e(h,b)$. Then, the definition of $\tilde{f}$ implies that there exists some time $\theta \in [b,b+D_e(h,b)]$ with
			\[ f_e(h,\theta) \overset{\text{\ref{ass:CapacityNN}}}{\geq} f_e(h,\theta) - c_e(\theta) > \tilde{f}_e(h,a) \geq 0.\] 
		The definition of $\tilde{f}_e(h,a)$ then implies $\theta \notin [a,a + D_e(h,a)]$ and, thus, $\theta > a+D_e(h,a)$. By~\ref{ass:FIFOcompatiblePrincipleOfCausation} this gives us some walk $p \in \Pc$ with $e=(v_j,v_{j+1}) \in p$ and some set $J^{-1} \subseteq \planningInterval$ of positive measure such that $h_p(t) > 0$ and $\tau^j_p(h,t) \in [a,b]$ for all $t \in J^{-1}$. But this is exactly what we need for \ref{ass:StrongPrincipleOfCausation} to hold for $\tilde{f}$.
	\end{proofClaim}
	Since the walk travel times $\tau^j_p$ are induced by the edge delays (\cf \labelcref{ass:tauInducedByEdgeDelays}), \cref{ass:EdgeDelaysConvergeUniform} implies that \cref{ass:tEqualOnCommonPrefix,ass:tauConvergesUniform} hold. Furthermore, the new edge-capacities $\tilde{c}_e$ are constant and non negative (\ie satisfy \labelcref{ass:CapacityNonDecreasing,ass:CapacityNN}) and assumption~\ref{ass:FeasibleAtStartInterval} on $c$ and $f$ implies that $\tilde{c}$ and $\tilde{f}$ satisfy \ref{ass:FeasibleAtStart}. Thus, we can now apply \Cref{thm:ExistenceFDAddSpaceExCP} to get the existence of a \sCDEuP{} $h^*$ for the edge-load functions $\tilde{f}$ and capacities $\tilde{c}$, which, by \Cref{claim:EquivaleneOfAddmEpsSetsWithAndWithoutInterval}, is also a \wCDEuP{} for the original edge-load functions $f$ and capacities $c$.
\end{proof}

Two important flow models for which we can now apply the above existence results are the Vickrey point queue model and the liner edge delay model described in \Cref{sec:counter}. More precisely, let the edge delays be defined as $D_e(h,\theta) \coloneqq \tau_e + \frac{q_e(h,\theta)}{\nu_e}$ or $D_e(h,\theta) \coloneqq \tau_e + \frac{x_e(h,\theta)}{\nu_e}$ and assume that all free flow travel times and service rates are strictly positive and finite:
\begin{enumerate}[label=(A\arabic*),resume=Assumptions]
	\item For all edges $e \in E$ we have $\tau_e, \nu_e \in \IR_{>0}$.\label[asmpt]{ass:TauNuPosFinite}
\end{enumerate}
Furthermore, we define the \effWalkDelay s by 
	\[\Psi_p(h,t) \coloneqq C_p(D_p(h,t), t + D_p(h,t)) \text{ for all } p \in \Pc_i,\]
where $C_p$ is a continuous function mapping travel and arrival time to total travel cost and assume that $\tau^j_p$ is derived from the edge delays (i.e., defined as in \cref{ass:tauInducedByEdgeDelays}). Finally, we use the flow volume on an edge as its edge load, \ie $f_e(h,t) \coloneqq x_e(h,t)$.

As the following \namecref{lemma:VickreyModelSatisfiesAssumptions} will show, both these models then automatically satisfy all the previous assumptions posed on the flow dynamics. Thus, we only have to check whether a given network also satisfies the additional assumptions on the network itself to be able to apply the two existence theorems of this section.

\begin{lemma}\label{lemma:VickreyModelSatisfiesAssumptions}
	Under \cref{ass:FinitelyManyWalks,ass:TauNuPosFinite} the two models described above both satisfy \cref{ass:EffectivePathDelayContinuous,ass:fContinuous,ass:fUniformlyContinuous,ass:fConvergesUniform,ass:StrongPrincipleOfCausation,ass:PsiConvergesUniform,ass:PsiBoundedFlowIndependent,ass:EdgeDelaysNonNegative,ass:EdgeDelaysContinuous,ass:EdgeDelaysConvergeUniform,ass:FIFOcompatiblePrincipleOfCausation,ass:tEqualOnCommonPrefix,ass:tauConvergesUniform,ass:tauInducedByEdgeDelays}.
\end{lemma}

\begin{proof}
	We first note, that under both the Vickrey point queue model and the linear edge delay model any fixed network with finite walk set $\Pc$ and fixed flow volume $Q$ has some constant $K > 0$ such that for any walk inflow $h \in \Lambda(Q)$ and edge $e \in E$ the support of $x_e(h,\emptyarg)$ is in $[0,K]$. This can be shown by the same proof as the one for \cite[Lemma~4.4 (full version)]{GHP22}. Together with the continuity of the functions $C_p$ this already implies that $\Psi_p$ is bounded, \ie \cref{ass:PsiBoundedFlowIndependent} holds.
	
	Then, for the Vickrey point queue model, one can show in the same way as in the proofs of \cite[Claims 3, 4, 5 (full version)]{GHP22} that the mappings $h \mapsto \int_0f^+_e(\zeta)\diff\zeta$ and $h \mapsto \int_0f^-_e(\zeta)\diff\zeta$ are sequentially weak-strong continuous from $\Lambda(Q) \subseteq L^2(\planningInterval)^\Pc$ to $C([0,K])^E$. It then immediately follows that the same is true for the mappings $h \mapsto x_e(h,\emptyarg)$, $h \mapsto q_e(h,\emptyarg)$, $h \mapsto D_e(h,\emptyarg)$, $h \mapsto D_p(h,\emptyarg)$ and $h \mapsto \Psi_p(h,\emptyarg)$. This shows that \cref{ass:EffectivePathDelayContinuous,ass:fContinuous,ass:fUniformlyContinuous,ass:fConvergesUniform,ass:PsiConvergesUniform,ass:EdgeDelaysConvergeUniform,ass:EdgeDelaysContinuous} are satisfied.
	
	For the linear edge delay model the required continuity properties have been shown by Zhu and Marcotte in~\cite{ZhuM00}. Namely, \cite[Corollary 5.1]{ZhuM00} shows that linear edge delays satisfy the strong FIFO condition and, thus, \cite[Theorem 3.3]{ZhuM00} implies that the mapping $h \mapsto x_e(h,\emptyarg)$ is sequentially weak-strong continuous from $\Lambda(Q) \subseteq L^2(\planningInterval)^\Pc$ to \mbox{$L^2([0,K])^E$}. Since both, walk inflow rates and edge outflow rates, are bounded, we can then apply \cite[Proposition 3.1]{ZhuM00} to show that the mapping is even sequentially weak-strong continuous to $C([0,K])^E$. From this we again immediately get \cref{ass:EffectivePathDelayContinuous,ass:fContinuous,ass:fUniformlyContinuous,ass:fConvergesUniform,ass:PsiConvergesUniform,ass:EdgeDelaysConvergeUniform,ass:EdgeDelaysContinuous}.
	
	\Cref{ass:EdgeDelaysNonNegative} follows directly from the definition of $D_e$ and the nonnegativity of the queue length $q_e$ or the flow volume $x_e$, respectively. \Cref{ass:tEqualOnCommonPrefix,ass:tauConvergesUniform} follow from \ref{ass:EdgeDelaysConvergeUniform} in the same way as in the proof of \Cref{thm:ExistenceFDAddSpaceExCPInterval}.
	
	It remains to show that \cref{ass:StrongPrincipleOfCausation,ass:FIFOcompatiblePrincipleOfCausation} (the principles of causation) hold as well. This can be shown in the same way for both models: For \ref{ass:StrongPrincipleOfCausation} we observe that the flow volume of an edge can only increase during some interval $[a,b]$ if there is positive inflow into this edge for some subset $J \subseteq [a,b]$ of positive measure. The definition of the flow dynamics then directly implies the existence of some walk $p$ with $e=(v_j,v_{j+1}) \in p$ and some set $J^{-1} \subseteq \planningInterval$ of positive measure such that $h_p(t) > 0$ and $\tau^j_p(h,t) \in J \subseteq [a,b]$ for all $t \in J^{-1}$.
	
	For \ref{ass:FIFOcompatiblePrincipleOfCausation} we also use flow conservation on edges~\eqref{eq:FlowConservationOnEdges}, which implies that any flow on some edge $e$ at time $a$ has left this edge by time $a+D_e(h,a)$. Thus, if there is still positive flow volume on edge $e$ at some later time $b > a+D_e(h,a)$, there must have been some additional inflow into this edge between $a$ and $b$. This then gives us the desired walk $p$ and set $J^{-1}$ in exactly the same way as before.
\end{proof}

\begin{corollary}\label{cor:ExistenceVickrey}
	Using either the Vickrey point queue model or the linear edge delay model together with volume constraints as described above
	\begin{itemize}
		\item any network satisfying \cref{ass:FinitelyManyWalks,ass:PathInflowBounds,ass:CapacityNonDecreasing,ass:CapacityNN,ass:ExistenceOfUnsaturatedPath,ass:FeasibleAtStart} has a \sCDEuP{} and
		\item any network satisfying \cref{ass:FinitelyManyWalks,ass:PathInflowBounds,ass:CapacityNN,ass:CapacityContinuous,ass:FeasibleAtStartInterval,ass:ExistenceOfUnsaturatedPathInterval} has a \wCDEuP.
	\end{itemize}
\end{corollary}

\begin{proof}
	This follows directly from \Cref{lemma:VickreyModelSatisfiesAssumptions} together with \Cref{thm:ExistenceFDAddSpaceExCP,thm:ExistenceFDAddSpaceExCPInterval}.
\end{proof}

We conclude by noting that, in contrast to unconstrained dynamic equilibria, it seems unlikely that an existence result like \Cref{cor:ExistenceVickrey} for the Vickrey point queue model with fixed network inflow rates can be shown using an extension approach instead wherein one ``constructs'' an equilibrium by iteratively extending the walk inflow over longer and longer time intervals (\cf \cite{Koch11,SeringThesis}). This is, because an important ingredient for the success of such a construction process is the observation that, in an unconstrained dynamic equilibrium, later starting particles can never influence earlier starting particles (and, thus, later extensions cannot invalidate earlier ones). For side-constrained dynamic equilibria, however, there exists a simple single-commodity instance wherein this is not true, \ie even in an equilibrium particles may still overtake each other. Moreover, this instance does indeed have a unique equilibrium which has a prefix which is not an equilibrium on its own:

\begin{figure}
	\centering
	\begin{tikzpicture}
	\node[namedVertex] (s) at (0,0) {$s$};
	\node[namedVertex] (v) at (4,0) {$w$};
	\node[namedVertex] (t)  at (8,0) {$t$};
	
	\draw[edge] (s) -- node[above]{$e_1$} node[below](e1){$\tau_{e_1}=1$} (v);
	\draw[edge] (s) to[bend left=40] node[below]{$e_2$} node[above]{$\tau_{e_2}=3$} (v);
	\draw[edge] (v) -- node[below]{$\tau_{vt}=1, \nu_{vt}=1$} (t);
	\draw[edge] (s) to[bend left=80] node[above]{$\tau_{st}=6$} (t);
	
	\node[below of=e1, anchor=north,node distance=.2cm] {
		\begin{tikzpicture}[scale=1,solid,black,
			declare function={
				c(\x)= 2;			
			}]

			\begin{axis}[xmin=0,xmax=3.5,ymax=2.5, ymin=0, samples=500,width=4cm,height=3cm,
				axis x line*=bottom, axis y line*=left, axis lines=middle, xtick={1,2,3,4}, ytick={1,2}]
				\addplot[blue, ultra thick,domain=0:5] {c(x)} node[below,pos=.4]{$c_{e_1}$};
			\end{axis}
			
		\end{tikzpicture}
	};
	
	\node[left of=s,blue,node distance=2cm](){$r=2\cdot\CharF[{[0,3]}] + \CharF[{[1,3]}]$};
\end{tikzpicture}
	\caption{A single-commodity network with fixed network inflow rate and a volume-constraint on one edge. All values of $\nu$ not explicitly given are infinity. Under the Vickrey point queue model this instance has a unique \wCDEu{} (which is also a \sCDEu) with the property that restricting it to $[0,2]$ is not an equilibrium for this instance.}\label{fig:PhaseExtensionImpossible}
\end{figure}
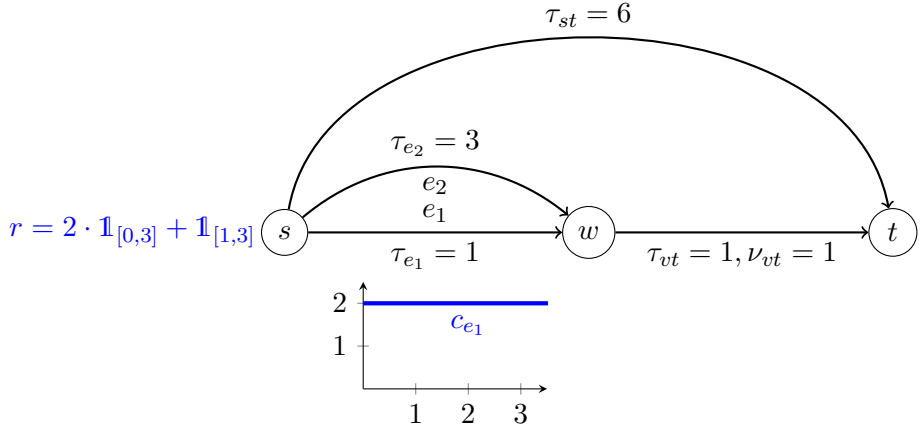

\begin{example}\label{ex:PhaseExtensionImpossible}
	\Cref{fig:PhaseExtensionImpossible} depicts a single-commodity network with fixed network inflow rate and a constant volume-constraint on one edge. Under the Vickrey point queue model this instance has a unique \wCDEu: Over the whole interval $[0,3]$ particles are send into the path $e_1,vt$ at a rate of $2$. Between time $1$ and $2$ the remaining particles are send over the direct edge towards $t$ and after time $2$ the remaining particles are sent into the path $e_2,vt$. To see why this is an equilibrium, observe that after time $1$ the volume constraint on edge $e_1$ is always tight and, thus, switching to path $e_1,vt$ is never an \addmDev. Furthermore, due to the queue on edge $vt$ the travel time over path $e_2,vt$ is strictly larger than $6$ during $(0,2)$. 
	
	However, restricting this path inflow to the interval $[0,2]$ does not yield a \wCDEu{} as then the queue on edge $vt$ will never grow beyond a length of $2$ and, consequently the travel time over path $e_2,vt$ will always be strictly shorter then over edge $st$ (except at time~$0$).
	
	Note, that in the equilibrium for the whole interval $[0,3]$ particles travelling along path $e_2,vt$ will get overtaken by later starting particles travelling along path $e_1,vt$. This is something that cannot happen in the unconstrained model as there the former particles could than always improve by copying the latter's strategy. Here, on the other hand, this is not allowed due to the capacity constraint. This effect is essentially independent of the exact equilibrium concept, \ie the above observation is equally true for all the equilibrium concepts defined in \Cref{def:TypesOfCDE}.
\end{example}


\section{Conclusion}

We provided a counterexample to a claimed existence result for dynamic equilibria with side constraints. The implications of this counterexample were shown to be severe since solutions to the canonical infinite dimensional variational inequality are in some sense useless and other approaches seem to be necessary. 
We then established a general framework for defining side-constrained dynamic equilibria based on two key objects: A \setS{} $S$ containing all feasible flows (given as walk inflows) and correspondences $A_p$ providing the flow-dependent set of \addmEpsDev s. We showed that this equilibrium concept not only encompasses the known unconstrained equilibria with and without departure time choice and capacitated dynamic equilibria with convex \setS{}s but also allows for a whole range of new dynamic equilibria inspired by static side-constrained equilibria.
We provided conditions under which they can be characterized as solutions to a quasi-variational or even a variational inequality. The latter characterization then also gave rise to a first existence result for certain side-constrained dynamic equilibria with convex \setS.
Finally, we turned to equilibria wherein the side-constraints are given by time-varying edge-load constraints. To deal with the non-convexity of the \setS{}, we employed an augmented Lagrangian approach by relaxing the hard edge-load-capacities and replacing them by penalty functions. We demonstrated that these existence results apply, in particular, for the widely used Vickrey point queue model as well as the linear edge delay model.

Several important questions remain open. First of all, it would be interesting to find an existence result for BSDE similar to \Cref{thm:ExistenceFDAddSpaceExCP} for LPDE and MNSDE. The main obstacle to obtaining such a result seems to be the fact that for BSDE, the definition of \addmEpsDev s involves the network loading which, in general, is a very complex mapping and, even for well-studied flow models, is not fully understood yet. Note that, due to \Cref{prop:RelationshipsOfCDE}, such a result would also directly imply existence of \globalEL{} as well as providing an alternative proof for the existence of LPDE. Another aspect is the multiplicity of equilibria and
the issue of selecting a particular type of equilibrium having desirable properties.
It is an interesting research direction to characterize equilibrium concepts
that admit equilibrium selection via appropriate optimization or optimal control reformulations
whose optimal solutions provide such desirable properties.

\paragraph*{Acknowledgments:} \acknowledgetext

\clearpage
\printbibliography
\clearpage

\appendix


\section{List of Symbols and Notation}

\ifarxiv\begin{longtable}{p{.39\textwidth}p{.57\textwidth}}\else\begin{longtable}{p{.41\textwidth}p{.54\textwidth}}\fi
	\textbf{Symbol}			& \textbf{Name/Description} \\\hline\endhead
	$L^2([a,b])$ 						& The set of $L^2$-integrable functions from $[a,b]$ to $\IR$ \\
	$L^2_+([a,b])$ 						& The set of $L^2$-integrable functions from $[a,b]$ to $\IR_{\geq 0}$ \\
	$\scalar{\emptyarg}{\emptyarg}$						& The scalar product. Specifically, for $f,g \in L^2([a,b])^d$ the scalar product is defined as $\scalar{f}{g} \coloneqq \sum_{j=1}^d\int_a^b f_j(\theta)g_j(\theta)\diff\theta$ \\
	$G = (V,E)$							& A directed graph with node set $V$ and edge set $E$ \\
	$\planningInterval \subseteq \IR_{\geq 0}$ 	& The \emph{planning horizon}, \ie the time interval during which particles may enter the network \\
	$I$									& The finite set of commodities \\
	$s_i, t_i$							& Source-/sink node of commodity $i \in I$ \\
	$\Pc_i$								& The set of feasible walks of commodity $i$ \\
	$r_i: \planningInterval \to \IR_{\geq 0} $ & The fixed \emph{network inflow rate} at which particles of commodity $i$ enter the network at $s_i$ \\
	$Q_i \geq 0$						& The fixed \emph{flow volume} of commodity $i$ \\
	$\Pc = \bigcup_{i \in I}\Pc_i$		& The set of feasible walks -- note that we assume that different commodities have disjoint sets of feasible walks $\Pc_i$ \\
	$B_p \geq 0$						& A given fixed upper bound on the walk inflow rate into walk $p \in \Pc$ \\
	$\Lambda(r) \subseteq L^2_+(\planningInterval)^\Pc$ 
										& The set of feasible walk inflows with fixed network inflow rates $r$ \\
	$\Lambda(Q) \subseteq L^2_+(\planningInterval)^\Pc$ 
										& The set of feasible walk inflows with fixed flow volume $Q$ \\
	$h \in \Lambda(r), h \in \Lambda(Q)$& A \emph{walk inflow} vector where $h_p: \planningInterval \to \IR_{\geq 0}$ describes the rate at which particles of commodity $i \in I$ enter walk $p \in \Pc_i$ \\
	$\mathcal{R}$						& A set containing a tuple $(p,j)$ for every walk $p \in \Pc$ and $j \in [\abs{p}]$ used to denote the $j$-th edge on $p$ \\
	$f^+ \in L^2_+([\tStart,\infty))^{\mathcal{R}}$ & \emph{Edge inflow rates}: $f^+_{e,j}(t)$ denotes the rate at which particle on walk $p$ enter its $j$-th edge at time $t$ \\
	$f^- \in L^2_+([\tStart,\infty))^{\mathcal{R}}$ & \emph{Edge outflow rates}: $f^-_{e,j}(t)$ denotes the rate at which particle on walk $p$ leave its $j$-th edge at time $t$ \\
	$f = (f^+,f^-)$						& \emph{Edge flow}: a flow described by its edge in- and outflow rates \\
	$\flowVolume[e](h,\emptyarg): \IR_{\geq 0} \to \IR_{\geq 0}$			
										& \emph{(Edge) flow volume}: the volume of flow $x_e(h,\theta) \coloneqq \int_{0}^\theta f^+_e(h,\vartheta)\diff\vartheta -  \int_{0}^\theta f^-_e(h,\vartheta)\diff\vartheta$ on edge $e$ at time $\theta$ under the flow induced by the walk inflow $h$ \\
	$q_e(h,\emptyarg): \IR_{\geq 0} \to \IR_{\geq 0}$						
										& \emph{Queue length}: The length $q_e(h,\theta) \coloneqq \int_{0}^\theta f^+_e(h,\emptyarg)\diff\vartheta -  \int_{0}^{\theta+\tau_e} f^-_e(h,\vartheta)\diff\vartheta$ of the queue on edge $e$ at time $\theta$ under the flow induced by the walk inflow $h$ \\
	$D_e(h,\emptyarg): \IR_{\geq 0} \to \IR_{\geq 0}$
										& \emph{Edge delay}: the delay $D_e(h,\theta)$ experience under the flow induced by $h$ by particles entering edge $e$ at time $\theta$ \\
	$D_p(h,\emptyarg): \planningInterval \to \IR_{\geq 0}$
										& \emph{Walk delay}: the travel time $D_p(h,t)$ experienced under the flow induced by $h$ by particles entering walk $p$ at time $t$ \\
	$\Psi:  L^2_+([t_0,t_f])^\Pc \to  \hat{M}([t_0,t_f])^\Pc$ 
										& \emph{\effWalkDelay}: the effective walk delay $\Psi_p(h,t)$ experienced under the flow induced by $h$ by particles entering walk $p$ at time $t$ (comprising \eg travel time, early/late arrival penalties, energy costs, \dots) \\
	$\truncated{\Psi}{M}: L^2_+([t_0,t_f])^\Pc \to  L^2([t_0,t_f])^\Pc$
										& \emph{truncated \effWalkDelay}: the \effWalkDelay{} operator $\Psi$ capped at $M \in \IR$ (see \Cref{def:truncatedPsi})\\
	$\tau_p^j(h,\emptyarg): \planningInterval \to \IR_{\geq 0}$						
										& The \emph{arrival time} $\tau^j_p(h,t)$ of particles starting along walk $p$ at time $t$ at the $j$-th node of this walk under the flow induced by $h$ \\
	$c_e: \IR_{\geq 0} \to \IR_{\geq 0}$ & \emph{Edge capacity}: a function denoting the capacity $c_e(\theta)$ of edge $e$ at time $\theta$ \\
	$f_e(h,\emptyarg): \IR_{\geq 0} \to \IR$	&  \emph{Edge load}: $f_e(h,\theta)$ denotes some measure of the flow induced by $h$ on edge $e$ at time $\theta$ \\
	$S \subseteq L^2_+(\planningInterval)^\Pc$ & \emph{\setS}: the set of all feasible walk inflows \\
	$A_p(h) \subseteq \Pc_i \times L^2_+(\planningInterval) \times \IR$
										& The set of \addmEpsDev s $(q,\shiftN,\Delta)$ under  $h$ from walk $p$ \\
	$(q,\shiftN,\Delta) \in A_p(h)$ & \emph{\addmEpsDev}: a tuple denoting an \addmEpsDev{} of particles in space from walk $p$ to walk $q$ at a rate of $\shiftN$ and with a time shift of $\Delta$ \\
	$H_{p\to q}(h,\shiftN,\Delta) \in L^2_+(\planningInterval)^\Pc$ 
										& The walk inflow obtained from $h$ by an \addmEpsDev{} $(q,\shiftN,\Delta)$ \\
	$M_i(h) \subseteq L^2_+(\planningInterval)^\Pc$ 
										& The set of walk inflows which can be obtained by \addmEpsDev s of commodity $i$ from $h$ \\
	$U_p(h,t) \subseteq \Pc_i \times \IR$ & The set of \addmDev s $(q,\Delta)$ for particles entering walk $p$ at time $t$ under $h$ \\
	$(q,\Delta) \in U_p(h,t)$			& \emph{\addmDev}: a tuple denoting an \addmDev{} of shifting in space from $p$ to $q$ and in time from $t$ to $t + \Delta$ \\ 
\end{longtable}

\section{List of Dynamic Equilibrium Concepts}

In the main paper we consider the following dynamic equilibrium concepts:
\begin{itemize}
	\item Dynamic equilibrium with fixed inflow rates: A walk inflow is an equilibrium if almost no particle can improve by switching to a different walk -- see \Cref{def:DE}.
	\item Dynamic equilibrium with fixed flow volume and departure choice: A walk inflow is an equilibrium if almost no particle can improve by switching to a different walk and/or departure time -- see \Cref{def:DE}.
	\item Dynamic equilibrium with elastic demands and departure choice: A walk inflow is an equilibrium if almost no particle can improve by switching to a different walk and/or departure time or by staying at home -- see \Cref{def:DE}.
	\item \SCDE[full] (\SCDE): Our general equilibrium concept: A walk inflow is an equilibrium if no particle has an \addmDev{} with strictly better \effWalkDelay -- see \Cref{def:DCE}.
	\item \globalSCDE[full] (\globalSCDE): Given any \setS{} $S \subseteq \Lambda(Q)$, the \SCDE{} where \addmEpsDev s are those \epsDev s that lead to another flow in $S$ -- see \Cref{def:strongCDE}.
	\item \globalEL[full] (\globalEL): The same as \globalSCDE{} but specifically for $S$ defined by edge-load constraints (\ie by \eqref{eq:FeasibilitySetforEdgeLoad}) -- see \ifarxiv\Cref{def:TypesOfCDE}\else\Cref{sec:SCviaNL}\fi.
	\item \sCDEdf[full] (\sCDEdf): $S$ defined by edge load constraints, \addmEpsDev s are those where the resulting flow is feasible for all deviating particles at times where such particles enter an edge -- see \ifarxiv\Cref{def:TypesOfCDE}\else\Cref{sec:EquilibriaLPMNSBS}\fi.
	\item \wCDEdf[full] (\sCDEdf): $S$ defined by edge load constraints, \addmEpsDev s are those where the resulting flow is feasible for all deviating particles at all time where such particles travel on an edge -- see \ifarxiv\Cref{def:TypesOfCDE}\else\Cref{sec:EquilibriaLPMNSBS}\fi.
	\item \sCDEu[full] (\sCDEu): $S$ defined by edge load constraints, \addmEpsDev s are those where deviating particles only enter unsaturated edges -- see \ifarxiv\Cref{def:TypesOfCDE}\else\Cref{sec:EquilibriaLPMNSBS}\fi.
	\item \wCDEu[full] (\wCDEu): $S$ defined by edge load constraints, \addmEpsDev s are those where deviating particles only travel on unsaturated edges -- see \ifarxiv\Cref{def:TypesOfCDE}\else\Cref{sec:EquilibriaLPMNSBS}\fi.
	\item \sCDEuP[full] (\sCDEuP): $S$ defined by edge load constraints, \addmEpsDev s are defined as for \sCDEu{} except that edge-capacities may be tight on a common prefix of the current and the alternative walk if $\Delta=0$ -- see \ifarxiv\Cref{def:TypesOfCDE}\else\Cref{sec:EquilibriaLPMNSBS}\fi.
	\item \wCDEuP[full] (\wCDEuP): $S$ defined by edge load constraints, \addmEpsDev s are defined as for \wCDEu{} except that edge-capacities may be tight on a common prefix of the current and the alternative walk if $\Delta=0$ -- see \ifarxiv\Cref{def:TypesOfCDE}\else\Cref{sec:EquilibriaLPMNSBS}\fi.
\end{itemize}


\section{An Example for the Usefulness of Cycles}

The following example shows why particles may prefer to travel along walks containing cycles in networks with capacity constraints. This means that the equilibria in such a network can be different depending on whether the strategy space of the particles only includes paths or also walks containing cycles.

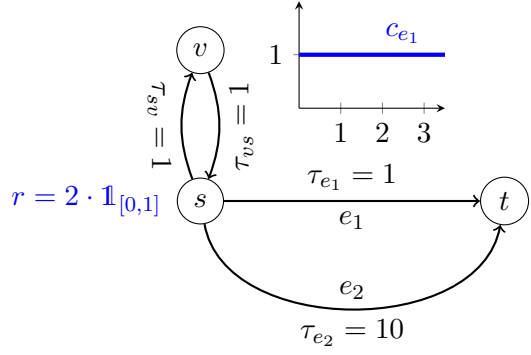
\begin{figure}
	\centering
	\begin{tikzpicture}
	\node[namedVertex] (s) at (0,0) {$s$};
	\node[namedVertex] (t) at (4,0) {$t$};
	\node[namedVertex] (v)  at (0,2) {$v$};
	
	\draw[edge] (s) -- node[below]{$e_1$} node[above](e){$\tau_{e_1}=1$} (t);
	\draw[edge] (s) to[bend right=80] node[above]{$e_2$} node[below]{$\tau_{e_2}=10$} (t);
	\draw[edge] (s) to[bend left=20] node[below, sloped]{$\tau_{sv}=1$} (v);
	\draw[edge] (v) to[bend left=20] node[below, sloped]{$\tau_{vs}=1$} (s);
	
	\node[above of=e, anchor=south,node distance=.2cm] {
		\begin{tikzpicture}[scale=1,solid,black,
			declare function={
				c(\x)= 1;			
			}]

			\begin{axis}[xmin=0,xmax=3.5,ymax=2, ymin=0, samples=500,width=3.5cm,height=3cm,
				axis x line*=bottom, axis y line*=left, axis lines=middle, xtick={1,2,3,4}, ytick={1}]
				\addplot[blue, ultra thick,domain=0:5] {c(x)} node[above,pos=.5]{$c_{e_1}$};
			\end{axis}
			
		\end{tikzpicture}
	};
	
	\node[left of=s,blue,node distance=1.5cm](){$r=2\cdot\CharF[{[0,1]}]$};
\end{tikzpicture}
	\caption{A single-commodity network with fixed network inflow rate where allowing particles to travel along cycles changes (improves) the equilibrium flow. All edges have a flow independent travel time as given in the figure. Edge~$e_1$ is the only edge with a hard volume constraint.}\label{fig:ImprovementByCycles}
\end{figure}

\begin{example}\label{ex:ImprovementByCycles}
	The network in \Cref{fig:ImprovementByCycles} is an example for a single commodity network with volume-constraints on the edges where it makes a difference for the resulting equilibrium whether travelling along cycles is allowed or not. If cycles are not allowed, only half of the flow can use the short edge $e_1$ towards the sink while the rest of the flow has to take the much longer edge $e_2$. If, on the other hand, cycles are allowed, particles can use the cycle $s\to v \to s$ to essentially wait at the source node until there is again room on edge $e_1$. In other word, in this example individual particles prefer to travel along a cycle.
\end{example}


\ifarxiv
	\section{Existence of Unconstrained Dynamic Equilibria}\label{app:VI-Existence}
\else
	\subsection{Existence of Unconstrained Dynamic Equilibria}
\fi

In \Cref{sec:dynamic} we restated\ifarxiv{} in \Cref{thm:VICharOfDE,thm:ExistenceUnconstrained}\fi{} a well known characterization and existence result for unconstrained dynamic equilibria in the notation and under the assumptions used throughout this paper. While none of the analogous results from literature known to us (\eg \ifarxiv\cite{CominettiCL15,Friesz93,Han2013,ZhuM00}\else\citealt{CominettiCL15,Friesz93,Han2013,ZhuM00}\fi) exactly match the model used in this paper, the respective theorems can still be proven in essentially the same way. For completeness we provide the adjusted proofs here:

\ifarxiv\else
\begin{repeattheorem}[\Cref{thm:VICharOfDE}]
	Assume that \ref{ass:PsiBounded} holds. Then, a walk inflow $h^* \in \Lambda(r)$ ($h^* \in \Lambda(Q)$) is a dynamic equilibrium with fixed inflow rates (with fixed flow volume) if and only if $h^*$ is a solution to \eqref{eqn:vi-fixed-inflow-uc} (to \eqref{eqn:vi-fixed-volume-uc}).
\end{repeattheorem}
\fi

\begin{proofNormal}[Proof of \Cref{thm:VICharOfDE}]
	First, consider the case of fixed inflow rates and let $h^*$ be a solution to the variational inequality \eqref{eqn:vi-fixed-inflow-uc}. Let $p,q \in \Pc_i$ be any two walks of some commodity $i \in I$ and $J \coloneqq \Set{t \in \planningInterval | h^*_p(t) > 0, \Psi_p(t) > \Psi_q(t)}$ the set of all times with strictly positive inflow of commodity $i$ into walk $p$ where $q$ would be a better alternative. We then define $h$ as the flow obtained by shifting all inflow of commodity~$i$ during $J$ from $p$ to $q$, i.e.
	\begin{align*}
		h_p(t) &= 0 &\text{ for all } t \in J \\
		h_q(t) &= h^*_q(t) + h^*_p(t) &\text{ for all } t \in J \\
		h_{p'}(t) &=h^*_{p'}(t) &\text{ in all other cases.}
	\end{align*}
	We then clearly have $h \in \Lambda(r)$ and, as $h^*$ is a solution to \eqref{eqn:vi-fixed-inflow-uc}, we get
	\begin{align*}
		0 	&\leq \scalar{\Psi(h^*)}{h-h^*} = \int_J \Psi_p(t)\left(0-h^*_p(t)\right)\diff t + \int_J \Psi_q(t)\left(h^*_q(t) + h^*_p(t)-h^*_q(t)\right)\diff t \\
		&=\int_J \left(\Psi_q(t)-\Psi_p(t)\right)h^*_p(t)\diff t.
	\end{align*}
	As $h^*_p$ is strictly positive and $\Psi_q(h,\emptyarg)-\Psi_p(h,\emptyarg)$ strictly negative on all of $J$, this implies that $J$ has measure zero. In other words, we have $\Psi_p(h,t) \leq \Psi_q(h,t)$ for almost all $t$ with $h^*_p(t)>0$. Thus, $h^*$ is indeed a dynamic equilibrium. 
	
	For the other direction, let $h^* \in \Lambda(r)$ be a dynamic equilibrium and $h \in \Lambda(r)$ any feasible flow. Defining
	\[\psi_i: \planningInterval \to \IR_{\geq 0}, t \mapsto \inf\Set{\Psi_p(h^*,t) | p \in \Pc_i}\]
	for every commodity $i \in I$ we get
	\[\Psi_p(h^*,t)\left(h_p(t) - h^*_p(t)\right) \geq \psi_i(t)\left(h_p(t) - h^*_p(t)\right)\]
	for almost all $t \in \planningInterval$ and every walk $p \in \Pc_i$ (by case-distinction on the second factor being negative or non-negative and using the fact that $h^*$ satisfies \eqref{eq:de-rate}). From this we directly get
	\begin{align*}
		&\scalar{\Psi(h^*)}{h-h^*} = \\
		&\quad=\sum_{i \in I}\sum_{p \in \Pc_i}\int_{t_0}^{t_f} \Psi_p(h^*,t)\left(h_p(t) - h^*_p(t)\right)\diff t \\
		&\quad\geq \sum_{i \in I}\sum_{p \in \Pc_i}\int_{t_0}^{t_f}\psi_i(t)\left(h_p(t) - h^*_p(t)\right)\diff t \\
		&\quad= \sum_{i \in I}\int_{t_0}^{t_f}\psi_i(t)\left(\sum_{p \in \Pc_i}h_p(t) - \sum_{p \in \Pc_i}h^*_p(t)\right)\diff t \\
		&\quad= \sum_{i \in I}\int_{t_0}^{t_f}\psi_i(t)\left(r_i(t) - r_i(t)\right)\diff t = 0.
	\end{align*}
	Therefore, $h^*$ is indeed a solution to the variational inequality \eqref{eqn:vi-fixed-inflow-uc}. 
	
	Now, consider the case of fixed flow volumes, let $h^*$ be a solution to the variational inequality \eqref{eqn:vi-fixed-volume-uc} and assume that $h^*$ is not a dynamic equilibrium. Then there must be a commodity $i \in I$, walks $p,q \in \Pc_i$, sets of positive measure $J_p, J_q \subseteq \planningInterval$ and constants $\varepsilon, \nu > 0$ such that
	\begin{align*}
		\forall t \in J_p: h^*_p(t) \geq \varepsilon \text{ and } \Psi_p(h^*,t) > \nu \\
		\forall t \in J_q: h^*_q(t) \leq B_p - \varepsilon \text{ and } \Psi_q(h^*,t) < \nu.
	\end{align*}
	We can also assume \wlofg that $J_p$ and $J_q$ have the same size (\wrt the Lebesgue measure on $\IR$). We now define $h$ as the flow obtained from $h^*$ by shifting flow at a rate of $\varepsilon$ in space from $p$ to $q$ and in time from $J_p$ to $J_q$, i.e.
	\begin{align*}
		h_{p}(t) &\coloneqq h^*_p(t) - \varepsilon &\text{ for all } t \in J_p \\
		h_{q}(t) &\coloneqq h^*_q(t) + \varepsilon &\text{ for all } t \in J_q \\
		h_{p'}(t) &=h^*_{p'}(t) &\text{ in all other cases.}
	\end{align*}
	Clearly, we have $h \in \Lambda(Q)$ and additionally we get
	\begin{align*}
		&\scalar{\Psi(h^*)}{h-h^*} \\
		&\quad= \int_{J_p} \Psi_{p}(h^*,t)\left(h^*_p(t)-\varepsilon-h^*_p(t)\right)\diff t + \int_{J_q} \Psi_{q}(h^*,t)\left(h^*_q(t) + \varepsilon - h^*_q(t)\right)\diff t \\
		&\quad= -\int_{J_p} \Psi_{p}(h^*,t)\varepsilon\diff t + \int_{J_q} \Psi_{q}(h^*,t) \varepsilon\diff t \\
		&\quad< - \int_{J_p} \nu\varepsilon \diff t + \int_{J_q} \nu\varepsilon \diff t = 0.
	\end{align*}
	But this is now a contradiction to $h^*$ being a solution to the variational inequality \eqref{eqn:vi-fixed-volume-uc}. Thus, $h^*$ must have been a dynamic equilibrium. 
	
	For the other direction, let $h^* \in \Lambda(Q)$ be a dynamic equilibrium with corresponding values $\nu_i \geq 0$ and $h \in \Lambda(Q)$ any feasible flow. Then we have 
	\[\Psi_p(h^*,t)\left(h_p(t) - h^*_p(t)\right) \geq \nu_i\left(h_p(t) - h^*_p(t)\right)\]
	for every time $t \in \planningInterval$ and walk $p \in \Pc_i$ (by case-distinction on the second factor being negative, positive or zero and the fact that $h^*$ satisfies \eqref{eq:de-volume}). From this we directly get
	\begin{align*}
		&\scalar{\Psi(h^*)}{h-h^*} = \\
		&\quad=\sum_{i \in I}\sum_{p \in \Pc_i}\int_{t_0}^{t_f} \Psi_p(h^*,t)\left(h_p(t) - h^*_p(t)\right)\diff t \\
		&\quad\geq \sum_{i \in I}\sum_{p \in \Pc_i}\int_{t_0}^{t_f}\nu_i\left(h_p(t) - h^*_p(t)\right)\diff t \\
		&\quad= \sum_{i \in I}\nu_i \left(\sum_{p \in \Pc_i}\int_{t_0}^{t_f}h_p(t)\diff t - \sum_{p \in \Pc_i}\int_{t_0}^{t_f}h^*_p(t)\diff t\right) \\
		&\quad= \sum_{i \in I}\nu_i\left(Q-Q\right) = 0.
	\end{align*}
	Therefore, $h^*$ is indeed a solution to the variational inequality \eqref{eqn:vi-fixed-volume-uc}. 
\end{proofNormal}

\ifarxiv\else
\begin{repeattheorem}[\Cref{thm:ExistenceUnconstrained}]
	For any network and \effWalkDelay s satisfying \labelcref{ass:FinitelyManyWalks,ass:nonEmptyPathset,ass:PsiWScont,ass:PsiBounded} there exists a dynamic equilibrium (both with and without departure time choice).
\end{repeattheorem}
\fi

\begin{proofNormal}[Proof of \Cref{thm:ExistenceUnconstrained}]
	It is easy to see that both $\Lambda(r)$ and $\Lambda(Q)$ are non-empty, convex, closed and bounded (with respect to the $L^2$-norm): 
	\begin{itemize}
		\item Non-emptyness follows from \labelcref{ass:nonEmptyPathset} and our general assumption on the choice of the walk-inflow bounds~$B_p$.
		\item Convexity follows from the fact that the constraints defining $\Lambda(r)$ and $\Lambda(Q)$ are all linear.
		\item For closedness let $h^n \in L^2_+(\planningInterval)^\Pc$ be a sequence of functions converging to some $h \in L^2_+(\planningInterval)^\Pc$ (with respect to the $L^2$-norm). If all $h^n$ are from $\Lambda(r)$ then so is $h$ as we have
		\begin{align*}
			&\int_{t_0}^{t_f}\abs{r_i(t) - \sum_{p \in \Pc_i}h_p(t)}\diff t 
			= \int_{t_0}^{t_f}\abs{\sum_{p \in \Pc_i}h^n_p(t) - \sum_{p \in \Pc_i}h_p(t)}\diff t \\
			&\quad\quad\leq \sum_{p \in \Pc_i}\int_{t_0}^{t_f}\abs{h^n_p(t) - h_p(t)}\diff t \\
			&\quad\quad\leq \sum_{p \in \Pc_i}\left(\int_{t_0}^{t_f}\left(h_p^n(t) - h_p(t)\right)^2\diff t\right)^{1/2}\cdot\left(\int_{t_0}^{t_f}1^2\diff t\right)^{1/2} \overset{n \to \infty}{\longrightarrow} 0
		\end{align*}
		and, therefore, $\sum_{p \in \Pc_i}h_p(t) = r_i(t)$ for all $i \in I$ and almost all $t \in \planningInterval$. 
		If all $h^n$ are from $\Lambda(Q)$ then so is $h$ as we have
		\begin{align*}
			&\abs{Q_i - \sum_{p \in \Pc_i}\int_{t_0}^{t_f}h_p(t)\diff t} 
				= \abs{\sum_{p \in \Pc_i}\int_{t_0}^{t_f}h^n_p(t)\diff t - \sum_{p \in \Pc_i}\int_{t_0}^{t_f}h_p(t)\diff t}\\
			&\quad\quad=\abs{\sum_{p \in \Pc_i}\int_{t_0}^{t_f} h^n_p(t) - h_p(t)\diff t} 
				\leq \sum_{p \in \Pc_i}\int_{t_0}^{t_f}\abs{h^n_p(t) - h_p(t)}\diff t \\
			&\quad\quad\leq \sum_{p \in \Pc_i}\left(\int_{t_0}^{t_f}\left(h^n_p(t) - h_p(t)\right)^2\diff t\right)^{1/2}\cdot\left(\int_{t_0}^{t_f}1^2\diff t\right)^{1/2} \overset{n \to \infty}{\longrightarrow} 0.
		\end{align*}
		Furthermore, $h$ is also bounded by $B_p$ almost everywhere as otherwise we would have some $p \in \Pc$, $\varepsilon > 0$ and some set $J \subseteq \planningInterval$ of positive measure with $h_p(t) \geq B_p + \varepsilon$ for all $t \in J$. But this would imply the following contradiction:
		\begin{align*}
			0 = \lim_n \int_{t_0}^{t_f}\left(h_p(t) - h^n_p(t)\right)^2\diff t \geq \int_J\varepsilon^2\diff t = \varepsilon^2\abs{J}.
		\end{align*}
		\item For boundedness observe that both for $\Lambda(r)$ and $\Lambda(Q)$ there exist fixed bounded $L^2$-functions bounding every walk inflow function $h_p$ of any feasible $h$ ($r_i$ and $B_p$, respectively).
	\end{itemize}
	Thus, we can choose $C = \Lambda(r) \subseteq L^2_+(\planningInterval)^\Pc$ or $C = \Lambda(Q) \subseteq L^2_+(\planningInterval)^\Pc$ and $\A = \Psi$ in \Cref{thm:Lions} to obtain a solution to \eqref{eqn:vi-fixed-inflow-uc} or \eqref{eqn:vi-fixed-volume-uc}, respectively. By \Cref{thm:VICharOfDE} those solution are then also dynamic equilibria.
\end{proofNormal}

\ifarxiv
\begin{remark}\label{rem:JustificationPathInflowBounds}
	Here we see why the walk inflow bounds $B_p$ are needed for the case with departure time choice as these ensure that $\Lambda(Q)$ is bounded. Without those bounds it is easy to construct instances without an equilibrium: Consider for example a network consisting of just a single edge $e$ and a flow independent cost function $\Psi_{1,\set{e}}(h,t) \coloneqq t$. Then for every possible flow all particles entering the network at a time different to $t=0$ can improve by shifting to a time closer to $0$. Thus, there is no equilibrium flow in this instance (and also no solution to the corresponding variational inequality).
\end{remark}
\fi

\end{document}